\documentclass[12pt]{article}
\usepackage{amsmath,amsfonts,amsthm,amssymb,mathtools,enumerate}
\usepackage{mathrsfs, graphicx,color,latexsym, tikz, calc
}
\usepackage[numbers]{natbib}
\usepackage{amsmath}
\allowdisplaybreaks

\usepackage[colorlinks=true,
linkcolor=blue,citecolor=blue,
urlcolor=blue]{hyperref}
\usetikzlibrary{shadows}
\usetikzlibrary{patterns,arrows,decorations.pathreplacing}
\usepackage{rotating}
\usepackage{booktabs}
\usepackage{pifont}
\usepackage{floatrow}
\usepackage{caption}
\usepackage{longtable}

\newtheorem{definition}{Definition}
\newtheorem{theorem}{Theorem}[section]
\newtheorem{lemma}[theorem]{Lemma}
\newtheorem{problem}{Problem}

\newtheorem{example}{Example}
\newtheorem{corollary}{Corollary}
\newtheorem{remark}{Remark}
\newtheorem{prop}{Proposition}

\newtheorem{fact}{Fact}
\newtheorem{claim}{Claim}

\title{\bf \Large Spectral bounds for vertex-weighted Laplacians of simplicial complexes}

\author{{Yueli Han, \ \ Lu Lu\footnote{Corresponding author.}\setcounter{footnote}{-1}\footnote{\emph{E-mail address:} lulumath@csu.edu.cn (L. Lu).}}\\[2mm]
	\small School of Mathematics and Statistics, Central South University,\\
	\small Changsha, Hunan, 410083, China\\
}

\date{ }
\begin{document}
	\maketitle
	\begin{abstract}
	The vertex-weighted Laplacian naturally extends the combinatorial Laplacian for simplicial complexes. Inspired by Lew’s foundational techniques for vertex-weighted Laplacians, we present a comprehensive spectral analysis of this operator. First, we determine how basic operations, including joins, complements, and Alexander duals, affect its spectrum. This yields a sharp upper bound on the spectral radius in terms of vertex weights, along with a lower bound on the multiplicity at which this bound is attained. Second, we establish a sharp lower bound for the spectral gap and characterize when the equality holds. Third, explicit lower bounds for the remaining eigenvalues are derived, linking the vertex-weighted Laplacian spectrum to that of a related weighted graph. Finally, we reveal new spectral relations between a simplicial complex and its subcomplexes. These results not only generalize numerous known theorems on combinatorial Laplacians but also provide deeper spectral insights into simplicial structures, ultimately unifying and extending a broad range of earlier work in this field.
		\end{abstract}
	
	

\tableofcontents	

	\section{Introduction}
	This paper studies the spectrum of the Laplace operator on simplicial complexes endowed with vertex weights. We begin by recalling some relevant historical developments.
	
	The study of combinatorial Laplacians on simplicial complexes originates from Kirchhoff's pioneering work \cite{K1847} on graph Laplacians. In his investigation of electrical networks, Kirchhoff established the celebrated matrix tree theorem connecting spanning trees to Laplacian determinants, and introduced the fundamental operator now known as the graph Laplacian. Concretely, for a real-valued function $f$ defined on the vertices of a graph, the Laplacian $L$ acts on $f$ according to
	\[
	(Lf)(v) = \deg(v)f(v) - \sum_{u \sim v} f(u),
	\]
	where $\deg(v)$ denotes the degree of the vertex $v$, and $u \sim v$ indicates that $u$ is adjacent to $v$. This foundational work lay dormant for nearly a century until Fiedler's breakthrough \cite{F1973}  in 1973, when he investigated the relationship between the smallest nonzero eigenvalue and the connectivity of a graph, sparking renewed interest in the spectral properties of graphs. Prior to Fiedler’s work, graph properties were primarily studied through the eigenvalues of adjacency matrices. Fiedler’s results shifted the focus toward the Laplacian operator, prompting extensive research on its spectral properties. For instance, Grone and Merris \cite{GM1990} investigated the spectrum of the graph Laplacian, emphasizing the multiplicities of integer eigenvalues and the effects of various modifications of the graph $G$ on its spectrum. In a subsequent work \cite{GM1994}, they further explored Laplacian integral graphs and established connections between degree sequences and the Laplacian spectrum via majorization. Later, Zhang \cite{Z2004} characterized all connected bipartite graphs whose third largest Laplacian eigenvalue is less than three. For further references and surveys on the spectrum of graph Laplacians, we refer the reader to \cite{HZ2005,AW2008,M1991,M1994}. 
	
In 1944, Eckmann \cite{E1944} made a pioneering contribution by extending the concept of Laplacians from graphs to simplicial complexes,
thereby defining higher order combinatorial Laplacians.  As part of this work, he formulated the combinatorial Hodge theorem, which gives a discrete analogue of the classical Hodge decomposition:
\[
\operatorname{ker}\left(d_i^* d_i + d_{i+1} d_{i+1}^*\right) = \widetilde{H}^i(X, \mathbb{R}),
\]
thus identifying the kernel of the combinatorial Laplacian
\[
L_i = d_i^*d_i + d_{i+1}d_{i+1}^*
\]
with the reduced cohomology of the complex. Building on this observation, Friedman \cite{F1996} proposed that the cohomology group $\widetilde{H}^i(X; \mathbb{R})$ can be efficiently approximated via the spectral decomposition of the combinatorial Laplacian $L_i$. In particular, for certain combinatorially significant families of complexes, notably the chessboard complexes, the spectrum of $L_i$ was observed to be entirely integral. This integrality was subsequently established rigorously by Friedman and Hanlon \cite{FP1998}, which in turn motivated further investigation of integral spectra for combinatorial Laplacians for other well-known families of simplicial complexes, including matching complexes \cite{DW2002} and matroid complexes \cite{KRS2000}.

Inspired by the normalized Laplacian on graphs, several efforts have been made to construct normalized analogues of the combinatorial Laplacian on simplicial complexes.  Chung \cite{C1993} introduced a normalized Laplacian of the form  
$\delta^{*}\delta + \rho\, \delta\delta^{*}$, where $\rho>0$ is a constant; however, no spectral upper bounds for this operator are known.  Later, Taszus \cite{T2010} proposed a different normalization, given in matrix form by  $D^{-1/2} L_iD^{-1/2}$, where $L_i$ represents the combinatorial Laplacian, and $D$ is the diagonal matrix of $L_i$.  As with Chung’s operator, the spectral upper bounds of this normalized form also remain undetermined. Lu and Peng \cite{LP2011} defined a normalized Laplacian on a hypergraph $H$ by
$L_s(H)=\Delta(G_H^s)$, where $G_H^s$ denotes the $(s-1)$-dual graph of $H$. While this construction is natural from the perspective of random walks, it has intrinsic limitations, as it does not fully reflect the higher-order relationships among hyperedges. To overcome these shortcomings, Horak and Jost \cite{HJ2013}, inspired by weighted graph Laplacians, developed a general framework for weighted Laplacians, providing a systematic definition of normalized Laplacians together with spectral upper bounds. In this framework, a weight function is assigned to each simplex of a simplicial complex, and the induced inner product on cochain groups is used to define generalized weighted Laplacians. By selecting different weight functions, one recovers, as special cases, the combinatorial Laplacian (when $\omega \equiv 1$), the normalized Laplacian, the vertex-weighted Laplacian, and several other variants of Laplacian operators. 

The combinatorial Laplacian on simplicial complexes has been extensively studied. In particular, after establishing the integrality of the combinatorial Laplacian spectra for chessboard complexes \cite{FP1998}, matroid complexes \cite{KRS2000}, and matching complexes \cite{DW2002}, Duval and Reiner \cite{DR2002} further extended these results to shifted simplicial complexes, proving that their combinatorial Laplacian spectra are integral and that the nonzero eigenvalues coincide with the conjugate partition of the degree sequence of the vertices. In addition, they established both upper and lower bounds for the spectral radius of the combinatorial Laplacian. Subsequently, Aharoni, Berger, and Meshulam \cite{ABM2005} investigated the relationship between the smallest eigenvalues of the combinatorial Laplacians of successive orders on the clique complex, and established a lower bound for the smallest eigenvalue of $L_i$. In 2020, Lew \cite{L20202} and Shukla--Yogeshwaran \cite{SY2020} independently extended the work of Aharoni et al. to general simplicial complexes, with the latter also establishing relationships between the smallest eigenvalues of the combinatorial Laplacians of a simplicial complex and those of its subcomplexes. Meanwhile, Lew \cite{L2020} generalized the lower bound of \cite{ABM2005} to arbitrary simplicial complexes. Recently, Zhan, Huang, and Lin \cite{ZHL2026} resolved the conjecture posed by Lew in \cite{L2020}, characterizing the simplicial complexes for which the smallest eigenvalue of the $L_i$ attains the aforementioned lower bound. For more advanced developments on the combinatorial Laplacian of simplicial complexes, see \cite{G2002,ZHZ2025, ZHL2025, FS2025, L20251, L20252}.

In 2024, Lew \cite{L2024} introduced a new class of Laplacian operators, known as vertex–weighted Laplacians, by imposing specific structural conditions on the weight functions of simplicial complexes, and derived a lower bound for all eigenvalues of the vertex-weighted Laplacians of independence complexes, expressed in terms of the eigenvalues of the weighted graph Laplacian. Despite the growing interest in simplicial complex Laplacians, the vertex--weighted Laplacian has been comparatively underexplored. Since it encompasses the combinatorial Laplacian as a special case, a natural question is whether the spectral properties established for combinatorial Laplacians on simplicial complexes also hold for the vertex--weighted Laplacian. 

Motivated by this question and the foundational work of Lew, in this paper we investigate several fundamental spectral properties of vertex–weighted Laplacians on simplicial complexes. Our study is significant for multiple reasons. First, it generalizes well-known results for combinatorial Laplacians to a broader, weighted  Laplacians, thereby providing a unified framework that encompasses both combinatorial and vertex–weighted operators. Second, it offers new insights into the interplay between the weight functions of simplicial complexes and their Laplacian spectra, which has implications for spectral gap estimates, eigenvalue bounds, and cohomological properties. Compared with previous studies that primarily focused on specific classes of simplicial complexes or on combinatorial Laplacians, we consider the vertex–weighted Laplacian in full generality. To investigate the eigenvalues and spectral gaps of these operators, we employ a variety of established tools and techniques—including spectral analysis of self-adjoint operators, additive compound matrices, and high-dimensional Laplacians—and implement certain modifications to some of them to better suit the vertex–weighted setting. In particular, we examine how various constructions, such as joins, complements, and Alexander duals, influence the spectra of vertex–weighted Laplacians. Building on this framework, we establish conditions under which eigenvalue bounds and spectral gap estimates, previously known for combinatorial Laplacians, remain valid in the vertex–weighted context. Moreover, we provide explicit constructions and examples of simplicial complexes that attain certain spectral bounds, and we characterize the families of complexes for which the lower bounds of the spectral gap are realized. Finally, we extend our analysis to the relationship between the spectra of a simplicial complex and those of its subcomplexes, yielding new results concerning the $i$-smallest eigenvalues and upper bounds on the dimensions of the corresponding cohomology groups.

\section{Notations and preliminaries}

\subsection{Notations, definitions and the weighted Laplacian operator}
Let $V$ be a finite set. An \textit{abstract simplicial complex} (or simply \textit{complex}) $X$ on $V$ is a collection of the subsets of $V$ closed under taking subsets; that is, if $F\in X$ and $F'\subseteq F$, then $F'\in X$.  A $\textit{subcomplex}$ of $X$ is a subset of $X$ that itself forms a complex under the same closure property. An \textit{$k$-simplex} (or \textit{$k$-face}) in $X$ is an element of $X$ with cardinality $k+1$, and its \textit{dimension} is defined as its cardinality minus one. The \textit{dimension} of $X$, denoted by $\operatorname{dim}(X)$, is the maximum dimension of all simplices of $X$. A \textit{missing face} of $X$ is a subset $\sigma \subseteq V$ such that $\sigma \notin X$ while every proper subset $\tau \subsetneq \sigma$ lies in $X$. Denote by $h(X)$ the maximal dimension among all missing faces of $X$. 

Two abstract simplicial complexes $X_1$ and $X_2$ are said to be \textit{isomorphic} if there is a bijective correspondence $f$ mapping the vertex set of $X_1$ to the vertex set of $X_2$ such that $\left\{v_0, v_1, \dots, v_k\right\} \in X_1$ if and only if $\left\{f\left(v_0\right), f\left(v_1\right), \dots, f\left(v_k\right)\right\} \in X_2$ for any $k$.

Let $X(k)$ be the set of all $k$-simplices of $X$ for $k \ge 0$, and write $f_k(X) \colon= |X(k)|$ for its cardinality.
The \textit{$p$-skeleton} $X^{(p)}$ is the subcomplex  containing all simplices of dimension at most $p$.  Specifically, the $1$-skeleton $X^{(1)}$ forms a graph with vertex set $V(X)$ and edge set $X(1)$. We refer to this graph as the \textit{underlying graph} of $X$, and denote it by $G_{X}$. A simplicial complex $X$ is called a \textit{clique complex} if it consists of all cliques of its underlying graph $G_X$. Note that $h(X)=1$ if and only if $X$ is the clique complex of its underlying graph $G_X$ (the missing faces of $X$ are the edges of the complement of $G_X$). For each $k$-simplex $\sigma$, define $\sigma(k-1) \colon= \{\tau \in X(k-1) \colon \tau \subseteq \sigma\}$,
called the \textit{boundary} of $\sigma$. Further, define $N^{+}_X(\sigma)\colon=\left\{\tau\in X(k+1)\colon\sigma\subseteq\tau\right\}$, 
called the \textit{upper neighborhood} of $\sigma$; its cardinality is referred to as the \textit{upper degree} of $\sigma$, denoted by $\deg^{+}_X(\sigma)$.
The \textit{link} of $\sigma$ is defined by
$\operatorname{lk}_X(\sigma)\colon= \{v \in V \setminus \sigma \colon \sigma \cup \{v\} \in X\}$.
For any $\sigma, \tau\in X(k)$, let $\sigma\triangle\tau\colon=\left(\sigma\setminus\tau\right)\cup\left(\tau\setminus\sigma\right)$ denote their symmetric difference, and write $\sigma\sim \tau$ if $\left|\sigma\cap \tau\right|=k$ and $\sigma\cup \tau\notin X(k+1)$.

A simplex $\sigma$ is said to be \textit{oriented} if its vertices are assigned a linear ordering, denoted by $[\sigma]$. Two orderings of the vertices of $\sigma$ define the \textit{same orientation} if they differ by an even permutation; otherwise, they induce \textit{opposite} orientations. An $\textit{orientation}$ of $X$ is defined as a choice of orientation for each simplex in $X$. In this paper, we fixed a linear ordering for $V$ as the orientation of $X$. For any $\sigma,\tau\in X(k)$ with $|\sigma \cap \tau|=k$, we define the \textit{sign} of $\sigma \cap \tau$ in $\sigma$ by
$\operatorname{sgn}(\sigma \cap \tau,\sigma)=(-1)^{j-1}$, where $j$ denotes the position (with respect to the fixed linear order on $V$) of the vertex of $\sigma$ that is removed to obtain $\sigma \cap \tau$.

The \textit{$k$-chain group} $C_k(X,\mathbb{R})$ is the real vector space formally generated by all oriented $k$-simplices $\sigma$ as a basis. We may assume that $\emptyset \in X$, called the \textit{empty simplex }with dimension $-1$. Then $(-1)$-chain group $C_{-1}(X,\mathbb{R})$ is the $1$-dimensional
$\mathbb{R}$-vector space generated by the singleton set $\{\emptyset\}$, yielding
$C_{-1}(X,\mathbb{R})\cong \mathbb{R}$. For each integer $0\le k\le \operatorname{dim} (X)$, the \textit{boundary map} $\partial_k\colon C_k(X, \mathbb{R}) \rightarrow C_{k-1}(X,\mathbb{R})$ acts on oriented $k$-simplices $\left[v_0, \dots, v_k\right]\in X$ via:
\[
\partial_k\left(\left[v_0, \dots, v_k\right]\right)=\sum_{j=0}^k(-1)^j\left[v_0, \dots, \hat{v}_j, \dots, v_k\right],
\]
where $\hat{v}_j$ indicates vertex deletion. This operator extends linearly to all chains in
$C_k(X, \mathbb{R})$. It is well known that the fundamental identity $\partial_k \circ \partial_{k+1}=0$ holds for all $0\le k<\operatorname{dim}(X)-1$. The \textit{$k$-th reduced homology group} of $X$ is given by \[\tilde{H}_k(X,\mathbb{R})\colon=\operatorname{Ker}\partial_k / \mathrm{im} \partial_{k+1}.\] We call the $k$-th reduced homology $\tilde{H}_k(X,\mathbb{R})$ \textit{vanishing} if $\tilde{H}_k(X,\mathbb{R}) = 0$.

The \textit{$k$-cochain groups $C^k(X, \mathbb{R})$} are defined as dual space of the $k$-chain groups $C_k(X,\mathbb{R})$, i.e., $C^k(X, \mathbb{R})\colon= \operatorname{Hom}(C_k(X,\mathbb{R}), \mathbb{R})$. The basis of $C^k(X, \mathbb{R})$ is given by the set of all functions $e_{\sigma}$ such that
\[e_{\sigma}\left(\tau\right)= \begin{cases}1 & \text { if }\tau=\sigma, \\ 0 & \text { otherwise,}\end{cases}\]
for all oriented $k$-simplices $\sigma$.
The functions $e_{\sigma}$ are also known as \textit{elementary $k$-cochains}.  Note that the one-dimensional vector space $C^{-1}(X, \mathbb{R})$ is generated by the identity function on the empty simplex, yielding
$C^{-1}(X,\mathbb{R})\cong C_{-1}(X,\mathbb{R})\cong \mathbb{R}$. For each integer $-1\le k\le \operatorname{dim} (X)-1$,
we define \textit{the simplicial coboundary map} $d_k : C^k(X, \mathbb{R})\rightarrow C^{k+1}(X, \mathbb{R})$ as the dual of the boundary map $\partial_{k+1}$. Specifically, for any $\phi\in C^k(X, \mathbb{R})$ and  $\sigma\in X(k+1)$, 
\[	\left(d_k (\phi)\right)\left(\sigma\right)=\sum_{j=0}^{k+1}(-1)^j \phi \left(\sigma_j\right),\]
where $\sigma_j$ denote the oriented $k$-simplex obtained from $\sigma$ by deleting its $(j+1)$-st vertex.
That is, $d_k (\phi)=\phi \circ\partial_{k+1}$. The coboundary maps $d_k$ are the connecting maps in \textit{the augmented cochain complex} of $K$ with coefficients in $\mathbb{R}$, i.e., the sequence of vector spaces and linear transformations
	\[\cdots \stackrel{\delta_{i+1}}{\longleftarrow} C^{i+1}(X, \mathbb{R}) \stackrel{\delta_i}{\longleftarrow} C^i(X, \mathbb{R}) \stackrel{\delta_{i-1}}{\longleftarrow} \cdots \longleftarrow C^{-1}(X, \mathbb{R}) \longleftarrow 0,\]
	It is well-known that $d_{k} d_{k-1} = 0$, so that the image of $d_{k-1}$ is contained in the kernel of $d_{k}$. 
	This allows us to define \textit{the $k$-th reduced cohomology group} for every $k \ge 0$ by
	\[
	\widetilde{H}^k(X;\mathbb{R})\colon= \ker d_k / \operatorname{im} d_{k-1}.
	\]
	Similarly, we say that $\widetilde{H}^k(X, \mathbb{R})$ is \textit{vanishing} if $\widetilde{H}^k(X, \mathbb{R}) = 0$. Moreover, as vector spaces, $\widetilde{H}^k(X;\mathbb{R})\cong \widetilde{H}_k(X;\mathbb{R})$.

Let $\omega \colon X \rightarrow \mathbb{R}_{>0}$ be a positive weight function on $X$, assigning a positive weight to each simplex. The positive weight function induces an inner product on each cochain group 
	$C^k(X;\mathbb{R})$ by declaring the basis elements to be orthogonal and setting
	\[
	\langle e_\sigma, e_\tau \rangle =
	\begin{cases}
		w(\sigma), & \text{if } \sigma = \tau,\\[2mm]
		0,          & \text{otherwise},
	\end{cases}
	\]
	for all $\sigma, \tau \in X(k)$. 
	The inner product on arbitrary cochains is then obtained by extending this rule bilinearly to all of $C^k(X;\mathbb{R})$.
Under the inner products $\left\langle , \right\rangle_{C^k}$ and $	\left\langle , \right\rangle_{C^{k+1}}$ on $C^k(X, \mathbb{R})$ and $C^{k+1}(X, \mathbb{R})$, respectively, the \textit{adjoint} $d_k^{\omega *}: C^{k+1}(X, \mathbb{R}) \rightarrow C^k(X, \mathbb{R})$ of the coboundary operator $d_i$ is defined by the adjoint relation
\[\left\langle d_k f_1, f_2\right\rangle_{C^{k+1}}=\left\langle f_1, d^{\omega *}_k f_2\right\rangle_{C^k}\]
for all $f_1 \in C^k(X, \mathbb{R})$ and $f_2 \in C^{k+1}(X, \mathbb{R})$. 

Given a positive weight function $\omega$, the following three fundamental operators on $C^{k}(X,\mathbb{R})$ are defined (see \cite{HJ2013}):
\begin{itemize}
	\item The \textit{weighted $k$-dimensional up-Laplacian operator} $L^{\omega \operatorname{up}}_k(X)\colon=d_k^{\omega*} d_k$;
	\item The \textit{weighted $k$-dimensional down-Laplacian operator} $L^{\omega \operatorname{down}}_k(X)\colon=d_{k-1} d_{k-1}^{\omega*}$;
	\item The \textit{weighted $k$-dimensional Laplacian operator} $L^{\omega}_k(X)=d_k^{\omega*} d_k+d_{k-1} d_{k-1}^{\omega*}$.
\end{itemize}
The \textit{$k$-th weighted spectral gap} of $X$, is the smallest eigenvalue of $L^{\omega}_k(X)$. Define the spectra $\mathbf{s}_k^{\omega \operatorname{up}}\left(X\right)$, $\mathbf{s}_k^{\omega \operatorname{down}}\left(X\right)$ and $\mathbf{s}^{\omega}_k\left(X\right)$ of operators $L_k^{\omega \operatorname{up}}\left(X\right)$, $L_k^{\omega\operatorname{down}}\left(X\right)$ and  $L^{\omega}_k\left(X\right)$ as the weakly decreasing rearrangements of their eigenvalues, respectively. Since both $L_k^{\omega \operatorname{up}}\left(X\right)$ and $L_k^{\omega\operatorname{down}}\left(X\right)$ are of the form $\phi^*\phi$, these operators are self-adjoint and positive semidefinite, as is their sum $L^{\omega}_k\left(X\right)$, which implies that all eigenvalues are real and nonnegative. The nilpotency conditions $d_kd_{k-1} = 0$ and $d^{\omega *}_{k-1}d^{\omega *}_{k} = 0$ yield the orthogonal relations
\[
L_k^{\omega \operatorname{up}}\left(X\right) \circ L_k^{\omega\operatorname{down}}\left(X\right) = 0 =L_k^{\omega\operatorname{down}}\left(X\right) \circ L_k^{\omega \operatorname{up}}\left(X\right) ,
\]
which implies that the nonzero spectrum of $L^{\omega}_k\left(X\right)$ is exactly the union (as multisets) of the nonzero spectra of $L_k^{\omega \operatorname{up}}\left(X\right)$ and $L_k^{\omega\operatorname{down}}\left(X\right)$, i.e.,
\begin{equation}\label{eq-01}
	\mathbf{s}_k^{\omega }\left(X\right) \stackrel{\circ}{=} 	\mathbf{s}_k^{\omega \operatorname{down}}\left(X\right)\stackrel{\circ}{\cup}  \mathbf{s}_{k}^{\omega \operatorname{up}}\left(X\right),
\end{equation}
where the symbol $\stackrel{\circ}{\cup}$ denotes the multiset union (i.e., the union in which multiplicities are taken into account), and symbol $\stackrel{\circ}{=}$ denotes equality of multisets up to zero eigenvalues, that is, the nonzero parts (counted with multiplicities) coincide. 
Moreover,
it is a standard result in linear algebra that for any linear operator $\phi$, the nonzero eigenvalues (including multiplicities) of $\phi^*\phi$ coincide with those of $\phi\phi^*$. Applying this to the boundary operator $d_k$, we obtain the following multiset equality up to zero eigenvalues:
\begin{equation}\label{eq-05}
	\mathbf{s}_k^{\omega \operatorname{down}}\left(X\right) \stackrel{\circ}{=} \mathbf{s}_{k-1}^{\omega \operatorname{up}}\left(X\right).
\end{equation}

Depending on the choice of the weight function, the weighted Laplacian encompasses several classical Laplacians on simplicial complexes. In particular:

\begin{itemize}
	\item If $\omega \equiv 1$, the weighted Laplacian coincides with the \textit{combinatorial Laplacian operator}, denoted by $L_k(X)$,  which has been extensively studied in \cite{F1996,DR2002,FWW2024,DW2002,ZHL2026,L2020,SY2020,KRS2000}.
	
	\item If $\omega$ is defined by
	\[
	\omega(\sigma)=
	\begin{cases}
		\displaystyle \sum_{v\in \operatorname{lk}_X(\sigma)} w(\sigma\cup\{v\}), & \text{if } \dim(\sigma)<\dim(X),\\[1mm]
		1, & \text{otherwise}
	\end{cases}
	\quad \text{for all } \sigma \in X,
	\]
	it yields the \textit{normalized combinatorial Laplacian operator}, as discussed in \cite{HJ2013,
		SWF2025}.
	
	\item If $\omega$ is defined by
	\[
	\omega(\sigma)=
	\begin{cases}
		\displaystyle \prod_{v\in \sigma}\omega(v), & \text{if } \sigma \neq \emptyset,\\[1mm]
		1, & \text{otherwise}
	\end{cases}
	\quad \text{for all } \sigma \in X,
	\]
	the resulting weighted Laplacian is called the \textit{vertex-weighted Laplacian operator}, which was introduced in \cite{L2024}. 
	The positive weight function defined in this way is called a \textit{vertex weight function}.  A simplicial complex $X$ together with a vertex weight function $\omega$ forms an ordered pair $(X, \omega)$, referred to as a \textit{vertex-weighted simplicial complex}.
\end{itemize}
Observe that the combinatorial Laplacian operator is a special case of the vertex-weighted Laplacian, obtained by setting $\omega(v)=1$ for all vertices $v$. In this paper, we mainly study the vertex-weighted Laplacian operator of simplicial complexes.

\subsection{Spectral properties of self-adjoint operators}
	For any linear operator on an $n$-dimensional vector space (or any $n\times n$ matrix) $\mathcal{A}$, let $\operatorname{m}_{\mathcal{A}}(\lambda)$ denote the multiplicity of an eigenvalue $\lambda$ of $\mathcal{A}$.  We write $\lambda_k^{\uparrow}(\mathcal{A})$ and $\lambda_k^{\downarrow}(\mathcal{A})$ for the $k$-th smallest and $k$-th largest eigenvalues of $\mathcal{A}$, respectively.
	For  $1 \le i \le n$, define
	\[
	S_i(\mathcal{A})
	\colon=
	\left\{
	\sum_{j\in J} \lambda_j^{\uparrow}(\mathcal{A})
	\colon
	J \subseteq [n],\ |J| = i
	\right\},
	\]
	the set of all sums of $i$ eigenvalues of $\mathcal{A}$. The elements of $S_i(\mathcal{A})$ can then be ordered, and we denote the $k$-th smallest and $k$-th largest elements by $S_{i,k}^{\uparrow}(\mathcal{A})$ and $S_{i,k}^{\downarrow}(\mathcal{A})$, respectively.

The following classical result of Weyl asserts that the spectrum of the sum of two symmetric matrices can be controlled by the spectra of the summands.
\begin{theorem}[{\cite[Theorem 2.8.1]{BH2012}}]\label{thm-8}
	Let A,B be real symmetric matrices of size $n \times n$. Then, for all $1 \leq i \leq n$,
	\[\lambda_i^{\uparrow}(A+B) \geq \lambda_i^{\uparrow}(A)+\lambda_1^{\uparrow}(B),\]
	or, equivalently,
	\[\lambda_i^{\downarrow}(A+B) \geq \lambda_i^{\downarrow}(A)+\lambda_n^{\downarrow}(B) .\]
\end{theorem}
The following result is known as Cauchy's interlacing theorem.
\begin{theorem}[{\cite[Corollary 2.5.2]{BH2012}}]\label{thm-7}
	Let $A$ be a real symmetric matrix of size $n \times n$ and $B$ a principal submatrix of $A$ of size $m \times m$. Then, for all $1 \le i \le m$,
	\[\lambda_{n-m+i}^{\downarrow}(A)\le \lambda_{i}^{\downarrow}(B)\le \lambda_{i}^{\downarrow}(A).\]
\end{theorem}

Let $(\mathscr{H}, \langle \cdot, \cdot \rangle)$ be an $n$-dimensional Hilbert space with an orthogonal basis $\beta=\{\beta_i : i\in [n]\}$. In what follows, we present several spectral results concerning the self-adjoint operator $\mathcal{A}$ that will be used later.
 \begin{theorem}[Min-Max Theorem, see {\cite[Theorem 2.1]{HJ2013}}]\label{thm-5}
	Let $\mathscr{H}_k$ and $\mathcal{H}_k$ denote a $k$-dimensional subspace of $\mathscr{H}$ and a family of $k$-dimensional subspaces of $\mathscr{H}$, respectively, and assume that $\mathcal{A}\colon \mathscr{H} \rightarrow \mathscr{H}$ is a compact, self-adjoint operator of Hilbert space $\mathscr{H}$. Then \[ \lambda_k^{\uparrow}\left(\mathcal{A}\right)=\min _{\mathscr{H}_k \in \mathcal{H}_k} \max _{\boldsymbol{x} \in \mathscr{H}_k} \frac{\left\langle\mathcal{A}\boldsymbol{x}, \boldsymbol{x}\right\rangle}{\left\langle \boldsymbol{x}, \boldsymbol{x}\right\rangle}=\max _{\mathscr{H}_{n-k+1} \in \mathcal{H}_{n-k+1}} \min _{\boldsymbol{x} \in \mathscr{H}_{n-k+1}} \frac{\left\langle\mathcal{A} \boldsymbol{x}, \boldsymbol{x}\right\rangle}{\left\langle \boldsymbol{x}, \boldsymbol{x}\right\rangle}.\]
The $\boldsymbol{x}$ realizing such a $\min$-$\max$ or $\max$-$\min$ then is corresponding eigenfunction, and the $\min$-$\max$ space $\mathscr{H}_k$ are spanned by the eigenfunctions for the eigenvalues $\lambda_1^{\uparrow}\left(\mathcal{A}\right), \dots, \lambda_k^{\uparrow}\left(\mathcal{A}\right)$, and analogously, the $\max$-$\min$ spaces $\mathscr{H}_{n-k+1}$ are spanned by the eigenfunctions for the eigenvalues $\lambda_k^{\uparrow}\left(\mathcal{A}\right), \dots, \lambda_n^{\uparrow}\left(\mathcal{A}\right)$.
\end{theorem}

\begin{lemma}\label{lem-12}
	Let $\mathcal{A}_1$ and $\mathcal{A}_2$ be two compact self-adjoint operators on $\mathscr{H}$. Then, for all $1 \le i \le n$,
	\[\lambda_i^{\uparrow}(\mathcal{A}_1+\mathcal{A}_2) \geq \lambda_i^{\uparrow}(\mathcal{A}_1)+\lambda_1^{\uparrow}(\mathcal{A}_2),\]
	or, equivalently,
	\[\lambda_i^{\downarrow}(\mathcal{A}_1+\mathcal{A}_2) \geq \lambda_i^{\downarrow}(\mathcal{A}_1)+\lambda_n^{\downarrow}(\mathcal{A}_2) .\]
\end{lemma}
\begin{proof}
Let $A_1$ and $A_2$ be the matrix representations of $\mathcal{A}_1$ and $\mathcal{A}_2$, respectively, with respect to a common orthonormal basis $\beta$. Since the spectrum of $\mathcal{A}_i$ coincides with that of $A_i$ for each $i$, and since both $A_1$ and $A_2$ are real symmetric matrices, the result follows directly from Theorem \ref{thm-8}.
\end{proof}

Recall that the concept of operator compression (see \cite{H1982}). Let $\mathscr{K}$ be a subspace of a Hilbert space $\mathscr{H}$, and let $\mathcal{P} \colon \mathscr{H} \to \mathscr{K}$ denote the orthogonal projection onto $\mathscr{K}$. For any operator $\mathcal{A}$ on $\mathscr{H}$, the \textit{compression} of $\mathcal{A}$ to $\mathscr{K}$ is the operator
\[
\mathcal{A}_{\mathscr{K}}\colon= \mathcal{P} \mathcal{A} \big|_{\mathscr{K}} \colon \mathscr{K} \to \mathscr{K}.
\]
Equivalently, one can write
\[
\mathcal{A}_{\mathscr{K}} \colon= \mathcal{P} \mathcal{A} \mathcal{P}^*,
\]
where $\mathcal{P}^* \colon \mathscr{K} \to \mathscr{H}$ is the adjoint of $\mathcal{P}$. The following result is an immediate corollary of Theorem \ref{thm-5}, and we include its proof for completeness.

\begin{lemma}\label{lem-11}
Let $\mathcal{A}$ be a compact self-adjoint operator on $\mathscr{H}$, and let $\mathcal{A}_{\mathscr{K}}$ denote its compression to a subspace $\mathscr{K}$ of dimension $m$. Then for all $i\in [m]$, 
\[
\lambda_{n-m+i}^{\downarrow}(\mathcal{A}) \le \lambda_i^{\downarrow}(\mathcal{A}_{\mathscr{K}}) \le \lambda_i^{\downarrow}(\mathcal{A}),
\]
or equivalently,
\[
\lambda_i^{\uparrow}(\mathcal{A}) \le \lambda_i^{\uparrow}(\mathcal{A}_{\mathscr{K}}) \le \lambda_{n-m+i}^{\uparrow}(\mathcal{A}).
\]
\end{lemma}
\begin{proof}
	Let $\mathcal{P}\colon \mathscr{H} \rightarrow \mathscr{K}$ denote the orthogonal projection. Then $\mathcal{A}_{\mathscr{K}}= \mathcal{P} \mathcal{A} \mathcal{P}^*$.
    Let $\boldsymbol{x}_1, \dots, \boldsymbol{x}_n$ be orthonormal eigenfunctions of $\mathcal{A}$ corresponding to the eigenvalues \[\lambda^{\downarrow}_1(\mathcal{A}), \dots,\lambda^{\downarrow}_n(\mathcal{A}),\] respectively. Let $\boldsymbol{y}_1, \dots, \boldsymbol{y}_m$ be orthonormal eigenfunction of $\mathcal{A}_{\mathscr{K}}$ corresponding to the eigenvalues \[\lambda^{\downarrow}_1\left(\mathcal{A}_{\mathscr{K}}\right), \dots,\lambda^{\downarrow}_m\left(\mathcal{A}_{\mathscr{K}}\right),\] respectively. For each $i \in\{1, \dots, m\}$, let $\boldsymbol{z}_i$ be a non-zero function in the subspace
	\[\left\langle\boldsymbol{y}_1, \dots, \boldsymbol{y}_i\right\rangle \cap\left\langle \mathcal{P}\boldsymbol{x}_1, \dots, \mathcal{P} \boldsymbol{x}_{i-1}\right\rangle^{\perp} .\]
	Then $\mathcal{P}^*\mathbf{z}_i \in\left\langle\boldsymbol{x}_1, \dots, \boldsymbol{x}_{i-1}\right\rangle^{\perp}$, and by Theorem \ref{thm-5},
	\[\lambda_{i}^{\downarrow}\left(\mathcal{A}\right)\ge \frac{\left\langle\mathcal{A}\mathcal{P}^*\boldsymbol{z}_i ,\mathcal{P}^*\boldsymbol{z}_i \right\rangle}{\left\langle\mathcal{P}^*\boldsymbol{z}_i ,\mathcal{P}^*\boldsymbol{z}_i \right\rangle}\ge\frac{\left\langle\mathcal{P}\mathcal{A}\mathcal{P}^*\boldsymbol{z}_i ,\boldsymbol{z}_i \right\rangle}{\left\langle\boldsymbol{z}_i ,\boldsymbol{z}_i \right\rangle}\ge\lambda_{i}^{\downarrow}\left(\mathcal{P}\mathcal{A}\mathcal{P}^*\right)=\lambda_{i}^{\downarrow}\left(\mathcal{A}_{\mathscr{K}}\right).\]
	As the second inequality, consider the self-operators $-\mathcal{A}$ and $-\mathcal{A}_{\mathscr{K}}$. It follows that
	\[\lambda_{i}^{\downarrow}\left(\mathcal{A}_{\mathscr{K}}\right)=-\lambda_{m-i+1}^{\downarrow}\left(-\mathcal{P}\mathcal{A}\mathcal{P}^*\right)\ge -\lambda_{m-i+1}^{\downarrow}\left(-\mathcal{A}\right)=\lambda_{n-m+i}^{\downarrow}\left(\mathcal{A}\right).
	\]
\end{proof}
\begin{remark}
It is worth noting that the subspace $\mathscr{K}$ in Lemma \ref{lem-11} has codimension $n-m$. Consequently, the present result may be viewed as a generalization of the finite-dimensional case in \cite{DD1987}, extending their finding from compressions onto subspaces of codimension $1$ to those of arbitrary codimension.
\end{remark}
\begin{corollary}\label{cor-1}
Let $A$ be the matrix representation of a compact self-adjoint operator $\mathcal{A}$ on $\mathscr{H}$ with respect to the orthogonal basis $\beta=\{\beta_i : i\in [n]\}$, and let $B$ be a principal submatrix of $A$ of order $m$. Then for all $i\in [m]$, 
\[
\lambda_{n-m+i}^{\downarrow}(A)\le \lambda_{i}^{\downarrow}(B)\le \lambda_{i}^{\downarrow}(A),
\]
or, equivalently,
\[
\lambda_{i}^{\uparrow}(A)\le \lambda_{i}^{\uparrow}(B)\le \lambda_{n-m+i}^{\uparrow}(A).
\]
\end{corollary}
\begin{proof}
 Let the columns of $B$ be indexed by $\mathcal{I}_B \subseteq [n]$, and let $\mathscr{H}_B $ be the subspace spanned by $\left\{\beta_i\colon i\in \mathcal{I}_B\right\}$. 
 Then $B$ is the matrix representation of the compression $\mathcal{A}_{\mathscr{H}_B}$ with respect to the basis $\left\{\beta_i\colon i\in \mathcal{I}_B\right\}$.  Since any linear operator on a finite-dimensional Hilbert space has the same eigenvalues as its matrix representation, the conclusion follows immediately from Lemma \ref{lem-11}.
\end{proof}
\begin{remark}
 If the basis $\beta=\{\beta_i : i\in[n]\}$ in Corollary \ref{c-1} is not orthonormal, then the matrix representations of self-adjoint operators are not necessarily real symmetric.
Therefore, our result extends Theorem \ref{thm-7}.
\end{remark}

The following lemma provides a necessary and sufficient condition under which a compact self-adjoint operator on a finite-dimensional Hilbert space and its compression to a subspace have the same smallest eigenvalue.
	\begin{lemma}\label{lem-me}
		Let $\mathcal{A}$ be a compact self-adjoint operator on $\mathscr{H}$, and let $\mathcal{A}_{\mathscr{K}}$ denote its compression to a subspace $\mathscr{K}$ of dimension $m$.  Then $\lambda_m^{\downarrow}\left(\mathcal{A}_{\mathscr{K}}\right)=\lambda_n^{\downarrow}\left(\mathcal{A}\right)$ if and only if $\lambda_n^{\downarrow}\left(\mathcal{A}\right)$ has a eigenfunction $\boldsymbol{x}\in \mathscr{K}\setminus \boldsymbol{0}$.        
	\end{lemma}
	\begin{proof}
		Let $\mathcal{P}\colon \mathscr{H} \rightarrow \mathscr{K}$ denote the orthogonal projection. Then $\mathcal{A}_{\mathscr{K}}= \mathcal{P} \mathcal{A} \mathcal{P}^*$.
		By Theorem \ref{thm-5}, we have
		\[\lambda_n^{\downarrow}\left(\mathcal{A}\right)=\min_{\boldsymbol{x}\in \mathscr{H}\setminus \boldsymbol{0}}\frac{\left\langle\mathcal{A}\boldsymbol{x},\boldsymbol{x}\right\rangle}{\left\langle\boldsymbol{x},\boldsymbol{x}\right\rangle}.\]
		where the minimum value is achieved if and only if $\boldsymbol{x}$ is an eigenfunction of $\mathcal{A}$ with respect to $\lambda_n^{\downarrow}\left(\mathcal{A}\right)$.
		By the Rayleigh quotient theorem,
		\[\lambda_m^{\downarrow}\left(\mathcal{P}\mathcal{A}\mathcal{P}^*\right)
		=\min_{\boldsymbol{x}\in \mathscr{K}\setminus \boldsymbol{0}}\frac{\left\langle\mathcal{P}\mathcal{A}\mathcal{P}^*\boldsymbol{x},\boldsymbol{x}\right\rangle}{\left\langle\boldsymbol{x},\boldsymbol{x}\right\rangle}=\min_{\boldsymbol{x}\in \mathscr{K}\setminus \boldsymbol{0}}\frac{\left\langle\mathcal{A}\boldsymbol{x},\boldsymbol{x}\right\rangle}{\left\langle\boldsymbol{x},\boldsymbol{x}\right\rangle}
		\ge \min _{\boldsymbol{x} \in \mathscr{H}\setminus\boldsymbol{0}} \frac{\left\langle\mathcal{A}\boldsymbol{x},\boldsymbol{x}\right\rangle}{\left\langle\boldsymbol{x},\boldsymbol{x}\right\rangle}=\lambda_n^{\downarrow}\left(\mathcal{A}\right).\]
		It follows that $\lambda_m^{\downarrow}\left(\mathcal{P}\mathcal{A}\mathcal{P}^*\right)=\lambda_n^{\downarrow}\left(\mathcal{A}\right)$ if and only if $\lambda_n^{\downarrow}\left(\mathcal{A}\right)$ has a eigenfunction $\boldsymbol{x}\in \mathscr{K}\setminus \boldsymbol{0}$.
	\end{proof}
	\begin{corollary}\label{cor-06}
		Let $A$ be the matrix representation of a compact self-adjoint operator $\mathcal{A}$ on $\mathscr{H}$ with respect to the orthogonal basis $\beta=\{\beta_i : i\in [n]\}$, and let $B$ be a principal submatrix of $A$ of order $m(m<n)$. Suppose that the columns of $B$ are indexed by $\mathcal{I}_B \subseteq [n]$. Then $\lambda_n^{\downarrow}(A)=\lambda_m^{\downarrow}(B)$ if and only if $A$ has an eigenvector $\operatorname{x}=\left(x_1, \dots, x_n\right)^{\top} \in \mathbb{R}^n$ with respect to $\lambda_n^{\downarrow}(A)$ such that $x_i=0$ for all $i \notin \mathcal{I}_B$.
	\end{corollary}
	\begin{proof}
		Let the columns of $B$ be indexed by $\mathcal{I}_B \subseteq [n]$, and let $\mathscr{H}_B $ be the subspace spanned by $\left\{\beta_i\colon i\in \mathcal{I}_B\right\}$. 
		Then $B$ is the matrix representation of the compression $\mathcal{A}_{\mathscr{H}_B}$ with respect to the basis $\left\{\beta_i\colon i\in \mathcal{I}_B\right\}$.  By Lemma \ref{lem-me} and the fact that any linear operator on a finite-dimensional Hilbert space has the same eigenvalues as its matrix representation, the conclusion follows immediately.
	\end{proof}
	\begin{remark}
		Zhan, Huang, and Lin \cite{ZHL2026} proved that Corollary 2 holds for real symmetric matrices $A$.  Our result in Corollary \ref{cor-06} extends this to matrix representations of self-adjoint operators with respect to an orthogonal basis, where the matrices involved are not necessarily real symmetric.
\end{remark}

The following theorem, which plays a crucial role for us, is commonly known as the Geršgorin circle theorem.
\begin{theorem}[{\cite[Theorem 6.1.1]{HJ2012}}]\label{thm-4}
	Let $A=\left(a_{i j}\right) \in \mathbb{C}^{n \times n}$, and let $\lambda \in \mathbb{C}$ be an eigenvalue of $A$. Let $\operatorname{x}=\left(x_1, \dots, x_n\right)^T \in \mathbb{C}^n$ be an eigenvector of $A$ with respect to $\lambda$, and let $i \in[n]$ be an index such that $\left|x_i\right|=\max _{1 \le j \le n}\left|x_j\right|$. Then
	\[	\left|\lambda-a_{i i}\right| \le \sum_{j \neq i}\left|a_{i j}\right|.\]
	Moreover, if the equality holds, then $\left|x_j\right|=\left|x_i\right|$ whenever $a_{i j} \neq 0$, and all complex numbers in $\left\{a_{i j} x_j\colon  j \neq i, a_{i j} \neq 0\right\}$ have the same argument.
\end{theorem}

\subsection{Additive compound matrices}

Let $V$ be an $n$-dimensional vector space over a field $\mathbb{F}$. For $1\le k \le n$, let $\bigwedge^k V$ be the $k$-th exterior power of $V$. Given a linear operator $\mathcal{M}: V \rightarrow V$, the \textit{$k$-th additive compound} of $\mathcal{M}$ is the linear operator $\mathcal{M}^{[k]}\colon \bigwedge^k V \rightarrow \bigwedge^k V$ defined by
\[\mathcal{M}^{[k]}\left(v_1 \wedge \cdots \wedge v_k\right)=\sum_{i=1}^k v_1 \wedge \cdots \wedge\left(\mathcal{M} v_i\right) \wedge \cdots \wedge v_k\]
for every $v_1, \dots, v_k \in V$. 
Let $e_1, \dots, e_n$ be the standard basis for $V$, and denote by $M$ the matrix representation of the operator $\mathcal{M}$ with respect to this basis. Then \[\left\{e_{i_1} \wedge \cdots \wedge e_{i_k}: 1 \leq i_1<\cdots<i_k \leq n\right\}\] forms a basis for the exterior power $\wedge^k V$. The matrix representation of the operator $\mathcal{M}^{[k]}$ with respect to this basis is called the \textit{$k$-th additive compound matrix} of $M$, and is denoted by $M^{[k]}$.

The study of additive compound operators was introduced by Wielandt \cite{W1967} in 1967, with further discussions available in \cite{MOA2011}. Schwarz \cite{S1970} and London \cite{L1976} investigated how additive compound operators can be applied to the differential equations. In \cite{F1974}, Fiedler proposed a family of generalized compound operators, which form a bridge between the classical multiplicative compounds and the additive compounds.

Let $\sigma$ and $\tau$ be two $k$-subsets of the ordered set $V$ such that $|\sigma \cap \tau| = k-1$, with $\sigma \setminus \tau = \{i\}$ and $\tau \setminus \sigma = \{j\}$. Denote by $\varepsilon(\sigma, \tau)$ the number of elements in $\sigma \cap \tau$ that are between $i$ and $j$. The following theorem provides an explicit formula for the additive compound of a matrix.
\begin{theorem}[\cite{F1974}]\label{thm-9}
	Let $M$ be an $n \times n$ matrix, and let $1 \le k \le n$. Then, $M^{[k]}$ is an $\binom{n}{k} \times\binom{ n}{k}$ matrix, with rows and columns indexed by the $k$-subsets of $[n]$, defined by
	\[	\left(M^{[k]}\right)_{\sigma, \tau}= \begin{cases}\sum _{i \in \sigma} M_{i, i} & \text { if } \sigma=\tau, \\[2mm] (-1)^{\varepsilon(\sigma, \tau)} M_{i, j} & \text { if }|\sigma \cap \tau|=k-1, \sigma \setminus \tau=\{i\}, \tau \setminus \sigma=\{j\}, \\[2mm] 0 & \text { otherwise },\end{cases}\]
	for every $\sigma, \tau \in\binom{[n]}{k}$.
\end{theorem}

The spectrum of the additive compound $M^{[k]}$ of $M$ is completely determined by that of $M$, and their relation is described as follows.

\begin{theorem}[\cite{F1974},Theorem 2.1]\label{thm-10}
	Let $M$ be an $n \times n$ matrix over a field $\mathbb{F}$, with eigenvalues $\lambda_1, \dots, \lambda_n$. Then, the $k$-th additive compound $M^{[k]}$ has eigenvalues $\lambda_{i_1}+\cdots+\lambda_{i_k}$, for $1 \leq i_1< \cdots<i_k \leq n$.
\end{theorem}

\subsection{High dimensional Laplacians}

The following result, originally formulated and proved by Eckmann \cite{E1944} and later restated and proved by Horak and Jost \cite{HJ2013}, is often referred to as the discrete Hodge theorem.
\begin{theorem}[\cite{E1944,HJ2013}]\label{thm-6}
 For an abstract simplicial complex $X$ with weight function $\omega$, \[\operatorname{ker}L^{\omega}_k(X) \cong \widetilde{H}^k(X;\mathbb{R}).\]
\end{theorem}

By identifying the weighted $k$-dimensional down-Laplacian and up-Laplacian, $L_k^{\omega \operatorname{down}}(X)$ and $L_k^{\omega \operatorname{up}}(X)$, with their matrix representations in the standard basis, we obtain the following explicit descriptions of $L_k^{\omega \operatorname{down}}(X)$ and $L_k^{\omega \operatorname{up}}(X)$, respectively.
\begin{theorem}[\cite{HJ2013}]\label{thm-13}
Let $X$ be a simplicial complex on vertex set $V$, and let $\omega\colon  X \rightarrow \mathbb{R}_{>0}$ be a weight function of $X$. Then, for all $k \geq-1$, $L_k^{\omega \operatorname{down}}(X)$ is $f_k(X) \times f_k(X)$ matrix, with rows and columns indexed by the $k$-dimensional simplices of $X$, defined by
\[	L_k^{\omega \operatorname{down}}(X)_{\sigma, \tau}= \begin{cases}\sum _{v \in \sigma} \frac{\omega(\sigma)}{\omega(\sigma \setminus\{v\})}& \text { if } \sigma=\tau, \\[3mm](-1)^{\varepsilon(\sigma, \tau)} \frac{\omega(\tau)}{\omega(\sigma \cap \tau)} & \text { if }|\sigma \cap \tau|=k,  \\[3mm] 0 & \text { otherwise, }\end{cases}\]
		for every $\sigma, \tau \in X(k)$.
\end{theorem}

\begin{theorem}[\cite{HJ2013}]\label{thm-15}
	Let $X$ be a simplicial complex on vertex set $V$, and let $\omega\colon  X \rightarrow \mathbb{R}_{>0}$ be a weight function of $X$. Then, for all $k \geq-1$, $L_k^{\omega \operatorname{up}}(X)$ is $f_k(X) \times f_k(X)$ matrix, with rows and columns indexed by the $k$-dimensional simplices of $X$, defined by
\[	L_k^{\omega \operatorname{up}}(X)_{\sigma, \tau}= \begin{cases}\sum _{u \in \operatorname{lk}_X(\sigma)} \frac{\omega(\sigma\cup \left\{u\right\})}{\omega(\sigma)}& \text { if } \sigma=\tau, \\[3mm]-(-1)^{\varepsilon(\sigma, \tau)} \frac{\omega(\sigma\cup \tau)}{\omega(\sigma)} & \text { if }|\sigma \cap \tau|=k, \sigma \cup \tau \in X,\\[3mm] 0 & \text { otherwise, }\end{cases}\]
for every $\sigma, \tau \in X(k)$.
\end{theorem}

By Theorem \ref{thm-13}, the following explicit formula for the vertex-weighted down-Laplacian follows immediately.
\begin{lemma}\label{lem-4}
	Let $X$ be a simplicial complex on vertex set $V$, and let $\omega\colon X \rightarrow \mathbb{R}_{>0}$ be a vertex weight function of $X$. Then, for all $k \geq-1$, $L_k^{\omega\operatorname{down}}(X)$ is an $f_k(X) \times f_k(X)$ matrix, with rows and columns indexed by the $k$-dimensional simplices of $X$, defined by
	\[L_k^{\omega\operatorname{down}}(X)_{\sigma, \tau}= \begin{cases}\sum _{v \in \sigma} \omega(v), & \text { if } \sigma=\tau, \\[3mm] (-1)^{\varepsilon(\sigma, \tau)} \omega(\tau \setminus \sigma), & \text { if }|\sigma \cap \tau|=k, \\[3mm] 0, & \text { otherwise, }\end{cases}\]	for every $\sigma, \tau \in X(k)$. \end{lemma}
    \begin{proof}
    By the definition of the vertex weight function, for any $\sigma \in X(k)$ we have
\[
\frac{\omega(\sigma)}{\omega(\sigma \setminus\{v\})}
 = \frac{\prod_{u\in \sigma}\omega(u)}{\prod_{u\in \sigma \setminus\{v\}}\omega(u)}
 = \omega(v),
\]
and for any $\sigma,\tau \in X(k)$ with $|\sigma \cap \tau| = k$,
\[
\frac{\omega(\sigma \cup \tau)}{\omega(\sigma)}
 = \frac{\prod_{u\in \sigma\cup\tau}\omega(u)}{\prod_{u\in \sigma}\omega(u)}
 = \omega(\tau\setminus\sigma).
\]
Thus, the results follows from Theorem \ref{thm-13} immediately.
\end{proof}

Let $\Delta_n$ denote the clique complex of the complete graph on $n+1$ vertices, which is also called the \textit{complete simplicial complex} on $n+1$ vertices.
Observe that while the vertex-weighted $k$-up Laplacian spectrum $\mathbf{s}_i^{\omega\operatorname{up}}$ is defined as the spectrum of the operator $L_k^{\omega \operatorname{up}}(X)= d_k^{\omega*} d_k$, which acts on $C^{k}(X,\mathbb{R})$, its nonzero eigenvalues are entirely determined by the $(k+1)$-simplices of $X$. This is because if an $k$-simplex $\sigma \in X$ is not contained in any $(k+1)$-simplex of $X$, then $d_k(e_{\sigma}) = 0$, and hence $e_{\sigma} \in \operatorname{Ker} L_k^{\omega,\operatorname{up}}(X)$.
For this reason, regard $L_k^{\omega,\operatorname{up}}(X)$ as acting on the $\binom{n}{k+1}$-dimensional real vector space $C^{k}(\Delta^{(k+1)}_{n-1}, \mathbb{R})$.
As a consequence of Theorem \ref{thm-15}, we obtain the following lemma.
\begin{lemma}\label{lem-14}
	Let $X$ be a simplicial complex on vertex set $V$, and let $\omega\colon X \rightarrow \mathbb{R}_{>0}$ be a vertex weight function of $X$. Then, for all $k \geq-1$, $L_k^{\omega \operatorname{up}}(X)$ is an $\binom{n}{k+1} \times \binom{n}{k+1}$ matrix, with rows and columns indexed by all elements in $\binom{[n]}{k+1}$, defined by
	\[	L_k^{\omega \operatorname{up}}(X)_{\sigma, \tau}= \begin{cases}\sum _{u \in \operatorname{lk}_X(\sigma)} \omega(u) & \text { if } \sigma=\tau, \\[3mm]-(-1)^{\varepsilon(\sigma, \tau)} \omega\left(\tau\setminus\sigma\right) & \text { if }|\sigma \cap \tau|=k, \sigma \cup \tau \in X,\\[3mm] 0  & \text { otherwise, }\end{cases}\]
	for every $\sigma, \tau \in \binom{[n]}{k+1}$.
\end{lemma}
\begin{proof}
 For any $\sigma, \tau \in \binom{[n]}{k+1}$ with $|\sigma \cap \tau|=k$ and $\sigma \cup \tau \in X$, we have $\sigma, \tau\in X(k)$. By Theorem \ref{thm-15},
 \[L_k^{\omega \operatorname{up}}(X)_{\sigma, \tau}=-(-1)^{\varepsilon(\sigma, \tau)} \frac{\omega(\sigma\cup \tau)}{\omega(\sigma)}=-(-1)^{\varepsilon(\sigma, \tau)} \frac{\omega(\sigma)\omega(\tau\setminus\sigma)}{\omega(\sigma)}=-(-1)^{\varepsilon(\sigma, \tau)} \omega(\tau\setminus\sigma).\]
If $\sigma\notin X(k)$, it follows from $d_k(e_{\sigma}) = 0$ that
\[L_k^{\omega \operatorname{up}}(X)(e_{\sigma})=0.\]  This implies that its matrix representation satisfies
\[L_k^{\omega \operatorname{up}}(X)_{\sigma, \sigma}=0=\sum _{u \in \operatorname{lk}_X(\sigma)}\omega(u)\] since $\operatorname{lk}_X(\sigma)=\emptyset$,
and
\[L_k^{\omega \operatorname{up}}(X)_{\sigma, \tau}=0\] for any $\sigma\ne \tau \in \binom{[n]}{k+1}$.
If $\sigma\in X(k)$, by Theorem \ref{thm-15} again, we have
\[L_k^{\omega \operatorname{up}}(X)_{\sigma, \sigma}=\sum _{u \in \operatorname{lk}_X(\sigma)} \frac{\omega(\sigma\cup \left\{u\right\})}{\omega(\sigma)}
 = \sum _{u \in \operatorname{lk}_X(\sigma)} \frac{\omega(\sigma)\omega(u)}{\omega(\sigma)}
 = \sum _{u \in \operatorname{lk}_X(\sigma)}\omega(u)\]
and \[L_k^{\omega \operatorname{up}}(X)_{\sigma, \tau}=0\] for any $\tau(\ne\sigma)\in \binom{[n]}{k+1}$ with $\sigma \cup \tau \notin X$.

This complete the proof.   
\end{proof}

The following lemma provides an explicit formula for the vertex-weighted Laplacian.
\begin{lemma}[\cite{L2024}]\label{lem-3} Let $X$ be a simplicial complex on vertex set $V$, and let $\omega\colon V \rightarrow \mathbb{R}_{>0}$. Then, for all $k \geq-1$, $L_k^{\omega}(X)$ is an $f_k(X) \times f_k(X)$ matrix, with rows and columns indexed by the $k$-dimensional simplices of $X$, defined by \[L_k^{\omega}(X)_{\sigma, \tau}= \begin{cases}\sum _{v \in \operatorname{lk}_X(\sigma)} \omega(v)+\sum _{v \in \sigma} \omega(v), & \text { if } \sigma=\tau, \\[4mm] (-1)^{\varepsilon(\sigma, \tau)} \omega(\tau\setminus\sigma), & \text { if } \sigma\sim \tau, \\[4mm] 0, & \text { otherwise, }\end{cases}\]	for every $\sigma, \tau \in X(k)$.\end{lemma} 

The notion of a \textit{join} of simplicial complexes goes back to the foundational work of Eilenberg and Zilber \cite{EZ1953} in algebraic topology. In 2002, Duval and Reiner \cite{DR2002} provided a standardized definition of this construction. 
\begin{definition}[{\cite[Definition 4.8]{DR2002}}]
 Given a simplicial complex $X_1$ on ordered vertex set $V_1$ and a simplicial complex $X_2$ on ordered vertex set $V_2$, where $V_1\cap V_2=\emptyset$, their join $X_1 * X_2$ is the simplicial complex on $V_1\cup V_2$ consisting of all simplices of the form $F_1 \cup F_2$,
where $F_1\in X_1, F_2\in X_2$ and the ordering on $V_1\cup V_2$ is obtained by keeping the internal orderings of $V_1$ and $V_2$ unchanged and placing all vertices of $V_2$ naturally after those of $V_1$.
\end{definition}

Later, Horak and Jost \cite{HJ2013} studied the cochain groups and coboundary operators of the join of simplicial complexes. 
Given a generalized weight function on the join, they analyzed the induced inner product on each cochain group, as well as the corresponding weighted adjoint coboundary operators with respect to this inner product. Let $X_1$ and  $X_2$ be two simplicial complex on vertex set $[m]$ and $[n]$, with wight function $\omega_{X_1}$ and $\omega_{X_1}$, respectively.
The cochain groups of $X_1 * X_2$ are
\[C^k\left(X_1 * X_2, \mathbb{R}\right)=\bigoplus_{k_1+k_2+1=k} C^{k_1}\left(X_1, \mathbb{R}\right) \otimes C^{k_2}\left(X_2, \mathbb{R}\right),\]
and the coboundary map $d_k$ is defined as the graded derivation
\[d_k(X_1 * X_2)(f \otimes g)=d_{k_1}(X_1)(f) \otimes g+(-1)^{|f|} f \otimes d_{k_2}(X_2)(g),\]
where $f \otimes g \in C^k\left(X_1 * X_2, \mathbb{R}\right)$, $f\in C^{k_1}\left(X_1, \mathbb{R}\right)$, $g\in C^{k_2}\left(X_2, \mathbb{R}\right)$ and $|f|$ denotes the order of a cochain group which contains $f$. Let
\[p\colon\left\{0,1,\dots,\operatorname{dim}(X_1)\right\} \rightarrow \mathbb{R}^{+}\text{ and } q\colon\left\{0,1,\dots, \operatorname{dim} (X_2)\right\} \rightarrow \mathbb{R}^{+}\]
are positive, real valued functions. 
The weight function of the join $X_1 * X_2$ is given by \begin{equation}\label{eq-p}
    \omega_{X_1 * X_2}\left(F_1 \otimes F_2\right)=p\left(\operatorname{dim} (F_1)\right) \omega_{X_1}\left(F_1\right) q\left(\operatorname{dim} (F_2)\right) \omega_{X_2}\left(F_2\right),
\end{equation} for all $F_1\in X_2$ and $F_2\in X_2$.
An induced inner product on $C^k\left(X_1 * X_2, \mathbb{R}\right)$ is 
\[\langle f_1 \otimes g_1, f_2 \otimes g_2\rangle=p\left(k_1\right)\langle f_1, f_2\rangle_{C^{k_1}\left(X_1, \mathbb{R}\right)} q\left(k_2\right)\langle g_1, g_2\rangle_{C^{k_2}\left(X_2, \mathbb{R}\right)},\]
where $f_1, f_2 \in C^{i_1}\left(X_1, \mathbb{R}\right)$, $ g_1, g_2 \in C^{i_2}\left(X_2, \mathbb{R}\right)$, and $k_1+k_2+1=i$.
The corresponding weighted adjoint operator is given by 
\[d_k^{\omega_{X_1 * X_2} *}(f \otimes g)=\frac{p(|f|)}{p(|f|-1)} d_{k_1}^{\omega_{X_1} *} {f} \otimes g+\frac{q(|g|)}{q(|g|-1)}(-1)^{|f|} f \otimes d_{k_2}^{\omega_{X_2} *}(g),\] for all $f\in C^{k_1}\left(X_1, \mathbb{R}\right)$ and $g\in C^{k_2}\left(X_2, \mathbb{R}\right)$.

Based on the above constructions, Horak and Jost defined the weighted Laplacian (see \cite[The proof of Theorem 6.5]{HJ2013}) of the join $X_1*X_2$ as the graded derivation
\begin{equation}\label{eq-j}
    \begin{aligned}
L^{\omega}_k(X_1*X_2)(f \otimes g)= & \left(\frac{p(|f|+1)}{p(|f|)} L^{\omega_{X_1} \operatorname{up}}_{|f|}(X_1)+\frac{p(|f|)}{p(|f|-1)} L^{\omega_{X_1}\operatorname{down} }_{|f|}(X_1)\right) f \otimes g \\
& +f \otimes\left(\frac{q(|g|+1)}{q(|g|)} L^{\omega_{X_2} \operatorname{up}}_{|g|}(X_2)+\frac{q(|g|)}{q(|g|-1)} L^{\omega_{X_2} \operatorname{down}}_{|g|}(X_2)\right) g
\end{aligned}
\end{equation}
for any $f \otimes g\in C^k\left(X_1 * X_2, \mathbb{R}\right).$

For vertex-weighted simplicial complexes, we establish a spectral formula for the Laplacian of the join. This result generalizes the combinatorial Laplacian formula for the join operation given by Duval and Reiner (see \cite[Theorem 4.10]{DR2002}).

\begin{theorem}\label{thm-j}
	Let $X_i$ be a simplicial complex on vertex set $V_i$, with vertex weight function $\omega_i \colon X_i \to \mathbb{R}_{>0}$ for $i = 1, \dots, m$, where the sets $V_i$ are pairwise disjoint. Let
	\[X = X_1 * \cdots * X_m\]
	be their join, endowed with the vertex weight function $\omega$ which is defined as $\omega(v)=\omega_i(v)$ for $v\in V_i$. Then for every $k \ge -1$, we have
    \[\mathbf{s}_{k}^{\omega}(X)=\bigcup_{\substack{\sum_{j\in [m]}i_j=k-(m-1)\\-1 \le i_j \le \dim(X_j),\ \forall j \in [m]\\\lambda_j \in \mathbf{s}_{i_j}^{\omega_j}(X_j)}}\sum_{j\in [m]}\lambda_j.\]
\end{theorem}
\begin{proof}
For any $\sigma = \sigma_1 \otimes \sigma_2 \otimes \dots \otimes \sigma_m \in X$, the vertex weight function $\omega$ satisfies
\[
\omega(\sigma) = \prod_{v \in \sigma} \omega(v) = \prod_{i \in [m]} \prod_{v \in \sigma_i} \omega_i(v) = \prod_{i \in [m]} \omega_i(\sigma_i).
\]
This shows that $\omega$ is precisely the weight function on the join complex $X_1 * \cdots * X_m$ as defined in \eqref{eq-p}, in the special case where the functions $p$ and $q$ are both taken to be the constant function $1$. Therefore, from \eqref{eq-j} we have
\[\begin{aligned}
    L^{\omega}_k(X)(f_1 \otimes f_2\otimes \cdots \otimes f_m)&=\left(L^{\omega_1\operatorname{up}}_{|f_1|}(X_1)+L^{\omega_1 \operatorname{down}}_{|f_1|}(X_1)\right)(f_1) \otimes f_2\otimes \dots \otimes f_m\\&\ \ \ +f_1\otimes\left(L^{\omega_2 \operatorname{up}}_{|f_2|}(X_2)+L^{\omega_2 \operatorname{down}}_{|f_2|}(X_2)\right)(f_2)\otimes \dots \otimes f_m\\&\ \ \ +\dots+f_1\otimes f_2\otimes\dots \otimes \left(L^{\omega_m \operatorname{up}}_{|f_m|}(X_m)+L^{\omega_m \operatorname{down}}_{|f_m|}(X_m)\right)(f_m)\\&=\sum^{m}_{i=1}f_1\otimes\dots\otimes L^{\omega}_{|f_i|}(X_i)(f_i)\otimes \dots \otimes f_m.
\end{aligned} \]
Therefore, for every $k \ge -1$, we have
	\[\mathbf{s}_{k}^{\omega}(X)=\bigcup_{\substack{\sum_{j\in [m]}i_j=k-(m-1)\\-1 \le i_j \le \dim(X_j),\ \forall j \in [m]\\\lambda_j \in \mathbf{s}_{i_j}^{\omega_j}(X_j)}}\sum_{j\in [m]}\lambda_j.\]
\end{proof}

\subsection{Geometric realization}
We adopt the notation of \cite{MKP2025} for algebraic topology. For completeness, we briefly recall the notation and some basic concepts concerning geometric simplicial complexes, which could be found in, for example, \cite{H2002,M2003,D2008}.

Let $F=\{\boldsymbol{v}_0, \boldsymbol{v}_1, \dots, \boldsymbol{v}_k\}$ be a set of points in $\mathbb{R}^n$.  
The set $F$ is \textit{geometrically independent} if, for any real numbers $t_i$, the conditions
\[
\sum_{i=0}^k t_{\boldsymbol{v}_i} = 0 
\qquad\text{and}\qquad
\sum_{i=0}^k t_{\boldsymbol{v}_i} \boldsymbol{v}_i = \mathbf{0}
\]
imply $t_{\boldsymbol{v}_0} = t_{\boldsymbol{v}_1} = \cdots = t_{\boldsymbol{v}_k} = 0$.  
Otherwise, $F$ is said to be \textit{geometrically dependent}. If $F$ is geometrically independent, define the \textit{geometric $k$-simplex} $\sigma_F$ spanned by $F$ to be the set
\[\sigma_F\colon=\left\{\sum^{k}_{i=1}t_{\boldsymbol{v}_i}\boldsymbol{v}_i \colon\ \sum^{k}_{i=1}t_{\boldsymbol{v}_i}=1\text{ and } t_{\boldsymbol{v}_i}\ge 0 \text{ for all } \boldsymbol{v}_i\in F \right\}.\]
The points of $F$ are called the \textit{vertices} of $\sigma_F$, and the \textit{dimension} of $\sigma_F$ is $\operatorname{dim}(\sigma_F)=|F|-1$. Thus, every $k$-dimensional geometric simplex has $k+1$ vertices. Any geometric simplex spanned by a subset of $F$ is called a \textit{face} of $\sigma_F$. The union of the all faces of $\sigma_F$ different from $\sigma_F$ itself is called the \textit{boundary} of $\sigma$ and is denoted $\operatorname{Bd}\sigma_F$.

A collection $K$ of geometric simplices in $\mathbb{R}^n$ is a \textit{geometric simplicial complex} if it satisfies:
\begin{enumerate}[(1)]
		\item Each face of any simplex $\sigma \in K$ is also a simplex of $K$;
		\item  The intersection of any two simplices of $K$ is a face of each of them.
	\end{enumerate}
Endowing each geometric simplex in $K$ with the subspace topology inherited from $\mathbb{R}^n$, the union of all geometric simplices of $K$ forms a topological space $\|K\|$, where $A \subseteq \|K\|$ is closed if and only if $A \cap \sigma$ is closed in $\sigma$ for every $\sigma \in K$. The space $\|K\|$ is called the \textit{underlying space} of $K$, or the \textit{polyhedron} of $K$.

Each geometric simplicial complex $K$ determines an abstract simplicial complex $X_K$, called the \textit{vertex scheme} of $K$, whose vertex set $V$ consists of the vertices of $K$ and whose simplices are precisely the subsets $\{\boldsymbol{v}_0, \boldsymbol{v}_1, \dots, \boldsymbol{v}_k\} \subseteq V$ that span a simplex of $K$. For any abstract simplicial complex $X$, the simplicial complex $K$ is said to be a \textit{geometric realization} of $X$ if $X$ is isomorphic to $X_K$, and the polyhedron of $K$ is also referred to as a polyhedron of $X$. 
Moreover, every abstract simplicial complex $X$ has a geometric realization $K^X$. Let the vertex set of $X$ be $\{v_0,v_1,\dots,v_k\}$ with $k\le n$.  
Identify $\{v_0,v_1,\dots,v_k\}$ with a geometrically independent set 
$\{\boldsymbol{v}_0, \boldsymbol{v}_1, \dots, \boldsymbol{v}_k\}\subseteq \mathbb{R}^{n}$ 
by setting $v_i = \boldsymbol{v}_i$ for each $i$.
The collection $K^X\colon=\{\sigma_F: F \in X\}$ is a geometric simplicial complex, and its vertex scheme is $X$. The relationships among $X$, $K^X$, and $\|K^X\|$ are illustrated in 
Fig.\ref{fig-2}.
\begin{figure}[htbp]
    \centering
    \includegraphics[width=13cm]{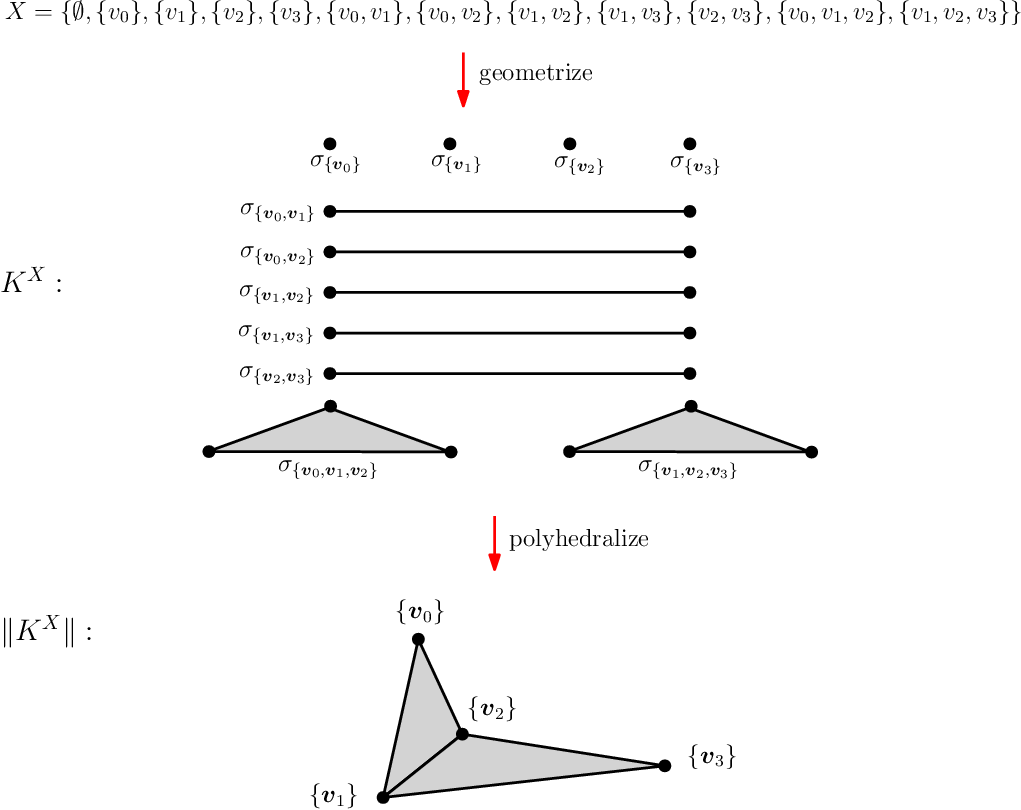}
    \caption{The relationships among $X$, $K^X$, and $\| K^X\|$}
    \label{fig-2}
\end{figure}

Let $\mathcal{X}$ and $\mathcal{Y}$ be two topological spaces. A mapping $f: \mathcal{X} \to \mathcal{Y}$ is called a \textit{homeomorphism} if it is bijective and both $f$ and $f^{-1}$ are continuous. The spaces $\mathcal{X}$ and $\mathcal{Y}$ are said to be \textit{homeomorphic} if there exists a homeomorphism between them, which is denoted by $\mathcal{X} \cong \mathcal{Y}$. An abstract simplicial complex $X$ is a \textit{triangulation} of a topological space $\mathcal{X}$ if its polyhedron $\|K^X\|$ is homeomorphic to $\mathcal{X}$. 
Recall that $\Delta_k^{(k-1)}$ is the $(k-1)$-skeleton of the abstract complete simplicial complex $\Delta_k$ on $k+1$ vertices. Identifying the vertex set $V(\Delta_k^{(k-1)})$ of $\Delta_k^{(k-1)}$ with a geometrically independent set $F=\{\boldsymbol{v}_0, \boldsymbol{v}_1, \dots, \boldsymbol{v}_k\}$ in $\mathbb{R}^n$, the boundary $\operatorname{Bd}\sigma_{F}$ of $\sigma_{F}$ is a geometric realization of $\Delta_k^{(k-1)}$. By elementary algebraic topology, the polyhedron $\|\operatorname{Bd}\sigma_{F}\|$ of the boundary of the geometric $k$-simplex $\sigma_{F}$ is homeomorphic to $(k-1)$-sphere $S^{k-1}$, and thus the following result holds.
\begin{lemma}[{\cite[p.10]{M2003}}]\label{lem-8}
Let $k$ be a positive integer. Then $\Delta_k^{(k-1)}$ is a triangulation of $S^{k-1}$.

\end{lemma}

The join of two topological spaces $\mathcal{X}$ and $\mathcal{Y}$, denoted by $\mathcal{X} * \mathcal{Y}$, is defined as the quotient space
\[
\mathcal{X} \times \mathcal{Y} \times [0,1] \big/ \!\sim,
\]
where the equivalence relation $\sim$ is specified by
\[
(x_1, y, 0) \sim (x_2, y, 0) \quad \text{for all } x_1, x_2 \in \mathcal{X}, y \in \mathcal{Y},
\]
and
\[
(x, y_1, 1) \sim (x, y_2, 1) \quad \text{for all } x \in \mathcal{X}, y_1, y_2 \in \mathcal{Y}.
\]
The following two lemmas characterize the structure of the join of two spheres and the structure of the join of polyhedra of two abstract simplicial complexes, respectively.
\begin{lemma}[{\cite[p.19, Exercise 18]{H2002}}]\label{lem-9}
Let $n$ and $m$ be two non-negative integers. Then
\[S^n * S^m \cong S^{n+m+1}.\]
\end{lemma}
\begin{lemma}[{\cite[p.77, Exercise 3]{M2003}}]\label{lem-10}
	Let $X$ and $Y$ be two finite abstract simplicial complexes. Then
\[\|K^X\| *\|K^Y\| \cong\|K^{X * Y}\|.\]
\end{lemma}

\section{Constructions and their effect on the vertex-weighted Laplacian spectra}
In this section, we introduce several new simplicial complexes constructed from a given simplicial complex and investigate the relationships between their vertex-weighted Laplacian spectra and that of the original complex. In addition, for a simplicial complex $X$ and each $-1\le k\le \operatorname{dim}(X)$, we establish a new upper bound for the largest eigenvalue $\lambda^{\downarrow}_{1}\left(L^{\omega}_k(X)\right)$ of the vertex-weighted $k$-dimensional Laplacian operator $L^{\omega}_k(X)$.

\begin{definition}
Let $X$ be a simplicial complex on $V$ with $|V|=n$. For each $-1\le k\le \operatorname{dim}(X)$, the simplicial complexes $X_k^*$ and $X^c_k$ are defined as
\[X_k^*\colon=\left\{\tau\subseteq V\colon \exists \ \sigma\in X(k) \text{ such that } \tau \subseteq V\setminus \sigma\right\},\]
\[X_k^{c}\colon=\left\{\tau\subseteq V\colon \exists \ \sigma\in \tbinom{V}{k+1}\setminus X(k) \text{ such that } \tau \subseteq \sigma\right\}.\]
\end{definition}
\subsection{Vertex-wighted Laplacian spectra of $X$ and $X^*$}

		\begin{theorem}\label{thm-14}
	Let $X$ be a simplicial complex on vertex set $V$, and let $\omega\colon X \rightarrow \mathbb{R}_{>0}$ be a vertex weight function of $X$. For any $0\le k\le \operatorname{dim}(X)$ and $1\le i\le f_k(X)$, we have \begin{equation}\label{eq-18}
		\lambda_{i}^{\downarrow}\left(L_k^{\omega\operatorname{down}}\left(X\right)\right)+\lambda_{f_k(X)+1-i}^{\downarrow}\left(L_{n-k-2}^{\omega\operatorname{down}}\left(X_k^*\right)\right)=\sum_{v \in V} \omega(v).
	\end{equation}
		\end{theorem}
	  \begin{proof}
		Without loss of generality, assume $V=[n]$ and assigned the natural orientation to $[n]$. Let $
			\sigma_1, \dots, \sigma_{f_k(X)}$ denote all $k$-simplices of $X$.
			For convenience, we may assume that $L_k^{\omega\operatorname{down}}(X)$ and $L_{n-k-2}^{\omega\operatorname{down}}\left(X_k^*\right)$ are both $f_k(X)\times f_k(X)$ matrices, where the $(i,j)$-entry of $L_k^{\omega\operatorname{down}}$ is the entry corresponding to $(\sigma_i,\sigma_j)$ and the $(i,j)$-entry of $L_{n-k-2}^{\omega\operatorname{down}}\left(X_k^*\right)$ is the entry corresponding to $(V\setminus \sigma_i,V\setminus\sigma_j)$ for $1\le i,j\le f_k(X)$.
			Define $\psi$ as the $f_k(X) \times f_k(X)$ diagonal matrix with diagonal entry $\psi_{i,i}\colon=(-1)^{\sum _{s\in V\setminus\sigma_i}s}$ for $1\le i\le f_k(X)$. It is easy to verify that $\psi$ is invertible.
			Therefore, by Lemma \ref{lem-4}, we have
					\[\left(\psi^{-1}L_{n-k-2}^{\omega\operatorname{down}}(X^{*})\psi\right)_{i, j}= \begin{cases}\sum _{v \in {V\setminus\sigma_i}} \omega(v) & \text { if } i=j, \\[6mm] (-1)^{\sum _{v\in V\setminus\sigma_i}v+\sum _{v\in V\setminus\sigma_j}v}(-1)^{\varepsilon(V\setminus\sigma_i, V\setminus\sigma_j)} \omega(\sigma_i\setminus \sigma_j)  & \text { if }|\sigma_i \cap \sigma_j|=k, \\[6mm] 0 & \text { otherwise. }\end{cases}\]
				Now, consider the $f_k(X) \times f_k(X)$ matrix $L_k^{\omega\operatorname{down}}(X)+\left(\psi^{-1}L_{n-k-2}^{\omega\operatorname{down}}(X^{*})\psi\right)^{\top}$. For $1\le i\le f_k(X)$, the diagonal $(i,i)$-entry is
				\[\sum _{v \in {\sigma_i}} \omega(v)+\sum _{v \in {V\setminus\sigma_i}}\omega(v)=\sum _{v \in {V}} \omega(v).\]
				 For $1\le i,j\le f_k(X)$ with $|\sigma_i \cap \sigma_j|=k$, assume $\sigma_i\setminus \sigma_j=\{u\}$ and $\sigma_j\setminus \sigma_i=\{u^{\prime}\}$. Then, the $(i,j)$-entry is given as
						\[\begin{array}{lll}
						&&(-1)^{\varepsilon(\sigma_i, \sigma_j)} \omega(u^{\prime})+(-1)^{\sum _{v\in V\setminus\sigma_i}v+\sum _{v\in V\setminus\sigma_j}v}(-1)^{\varepsilon(V\setminus\sigma_i, V\setminus\sigma_j)} \omega(u^{\prime})\\[2mm]
                        &=&(-1)^{\varepsilon(\sigma_i,\sigma_j)} \omega(u^{\prime})+(-1)^{u+u^{\prime}}(-1)^{\varepsilon(V\setminus\sigma_i, V  \setminus\sigma_j)} \omega(u^{\prime})\\[2mm]&=&(-1)^{\varepsilon(\sigma_i, \sigma_j)} \omega(u^{\prime})+(-1)^{u+u^{\prime}}(-1)^{|u-u^{\prime}|-1}(-1)^{\varepsilon(\sigma_i, \sigma_j)} \omega(u^{\prime})\\&=&0,
					\end{array} \]
                    where the second equality holds because $\varepsilon(V\setminus\sigma_i,V\setminus\sigma_j)+\varepsilon(\sigma_i,\sigma_j)=|u-u'|-1$. 
	This implies that \[L_k^{\omega\operatorname{down}}(X)+\left(\psi^{-1}L_{n-k-2}^{\omega\operatorname{down}}(X^{*})\psi\right)^{\top}=\sum_{v \in V} \omega(v)I_{f_k(X)},\]
	where $I_{f_k(X)}$ is the $f_k(X) \times f_k(X)$ identity matrix. Therefore, we have
\[\lambda_{i}^{\downarrow}\left(L_k^{\omega\operatorname{down}}\left(X\right)\right)+\lambda_{f_k(X)+1-i}^{\downarrow}\left(L_{n-k-2}^{\omega\operatorname{down}}\left(X_k^*\right)\right)=\sum_{v \in V} \omega(v).\]
This completes the proof.
		\end{proof}
		
Applying Theorem \ref{thm-14}, we could get the spectra of simplicial complexes with complete skeleton.		
\begin{lemma}\label{lem-15}
	Let $X$ be a simplicial complex with complete $p$-skeleton on vertex set $V$ of size $n$, and let $\omega\colon X \rightarrow \mathbb{R}_{>0}$ be a vertex weight function of $X$. Then \[\mathbf{s}_{k }^{\omega}\left(X\right)=\left\{\left(\sum_{v\in V}\omega(v)\right)^{\left(\binom{n}{k+1}\right)} \right\},\]
	for each $-1\le k\le p-1$. In particular, if $k=p=\operatorname{dim} (X)$,  then 
	\[\mathbf{s}_{k }^{\omega}\left(X\right)=\left\{\left(\sum_{v\in V}\omega(v)\right)^{\left(\binom{n-1}{k}\right)}, 0^{\left(\binom{n-1}{k+1}\right)} \right\}.\]
\end{lemma}
\begin{proof}
	Recall that $L^{\omega}_{k}(X)=L_{k}^{\omega \operatorname{up}}(X)+L_{k}^{\omega \operatorname{down}}(X)=d_{k}^{\omega*} d_{k}+d_{k-1} d_{k-1}^{\omega*}$. If $k=-1$, then $L^{\omega}_{-1}(X)=L_{-1}^{\omega \operatorname{up}}(X)$, and we get $\mathbf{s}_{-1}^{\omega}\left(X\right)=\left\{\sum_{v\in V}\omega(v) \right\}$ immediately from Lemma \ref{lem-14}. If $-1<k\le p-1$, since $d^{\omega *}_{k-1}d^{\omega *}_{k } = 0$, we have 
	$\operatorname{im}\left(d^{\omega *}_{k}\right)\subseteq \operatorname{Ker}\left(d_{k-1}^{\omega*}\right)=\operatorname{Ker}\left(L^{\omega \operatorname{down}}_{k }(X)\right)$. Taking $v_0\in V$, by noticing that $\left\{d^{\omega *}_{k }\left(e_{\sigma}\right)\colon v_0\in \sigma\in X(k+1)\right\}\subseteq \operatorname{im}\left(d^{\omega *}_{k}\right)$ are linearly independent, we have 
   \[ \operatorname{dim}\operatorname{Ker}\left(L^{\omega \operatorname{down}}_{k}(X)\right)\ge \operatorname{dim}\operatorname{im}\left(d^{\omega *}_{k}\right)\ge \left|\left\{d^{\omega *}_{k }\left(e_{\sigma}\right)\colon v_0\in \sigma\in X(k+1)\right\}\right|=\binom{n-1}{k+1}. \]
  Thus, $L^{\omega \operatorname{down}}_{k }(X)$ has at most 
	$\binom{n}{k+1}-\binom{n-1}{k+1}=\binom{n-1}{k }$ non-zero eigenvalues.
Therefore, from Theorem \ref{thm-14}, we get
   \begin{equation}\label{eq-w-1}
       \sum_{i=1}^{f_{k }}\lambda_{i}^{\downarrow}(L_{k }^{\omega \operatorname{down}}(X))\le \binom{n-1}{k }\lambda_{1}^{\downarrow}(L_{k}^{\omega,\operatorname{down}}(X))\le\binom{n-1}{k }\sum_{v\in V}\omega(v).
   \end{equation}
On the other hand, by Lemma \ref{lem-4}, we have
	\[\sum_{i=1}^{f_{k }}\lambda_{i}^{\downarrow}(L_{k }^{\omega \operatorname{down}}(X))=\sum_{\sigma\in X(k )}\sum _{v \in \sigma} \omega(v)=\binom{n-1}{k }\sum_{v\in V}\omega(v).\]
This implies that all the inequalities in \eqref{eq-w-1} are actually equalities.
	Therefore, all non-zero eigenvalues of $L_{k}^{\omega,\operatorname{down}}(X)$ are $\sum_{v\in V}\omega(v)$, and its mulitiplicity is $\binom{n-1}{k }$, i.e.	\begin{equation}\label{eq-2}
		\mathbf{s}_{k }^{\omega,\operatorname{down}}\left(X\right)=\left\{\left(\sum_{v\in V}\omega(v)\right)^{\left(\binom{n-1}{k}\right)}, 0^{\left(\binom{n-1}{k+1}\right)} \right\}.
	\end{equation}
	Similarly, 	since $d_{k }d_{k-1} = 0$, we obtain
	$\operatorname{im}\left(d_{k-1}\right)\subseteq \operatorname{Ker}\left(d_{k }\right)=\operatorname{Ker}\left(L^{\omega \operatorname{up}}_{k }(X)\right)$. Combining the fact	that $\left\{d_{k-1}\left(e_{\sigma}\right)\colon v_0\notin \sigma\in X(k-1)\right\}\subseteq \operatorname{im}\left(d_{k-1}\right)$ are linearly independent, it follows that \[\operatorname{dim}\operatorname{Ker}\left(L^{\omega \operatorname{up}}_{k }(X)\right)\ge \operatorname{dim}\operatorname{im}\left(d_{k-1}\right)\ge \left|\left\{d_{k-1}\left(e_{\sigma}\right)\colon v_0\notin \sigma\in X(k-1)\right\}\right|=\binom{n-1}{k},\]
  Therefore, $L^{\omega \operatorname{up}}_{k }(X)$ has at most 
	\[\binom{n}{k+1}-\binom{n-1}{k }=\binom{n-1}{k+1}\] non-zero eigenvalues. Combining \eqref{eq-05} and Theorem \ref{thm-14}, we have
    $\lambda_{1}^{\downarrow}(L_{k}^{\omega\operatorname{up}}(X))\le \sum_{v\in V}\omega(v)$, and thereby \begin{equation}\label{eq-w-2}
        \sum_{i=1}^{f_{k}}\lambda_{i}^{\downarrow}(L_{k}^{\omega \operatorname{up}}(X))\le \binom{n-1}{k+1}\lambda_{1}^{\downarrow}(L_{k}^{\omega,\operatorname{up}}(X))\le \binom{n-1}{k+1}\sum_{v\in V}\omega(v).
    \end{equation}
On the other hand, by Lemma \ref{lem-14}, we obtain
	\[\sum_{i=1}^{f_{k}}\lambda_{i}^{\downarrow}(L_{k}^{\omega \operatorname{up}}(X))=\sum_{\sigma\in X(k)}\sum _{v \in \operatorname{lk}_X(\sigma)} \omega(v)=\binom{n-1}{k+1}\sum_{v\in V}\omega(v).\]
This implies that each inequality in \eqref{eq-w-2} holds with equality.
Therefore, all non-zero eigenvalues of $L_{k}^{\omega\operatorname{up}}(X)$ are $\sum_{v\in V}\omega(v)$, and its mulitiplicity is $\binom{n-1}{k+1}$, i.e.,
	\begin{equation}\label{eq-0}
		\mathbf{s}_{k }^{\omega\operatorname{up}}\left(X\right)=\left\{\left(\sum_{v\in V}\omega(v)\right)^{\left(\binom{n-1}{k+1}\right)}, 0^{\left(\binom{n-1}{k}\right)} \right\}.
	\end{equation}
Combining \eqref{eq-01}, \eqref{eq-2}, and \eqref{eq-0}, we have
	\[\mathbf{s}_{k }^{\omega}\left(X\right)=\left\{\left(\sum_{v\in V}\omega(v)\right)^{\left(\binom{n}{k+1}\right)} \right\}.\]
	In particular, if $k=p=\operatorname{dim}(X)$,  then $L^{\omega}_{k}(X)=L_{k}^{\omega \operatorname{down}}(X)$, and thereby
	$\mathbf{s}_{k }^{\omega}\left(X\right)=\mathbf{s}_{k }^{\omega \operatorname{down}}\left(X\right).$
	
	We complete this proof.
\end{proof}

Recall that for the complete simplicial complex $\Delta_n$ on vertex set $V$ of size $n+1$, all of its skeleton are complete. Hence, the following result is immediately from Lemma \ref{lem-15}.
\begin{corollary}\label{cor-2} Let $\omega\colon \Delta_{n-1} \rightarrow \mathbb{R}_{>0}$ be a vertex weight function of $\Delta_{n-1}$. Then
	\[\mathbf{s}_{k }^{\omega}\left(\Delta^{(p)}_{n-1}\right)=\begin{cases}
		\left\{\left(\sum_{v\in V}\omega(v)\right)^{\left(\binom{n}{k+1}\right)} \right\}\quad\quad\quad\quad, &  -1\le k\le p-1, \\[6mm]
		\left\{\left(\sum_{v\in V}\omega(v)\right)^{\left(\binom{n-1}{k}\right)}, 0^{\left(\binom{n-1}{k+1}\right)}\right\},	& k=p.
	\end{cases}\]
\end{corollary}
\begin{remark} Gundert and Wagner \cite{GW2016} proved that the spectra of the combinatorial  $k$-dimensional Laplacian operators $L_k(X)$ of $\Delta^{(p)}_{n-1}$ is
\[\mathbf{s}_{k }\left(\Delta^{(p)}_{n-1}\right)=\begin{cases}
		\left\{n^{\left(\binom{n}{k+1}\right)} \right\}\quad\quad\quad\ \ , &  -1\le k\le p-1, \\[6mm]
		\left\{n^{\left(\binom{n-1}{k}\right)}, 0^{\left(\binom{n-1}{k+1}\right)}\right\},	& k=p.
	\end{cases}\]
Therefore, for $\Delta^{(p)}_{n-1}$, Corollary \ref{cor-2} extends the spectral result from the combinatorial Laplacian operators to the vertex-weighted Laplacian operators.
\end{remark}		

From Theorem \ref{thm-14}, we could also establish an upper bound for the largest eigenvalue $\lambda^{\downarrow}_{1}\left(L^{\omega}_k(X)\right)$ of the vertex-weighted $k$-dimensional Laplacian operator $L^{\omega}_k(X)$. 
\begin{theorem}\label{thm-18}
		Let $X$ be a simplicial complex on vertex set $V$ with $|V|=n$, and let $\omega\colon X \rightarrow \mathbb{R}_{>0}$ be a vertex weight function of $X$. 
 For every $0 \le k \le \operatorname{dim}(X)$, we have
 \begin{equation}\label{eq-19}
 	 \lambda^{\downarrow}_{1}\left(L^{\omega}_k(X)\right)\le \sum_{v \in V} \omega(v).
 \end{equation}
	Moreover, if $f_k(X)>\binom{n-1}{k+1}$ or $f_{k+1}(X)>\binom{n-1}{k+2}$, then $\sum_{v\in V}\omega(v)$ is an eigenvalue of $L_k^{\omega}(X)$ with mulitiplicity at least \[\max\left\{f_k(X)+f_{k+1}(X)-\binom{n}{k+2}, f_k(X)-\binom{n-1}{k+1},f_{k+1}(X)-\binom{n-1}{k+2}\right\}.  \]
\end{theorem}
\begin{proof}
Since $\lambda_{i}^{\downarrow}\left(L_k^{\omega\operatorname{down}}\left(X\right)\right)$ and $\lambda_{f_k(X)+1-i}^{\downarrow}\left(L_{n-k-2}^{\omega\operatorname{down}}\left(X_k^*\right)\right)$ are nonnegative, it follows from Theorem \ref{thm-14} that \begin{equation}\label{eq-11}
\lambda_{i}^{\downarrow}\left(L_k^{\omega\operatorname{down}}\left(X\right)\right)\le \sum_{v \in V} \omega(v)
\end{equation}
for any $0\le k\le \operatorname{dim}(X)$ and $1\le i\le f_k(X)$. By Equations \eqref{eq-01} and \eqref{eq-05}, we have
 \[\lambda^{\downarrow}_{1}\left(L^{\omega}_k(X)\right)\le \sum_{v \in V} \omega(v).\]

From Lemma \ref{lem-4}, for every $u,v \in V$, we have \[L_0^{\omega\operatorname{down}}(X)_{u, v}= \begin{cases}\omega(u) & \text { if } u=v, \\  \omega(v) & \text { if } u \ne v.  \end{cases}\] Let $L_0^{\omega\operatorname{down}}(X)^{[k+1]}$ be the $(k+1)$-th additive compound matrix of $L_0^{\omega\operatorname{down}}(X)$, and let $L$ be the principal submatrix of $L_0^{\omega\operatorname{down}}(X)^{[k+1]}$ obtained by removing all rows and columns corresponding to the $(k+1)$-sets not in $X$. Therefore, by Lemma \ref{thm-9}, we have
	\[	\begin{aligned}
		L_{\sigma, \tau} & = \begin{cases}\sum_{v \in \sigma}
			\omega(v),& \text { if } \sigma=\tau, \\
			(-1)^{\varepsilon(\sigma, \tau)} \omega( \tau \setminus \sigma), & \text { if }|\sigma \cap \tau|=k, \\
			0, & \text { otherwise, }\end{cases}
	\end{aligned}\]
	for all $\sigma, \tau \in X(k)$. Applying Lemma \ref{lem-4} again, we have
	\[	L_k^{\omega\operatorname{down}}(X)=L.\]
	For any $1\le m\le f_k(X)$, it follows from Corollary \ref{cor-1} and Theorem \ref{thm-10} that
	\[	\lambda_{m}^{\downarrow}(L)\geq \lambda_{\binom{n}{k+1}-f_k(X)+m}^{\downarrow}\left(L_0^{\omega,\operatorname{down}}(X)^{[k+1]}\right)=S_{k+1, \binom{n}{k+1}-f_k(X)+m}^{\downarrow}\left(L_0^{w,\operatorname{down}}(X)\right).\]
	Since $\mathbf{s}_0^{\omega \operatorname{down}}\left(X\right) =\left\{\sum_{v\in V}\omega(v), 0^{(n-1)}\right\}$,
	we have \[S_{k+1}\left(L_0^{\omega,\operatorname{down}}(X)\right)=\left\{\left(\sum_{v\in V}\omega(v)\right)^{\left( \binom{n-1}{k}\right)},0^{\left(\binom{n-1}{k+1}\right)}\right\}.\]
	Therefore, if $\binom{n}{k+1}-f_k(X)+m\le \binom{n-1}{k}$, i.e., $m\le f_k(X)-\binom{n-1}{k+1}$, we have
\[\lambda_{m}^{\downarrow}\left(L_k^{\omega\operatorname{down}}\left(X\right)\right)=\lambda_{m}^{\downarrow}(L)\geq S_{k+1, \binom{n}{k+1}-f_k(X)+m}^{\downarrow}\left(L_0^{w,\operatorname{down}}(X)\right)=\sum_{v\in V}\omega(v).\]
Combining this with \eqref{eq-11}, we conclude that if $m\le f_k(X)-\binom{n-1}{k+1}$, then $\lambda_{m}^{\downarrow}\left(L_k^{\omega\operatorname{down}}\left(X\right)\right)= \sum_{v\in V}\omega(v)$. Similarly, by replacing $k$ with $k+1$ , if  $m\le f_{k+1}(X)-\binom{n-1}{k+2}$, we have  \[\lambda_{m}^{\downarrow}\left(L_k^{\omega\operatorname{up}}\left(X\right)\right)=\lambda_{m}^{\downarrow}\left(L_{k+1}^{\omega\operatorname{down}}\left(X\right)\right)\ge\sum_{v\in V}\omega(v),\]
	 and thus $\lambda_{m}^{\downarrow}(L_k^{\omega\operatorname{up}}(X))= \sum_{v\in V}\omega(v)$. Therefore,  if $f_k(X)>\binom{n-1}{k+1}$ or $f_{k+1}(X)>\binom{n-1}{k+2}$, then $\sum_{v\in V}\omega(v)$ is an eigenvalue of $L_k^{\omega}(X)$ with mulitiplicity at least \[\max\left\{f_k(X)+f_{k+1}(X)-\binom{n}{k+2}, f_k(X)-\binom{n-1}{k+1},f_{k+1}(X)-\binom{n-1}{k+2}\right\}.  \]

The proof is completed.
\end{proof}

\begin{remark}
	In their work \cite[Theorem 4.3]{DR2002}, Duval and Reiner established a spectral relation for a simplicial complex $X$ with $n$ vertices. Specifically, they demonstrated that the spectra of the combinatorial $k$-dimensional down-Laplacian $L_k^{\operatorname{down}}(X)$ and the $(n-k-2)$-dimensional down-Laplacian $L_{n-k-2}^{\operatorname{down}}(X_k^*)$ satisfy the identity\[\lambda_{i}^{\downarrow}\left(L_k^{\operatorname{down}}\left(X\right)\right)+\lambda_{f_k(X)+1-i}^{\downarrow}\left(L_{n-k-2}^{\operatorname{down}}\left(X_k^*\right)\right)=n.\]
	A direct consequence of this relation is the upper bound
	\[\lambda^{\downarrow}_{1}\left(L_k(X)\right)\le n\]
    on the largest eigenvalue $L_k(X)$. Therefore, Theorems \ref{thm-14} and \ref{thm-18} can be situated as a generalization of these results in two principal directions. Firstly, the equations \eqref{eq-18} and \eqref{eq-19} extend the aforementioned spectral relation and its concomitant eigenvalue bound from the combinatorial Laplacian to the vertex-weighted Laplacian framework. Secondly, Theorem \ref{thm-18} also provides a characterization of the extremal cases in which the upper bound is attained, alongside a lower bound on the multiplicity of the corresponding eigenvalue.
\end{remark}
\begin{example}
	Let $\Delta_{n-1}$ be the complete simplicial complex on vertex set $V$ of size $n$. Then $\operatorname{dim}(\Delta_{n-1})=n-1$.
	Noticing that $f_k(\Delta_{n-1})=\binom{n}{k+1}>\binom{n-1}{k+1}$ for each $0\le k\le n-1$ and $f_{k+1}(\Delta_{n-1})=\binom{n}{k+2}>\binom{n-1}{k+2}$ for $-1\le k\le n-2$,
it follows from Theorem \ref{thm-18} that $\sum_{v\in V}\omega(v)$ is an eigenvalue of $L_k^{\omega}(\Delta_{n-1})$ with mulitiplicity at least 
	\[\max\left\{f_k(\Delta_{n-1})+f_{k+1}(\Delta_{n-1})-\binom{n}{k+2}, f_k(\Delta_{n-1})-\binom{n-1}{k+1},f_{k+1}(\Delta_{n-1})-\binom{n-1}{k+2}\right\},\]
which equals $1$ if $k\in\{1,n-1\}$, and $\binom{n}{k+1}$ if $0\le k\le n-2$.
By Corollary \ref{cor-2}, the multiplicity of $\sum_{v \in V} \omega(v)$ in $\mathbf{s}_{k }^{\omega}\left(\Delta_{n-1}\right)$ attains precisely the lower bound provided by Theorem \ref{thm-18}.
    \end{example}
\begin{figure}[htbp]
    \centering
    \includegraphics[width=10cm]{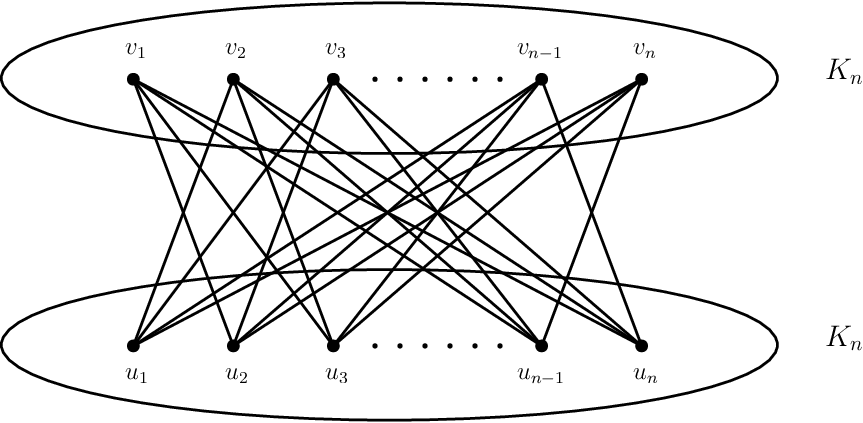}
    \caption{The cocktail party graph $CP(n)$}
    \label{fig-1}
\end{figure}
\begin{example}
The cocktail party graph $CP(n)$ is the graph obtained from the complete graph $K_{2n}$ by removing a perfect matching (see Fig.\ref{fig-1}). Let $X$ be the clique complex of $CP(n)$. Noticing that \[f_0(X)=2n>\binom{2n-1}{1},\ f_1(X)=\binom{2n}{2}-n>\binom{2n-1}{2} \text{ and } f_k(X)<\binom{2n-1}{k+1}\] for any $k>1$, it follows from Theorem \ref{thm-18} that $\sum_{v\in V}\omega(v)$ is an eigenvalue of $L_{-1}^{\omega}(X)$ with mulitiplicity at least $f_0(X)-\binom{2n-1}{1}=1$, that $\sum_{v\in V}\omega(v)$ is an eigenvalue of $L_0^{\omega}(X)$ with mulitiplicity at least \[f_1(X)-\binom{2n-1}{2}+f_0(X)-\binom{2n-1}{1}=n,\]
and that $\sum_{v\in V}\omega(v)$ is an eigenvalue of $L_1^{\omega}(X)$ with mulitiplicity at least \[f_1(X)-\binom{2n-1}{2}=n-1.\]
Moreover, a straightforward computation shows that the eigenvalue 
$\sum_{v\in V}\omega(v)$ of $L_{0}^{\omega}(X)$ has multiplicity exactly equal to the lower bound $n$.
\end{example}

\subsection{Vertex-wighted Laplacian spectra of $X$ and $X^c$}

We then study the relationship between the spectra of the vertex-weighted $k$-dimensional up-Laplacian operators of $X$ and those of the complexes $X_{k+1}^{c}$. Let $\Delta_n$ be the complete simplicial complex on a $(n+1)$-set $V$. 
\begin{lemma}\label{lem-16}
	Let $X$ be a simplicial complex on vertex set $V$ with $|V|=n$, and let $\omega\colon X \rightarrow \mathbb{R}_{>0}$ be a vertex weight function of $X$. For any $-1\le k\le \operatorname{dim}(X)$, we have
	\begin{equation}\label{eq-12}
			L_k^{\omega \operatorname{up}}\left(X\right)+L_k^{\omega \operatorname{up}}\left(X^c_{k+1}\right)=L_k^{\omega \operatorname{up}}\left(\Delta^{(k+1)}_{n-1}\right),
	\end{equation}
	 and all three operators pairwise commute.
\end{lemma}
\begin{proof}
The rows and columns of all these three matrices are labeled by $\binom{V}{k+1}$. We consider the $(\sigma,\tau)$-entry of both sides of \eqref{eq-12} for $\sigma,\tau\in \binom{V}{k+1}$. If $\sigma=\tau$, from Lemma \ref{lem-14}, we have
\[L_k^{\omega \operatorname{up}}\left(X\right)_{\sigma,\sigma}=\sum_{u\in \operatorname{lk}_X(\sigma)}\omega(u)\]
and 
\[L_k^{\omega \operatorname{up}}\left(X^c_{k+1}\right)_{\sigma,\sigma}=\sum_{u\in \operatorname{lk}_{X^c_{k+1}}(\sigma)}\omega(u).\]
Since $\operatorname{lk}_X(\sigma)=\{u\not\in \sigma\colon \sigma\cup\{u\}\in X(k+1)\}$ and $\operatorname{lk}_{X^c_{k+1}}(\sigma)=\{u\not\in \sigma\colon \sigma\cup\{u\}\in \binom{V}{k+2}\setminus X(k+1)\}$, we get 
\[\left(L_k^{\omega \operatorname{up}}\left(X^c_{k+1}\right)+L_k^{\omega \operatorname{up}}\left(X^c_{k+1}\right)\right)_{\sigma,\sigma}=\sum_{u\not\in\sigma}\omega(u)=L_k^{\omega \operatorname{up}}\left(\Delta^{(k+1)}_{n-1}\right)_{\sigma,\sigma}.\]
If $\sigma,\tau\in\binom{V}{k+1}$ with $|\sigma\cap\tau|=k$, then either $\sigma\cup\tau\in X$ or $\sigma\cup\tau\in X^c_{k+1}$. Therefore, from Lemma \ref{lem-14}, we have
\[\left(L_k^{\omega \operatorname{up}}\left(X^c_{k+1}\right)+L_k^{\omega \operatorname{up}}\left(X^c_{k+1}\right)\right)_{\sigma,\tau}=-(-1)^{\varepsilon(\sigma,\tau)}\omega(\tau\setminus \sigma)=L_k^{\omega \operatorname{up}}\left(\Delta^{(k+1)}_{n-1}\right)_{\sigma,\tau}.\]
If $\sigma,\tau\in\binom{V}{k+1}$ with $|\sigma\cap\tau|<k$, then, from Lemma \ref{lem-14}, we have
\[\left(L_k^{\omega \operatorname{up}}\left(X^c_{k+1}\right)+L_k^{\omega \operatorname{up}}\left(X^c_{k+1}\right)\right)_{\sigma,\tau}=0=L_k^{\omega \operatorname{up}}\left(\Delta^{(k+1)}_{n-1}\right)_{\sigma,\tau}.\]
Hence, \eqref{eq-12} holds.

Next, we show that these operators commute pairwise.
In fact, it is sufficient to verify that $L_k^{\omega \operatorname{up}}(X)$ commutes with $L_k^{\omega \operatorname{up}}\left(\Delta^{(k+1)}_{n-1}\right)$. By Lemma \ref{lem-15}, we have 	
\begin{equation}\label{eq-13}
	\mathbf{s}_{k }^{\omega \operatorname{up}}\left(\Delta^{(k+1)}_{n-1}\right)=\left\{\left(\sum_{v\in V}\omega(v)\right)^{\left(\binom{n-1}{k+1}\right)},0^{\left(\binom{n-1}{k}\right)} \right\}.
\end{equation}
Consequently, the space $C^{k}\left(\Delta^{(k+1)}_{n-1}, \mathbb{R}\right)$ decomposes as a direct sum  \[C^{k}(\Delta^{(k+1)}_{n-1}, \mathbb{R})=V_0\oplus V_1,\]
where $V_0$ and $V_1$ denote the eigenspaces corresponding to the eigenvalues $0$ and $\sum_{v\in V}\omega(v)$, respectively.
It follows from the self-adjointness of the three operators above and \eqref{eq-12} that
\begin{equation}\label{eq-14}
	V_0=\operatorname{Ker}\left(L_k^{\omega \operatorname{up}}\left(\Delta^{(k+1)}_{n-1}\right)\right)=\operatorname{Ker}\left(L_k^{\omega \operatorname{up}}(X)\right)\cap \operatorname{Ker}\left(L_k^{\omega \operatorname{up}}(X^c_{k+1})\right),
\end{equation}
and thus $L_k^{\omega \operatorname{up}}(X)(V_0)=\boldsymbol{0}$.
For any $\boldsymbol{x}=\boldsymbol{x}_0+\boldsymbol{x}_1\in C^{k}\left(\Delta^{(k+1)}_{n-1}, \mathbb{R}\right)$ with $\boldsymbol{x}_0\in V_0$ and $\boldsymbol{x}_1\in V_1$, we have
\[L_k^{\omega \operatorname{up}}(X)L_k^{\omega \operatorname{up}}\left(\Delta^{(k+1)}_{n-1}\right)(\boldsymbol{x})=\left(\sum_{v\in V}\omega(v)\right)L_k^{\omega \operatorname{up}}(X)(\boldsymbol{x}_1),\]
and
\[L_k^{\omega \operatorname{up}}\left(\Delta^{(k+1)}_{n-1}\right)L_k^{\omega \operatorname{up}}(X)(\boldsymbol{x})=L_k^{\omega \operatorname{up}}\left(\Delta^{(k+1)}_{n-1}\right)	L_k^{\omega \operatorname{up}}(X)(\boldsymbol{x}_1).\]
Note that for any $\boldsymbol{y} \in V_0$, we have \[\left\langle L_k^{\omega \operatorname{up}}(X)(\boldsymbol{x}_1),\boldsymbol{y}\right\rangle=\left\langle \boldsymbol{x}_1,L_k^{\omega \operatorname{up}}(X)(\boldsymbol{y})\right\rangle =0.\]
This implies $L_k^{\omega \operatorname{up}}(X)(\boldsymbol{x}_1)\in V_0^{\perp}=V_1$. Therefore, 
\[L_k^{\omega \operatorname{up}}\left(\Delta^{(k+1)}_{n-1}\right)L_k^{\omega \operatorname{up}}(X)(\boldsymbol{x})=\left(\sum_{v\in V}\omega(v)\right)L_k^{\omega \operatorname{up}}(X)(\boldsymbol{x}_1)=L_k^{\omega \operatorname{up}}(X)L_k^{\omega \operatorname{up}}\left(\Delta^{(k+1)}_{n-1}\right)(\boldsymbol{x}),\] and hence the two operators commute.
\end{proof}

Applying Lemma \ref{lem-16}, we get the relation between the spectra of $\left(L_k^{\omega\operatorname{up}}\left(X\right)\right)$ and those of $\left(L_k^{\omega\operatorname{up}}\left(X^c_{k+1}\right)\right)$.
\begin{theorem}\label{thm-17}
	Let $X$ be a simplicial complex on vertex set $V$ of size $n$, and let $\omega\colon X \rightarrow \mathbb{R}_{>0}$ be a vertex weight function of $X$. For any $-1\le k\le \operatorname{dim}(X)$,
we have $\lambda_{i}^{\downarrow}\left(L_k^{\omega\operatorname{up}}\left(X\right)\right)=\lambda_{i}^{\downarrow}\left(L_k^{\omega\operatorname{up}}\left(X^c_{k+1}\right)\right)=0$ for $\binom{n-1}{k+1}+1\le i\le \binom{n}{k+1}$, and 
\begin{equation}\label{eq-20}
	\lambda_{i}^{\downarrow}\left(L_k^{\omega\operatorname{up}}\left(X\right)\right)+\lambda_{\binom{n-1}{k+1}-i+1}^{\downarrow}\left(L_k^{\omega\operatorname{up}}\left(X^c_{k+1}\right)\right)=\sum_{v\in V}\omega(v)
\end{equation}
for  $1\le i\le \binom{n-1}{k+1}$.
\end{theorem}
\begin{proof}
By \eqref{eq-13}, the multiplicity of $0$ as eigenvalue of operator $L_k^{\omega \operatorname{up}}\left(\Delta^{(k+1)}_{n-1}\right)$ is $\operatorname{m}_{L_k^{\omega \operatorname{up}}\left(\Delta^{(k+1)}_{n-1}\right)}(0)=\binom{n-1}{k}$. 
From \eqref{eq-14}, it follows that the multiplicities of $0$ as an eigenvalue of operators $L_k^{\omega \operatorname{up}}\left(X\right)$ and $L_k^{\omega \operatorname{up}}\left(X^c_{k+1}\right)$ satisfy
\[\operatorname{m}_{L_k^{\omega \operatorname{up}}\left(X\right)}(0)\ge \binom{n-1}{k}\text{ and } \operatorname{m}_{L_k^{\omega \operatorname{up}}\left(X^c_{k+1}\right)}(0)\ge  \binom{n-1}{k},\]
and the eigenfunctions corresponding to such $0$ belong to $V_0$, where $V_0$ and $V_1$ are the space defined in the proof of Lemma \ref{lem-16}.. Therefore, both of them have at most $\binom{n}{k+1}-\binom{n-1}{k}=\binom{n-1}{k+1}$ positive eigenvalues, which implies that
\[\lambda_{i}^{\downarrow}\left(L_k^{\omega\operatorname{up}}\left(X\right)\right)=\lambda_{i}^{\downarrow}\left(L_k^{\omega\operatorname{up}}\left(X^c_{k+1}\right)\right)=0,\]
for $i\ge \binom{n-1}{k+1}+1$.

For $1\le i\le \binom{n-1}{k+1}$, let $\boldsymbol{x}$ be an eigenfunction corresponding to $\lambda_{i}^{\downarrow}\left(L_k^{\omega\operatorname{up}}\left(X\right)\right)$ which is independent to the functions in $V_0$. Clearly, $\boldsymbol{x}$ can be decomposed as $\boldsymbol{x}=\boldsymbol{x}_0+\boldsymbol{x}_1$ with $\boldsymbol{x}_0\in V_0$ and $\boldsymbol{x}_1\in V_1\setminus\{\boldsymbol{0}\}$.  Then we have
\begin{equation}
	\begin{aligned}\label{eq-15}
		L_k^{\omega\operatorname{up}}\left(X^c_{k+1}\right)\left(\boldsymbol{x}_1 \right)&=L_k^{\omega \operatorname{up}}\left(\Delta^{(k+1)}_{n-1}\right)\left(\boldsymbol{x}_1 \right)-L_k^{\omega\operatorname{up}}\left(X\right)\left(\boldsymbol{x}_1 \right)\\&=\left(-\lambda_{i}^{\downarrow}\left(L_k^{\omega\operatorname{up}}\left(X\right)\right)+\sum_{v\in V}\omega(v)\right)\boldsymbol{x}_1-\lambda_{i}^{\downarrow}\left(L_k^{\omega\operatorname{up}}\left(X\right)\right)\boldsymbol{x}_0.
	\end{aligned}
\end{equation}
For convenience, we write $-\lambda_{i}^{\downarrow}\left(L_k^{\omega\operatorname{up}}\left(X\right)\right)+\sum_{v\in V}\omega(v)$ as $\beta_i$. If $\lambda_{i}^{\downarrow}\left(L_k^{\omega\operatorname{up}}\left(X\right)\right)\ne \sum_{v\in V}\omega(v)$, then $\beta_i\ne 0$, and we have
\[	\begin{array}{lll}
	&&L_k^{\omega\operatorname{up}}\left(X^c_{k+1}\right)\left(\boldsymbol{x}_1 -\frac{\lambda_{i}^{\downarrow}\left(L_k^{\omega\operatorname{up}}\left(X\right)\right)}{\beta_i}\boldsymbol{x}_0\right)=L_k^{\omega\operatorname{up}}\left(X^c_{k+1}\right)\left(\boldsymbol{x}_1\right)\\[2mm]
	&=&\beta_i\boldsymbol{x}_1-\lambda_{i}^{\downarrow}\left(L_k^{\omega\operatorname{up}}\left(X\right)\right)\boldsymbol{x}_0=\beta_i\left(\boldsymbol{x}_1 -\frac{\lambda_{i}^{\downarrow}\left(L_k^{\omega\operatorname{up}}\left(X\right)\right)}{\beta_i}\boldsymbol{x}_0\right).
\end{array}\] 
If $\lambda_{i}^{\downarrow}\left(L_k^{\omega\operatorname{up}}\left(X\right)\right)= \sum_{v\in [n]}\omega(v)$, then $\beta_i=0$.
From \eqref{eq-15}, we have $L_k^{\omega\operatorname{up}}\left(X^c_{k+1}\right)\left(\boldsymbol{x}_1 \right)=-\lambda_{i}^{\downarrow}\left(L_k^{\omega\operatorname{up}}\left(X\right)\right)\boldsymbol{x}_0\in V_0$. However, by noticing that $L_k^{\omega\operatorname{up}}\left(X^c_{k+1}\right)\left(\boldsymbol{x}_1 \right)\in V_0^{\perp}=V_1$, we have
$\boldsymbol{x}_0=\boldsymbol{0}$, and thus $L_k^{\omega\operatorname{up}}\left(X^c_{k+1}\right)\left(\boldsymbol{x}_1 \right)=\boldsymbol{0}=\beta_i \boldsymbol{x}_1$. Therefore,  for $1\le i\le \binom{n-1}{k+1}$, we always get that $\beta_i=-\lambda_{i}^{\downarrow}\left(L_k^{\omega\operatorname{up}}\left(X\right)\right)+\sum_{v\in V}\omega(v)$ is an eigenvalue of $L_k^{\omega\operatorname{up}}\left(X^c_{k+1}\right)$. Moreover, the corresponding eigenfunction is linear independent to the functions in $V_0$. Hence, $\beta_i=\lambda_{\binom{n-1}{k+1}-i+1}^{\downarrow}\left(L_k^{\omega\operatorname{up}}\left(X^c_{k+1}\right)\right)$, and \eqref{eq-20} holds.

The proof is completed.
\end{proof}

\begin{remark}
Lemma \ref{lem-16} and Theorem \ref{thm-17} generalize the corresponding results \cite[Proposition 4.4, Theorem 4.5]{DR2002} by Duval and Reiner from combinatorial Laplacian operators to vertex-weighted Laplacian operators.
\end{remark}

\subsection{Vertex-weighted Laplacian of Alexander dual of $X$}
Given a simplicial complex $X$, having discussed the relationship between the spectra of the vertex-weighted Laplacian operators of $X$ and those of the complexes $X_k^{*}$ and $X_{k+1}^{c}$, we now turn to the study of the corresponding relationships between $X$ and its \textit{canonical Alexander dual}.
\begin{definition}[\cite{MS2005}]
	For a simplicial complex $X$ on vertex set $V$, the canonical Alexander dual $X^{\vee} $ is defined by
	\[X^{\vee}\colon=\left\{V\setminus\sigma\colon \sigma\notin X\right\}.\]
\end{definition}

Miller and Sturmfels \cite{MS2005} established an isomorphism between the homology of the Alexander dual of a simplicial complex and the cohomology of the complex itself:
\begin{theorem}[{\cite[Alexander duality]{MS2005}}]\label{thm-16}
	Let $X$ be a simplicial complex on $n$ vertices, and $\mathbb{F}$ be a field. For any $0\le k\le n-1$,
\[\widetilde{H}_{k-1}\left(X^{\vee} ; \mathbb{F}\right) \cong \widetilde{H}^{n-2-k}(X ; \mathbb{F}).\]
\end{theorem}
\begin{remark}\label{rem-1}
For any $-1\le k\le \operatorname{dim}(X)$,
\[\begin{aligned}
	\left( X^{\vee}\right)^{(n-k-2)}&=\left\{V\setminus\sigma\colon \sigma\notin X\text{ with }|\sigma|\ge k+1\right\}\\&=\left\{\tau\colon \exists\ \sigma\notin X \text{ such that } |\sigma|= k+1\text{ and }\tau\subseteq V\setminus\sigma\right\}\\&=\left\{\tau\colon \exists\ \sigma\in X_k^{c}(k)\text{ such that }\tau\subseteq V\setminus\sigma\right\}\\&=\left(X_k^{c}\right)_k^*,
\end{aligned}\]
and 
\[\begin{aligned}
	\left( X^{\vee}\right)^{(n-k-2)}&=\left\{V\setminus\sigma\colon \sigma\notin X \text{ with } |\sigma|\ge k+1\right\}\\&=\left\{\tau\colon \exists\ \sigma\notin X\text{ such that }|\sigma|= k+1 \text{ and }\tau\subseteq V\setminus\sigma\right\}\\&=\left\{\tau\colon \exists\ V\setminus\sigma\in \tbinom{V}{n-k-1}\setminus X_k^*(n-k-2)\text{ such that }\tau\subseteq V\setminus\sigma\right\}\\&=\left(X_k^{*}\right)_{n-k-2}^c.
\end{aligned}\]
\end{remark}
In what follows, we consider the relation between $\mathbf{s}^{\omega}_k\left(X\right)$ and $\mathbf{s}^{\omega}_{n-k-3}\left(X^{\vee}\right)$ for $-1\le k\le n-2$. Theorem \ref{thm-18} indicates $\lambda\le \sum_{v\in V}\omega(v)$ for any $\lambda\in \mathbf{s}^{\omega}_k\left(X\right)$.
\begin{theorem}\label{thm-19}
	Let $X$ be a simplicial complex on vertex set $V$ of size $n$, and let $\omega\colon X \rightarrow \mathbb{R}_{>0}$ be a vertex weight function of $X$. For any $-1\le k\le n-2$ and $\lambda\in \mathbf{s}^{\omega}_k\left(X\right)$, if $\lambda<\sum_{v\in V}\omega(v)$, then $\operatorname{m}_{L_k^{\omega}(X)}(\lambda)=\operatorname{m}_{L_{n-k-3}^{\omega}(X^{\vee})}(\lambda)$; if $\lambda=\sum_{v\in V}\omega(v)$, then 
	\[\operatorname{m}_{L_k^{\omega}(X)}(\lambda)-\operatorname{m}_{L_{n-k-3}^{\omega}(X^{\vee})}(\lambda)=f_k\left(X\right)+f_{k+1}\left(X\right)-\binom{n}{k+2}.\]
\end{theorem}
\begin{proof}
By Theorem \ref{thm-16}, for any $-1\le k\le n-2$, we have 
\[\widetilde{H}^{k}\left(X; \mathbb{R}\right) \cong\widetilde{H}_{k}\left(X; \mathbb{R}\right) \cong \widetilde{H}^{n-k-3}(X^{\vee}  ; \mathbb{R}).\]
This implies that the multiplicity of the eigenvalue zero in $\mathbf{s}^{\omega}_k\left(X\right)$ coincides with that in $\mathbf{s}^{\omega}_{n-k-3}\left(X^{\vee}\right)$. For every non-zero eigenvalue $\lambda<\sum_{v\in V}\omega(v)$, it follows from Theorems \ref{thm-14}, \ref{thm-17}, and Remark \ref{rem-1} that
\begin{equation}\label{eq-16}
	\begin{aligned}
		\operatorname{m}_{L_k^{\omega \operatorname{up}}\left(X\right)}(\lambda)&=\operatorname{m}_{L_k^{\omega \operatorname{up}}\left(X^c_{k+1}\right)}\left(\sum_{v \in V} \omega(v)-\lambda\right)=\operatorname{m}_{L_{k+1}^{\omega \operatorname{down}}\left(X^c_{k+1}\right)}\left(\sum_{v \in V} \omega(v)-\lambda\right)\\&=\operatorname{m}_{L_{n-k-3}^{\omega \operatorname{down}}\left(\left(X^c_{k+1}\right)^*_{k+1}\right)}\left(\lambda\right)=\operatorname{m}_{L_{n-k-3}^{\omega \operatorname{down}}\left(	\left( X^{\vee}\right)^{(n-k-3)}\right)}\left(\lambda\right)=\operatorname{m}_{L_{n-k-3}^{\omega \operatorname{down}}\left(X^{\vee}\right)}\left(\lambda\right),
	\end{aligned}
\end{equation}
and
\begin{equation}\label{eq-17}
	\begin{aligned}
		\operatorname{m}_{L_k^{\omega \operatorname{down}}\left(X\right)}(\lambda)&=\operatorname{m}_{L_{n-k-2}^{\omega \operatorname{down}}\left(X^*_{k}\right)}\left(\sum_{v \in V} \omega(v)-\lambda\right)=\operatorname{m}_{L_{n-k-3}^{\omega \operatorname{up}}\left(X^*_{k}\right)}\left(\sum_{v \in V} \omega(v)-\lambda\right)\\&=\operatorname{m}_{L_{n-k-3}^{\omega \operatorname{up}}\left(\left(X^*_{k}\right)^c_{n-k-2}\right)}\left(\lambda\right)=\operatorname{m}_{L_{n-k-3}^{\omega \operatorname{up}}\left(\left( X^{\vee}\right)^{(n-k-2)}\right)}\left(\lambda\right)=\operatorname{m}_{L_{n-k-3}^{\omega \operatorname{up}}\left(X^{\vee}\right)}\left(\lambda\right).
	\end{aligned}
\end{equation}
Recall that the nonzero spectrum of $L^{\omega}_k\left(X\right)$ is exactly the union (as multisets) of the nonzero spectra of $L_k^{\omega \operatorname{up}}\left(X\right)$ and $L_k^{\omega\operatorname{down}}\left(X\right)$ for any simplicial complex $X$. Combinning \eqref{eq-16} with \eqref{eq-17}, we have
\[\operatorname{m}_{L_k^{\omega}\left(X\right)}(\lambda)=\operatorname{m}_{L_{n-k-3}^{\omega}\left(X^{\vee}\right)}\left(\lambda\right).\]
Hence, (i) holds.
From (i), we further deduce that
\[\begin{aligned}
	\operatorname{m}_{L_k^{\omega}\left(X\right)}\left(\sum_{v \in V} \omega(v)\right)-\operatorname{m}_{L_{n-k-3}^{\omega}\left(X^{\vee}\right)}\left(\sum_{v \in V} \omega(v)\right)&=f_k\left(X\right)-f_{n-k-3}\left(X^{V}\right)\\&=f_k\left(X\right)+f_{k+1}\left(X\right)-\binom{n}{k+2}.
\end{aligned}\]
This completes the proof.
\end{proof}
\begin{remark}
Theorem \ref{thm-19} generalizes a corresponding result by Duval and Reiner \cite[Corollary 4.7]{DR2002} from combinatorial Laplacian to vertex-weighted Laplacian.
\end{remark}

\section{ The spectral gap of vertex-weighted Laplacian operators of simplicial complexes}

Recall that the $k$-th spectral gap (resp. weighted spectral gap) of a simplicial complex $X$ is defined as the smallest eigenvalue $\lambda_1^{\uparrow}(L_k(X))$ (resp. $\lambda_1^{\uparrow}(L_k^{\omega}(X))$). In their work \cite{L2020}, Lew obtained a lower bound on the spectral gap of the combinatorial Laplacian operators of simplicial complexes. Define $\delta_k(X)\colon=\min _{\sigma \in X(k)} \operatorname{deg}_X^{+}(\sigma)$ the minimal upper degree of $k$-simplex in $X$.
\begin{theorem}[\cite{L2020}]\label{thm-01}
	Let $X$ be a simplicial complex on vertex set $V$ of size $n$, with $h(X)=d$.
	Then for $k \geq-1$,
\[\lambda_1^{\uparrow}(L_k(X)) \geq(d+1)\left(\delta_k(X)+k+1\right)-d n.\]
\end{theorem}

In this section, we first derive a lower bound for the spectral gap of the vertex-weighted $k$-dimensional Laplacian operator on simplicial complexes. We then show that this bound is tight by constructing a simplicial complex that attains it. Finally, we give a full characterization of all complexes that achieve this lower bound.

\subsection{The lower bound for the spectral gap}
\begin{theorem}\label{thm-1}
	Let $X$ be a simplicial complex on vertex set $V$ of size $n$, with $h(X)=d$, and let $\omega\colon X \rightarrow \mathbb{R}_{>0}$ be a vertex weight function of $X$. Then for all $-1\le k\le \operatorname{dim}(X)$, we have	
\[\lambda_1^{\uparrow}(L_k^{\omega}(X))\ge (d+1)m_k-d\sum_{v\in V}\omega(v), \]
	where $m_k=\min_{\sigma\in X(k)}\sum_{v \in \sigma\cup \operatorname{lk}_X(\sigma)} \omega(v)$.
\end{theorem}
\begin{remark}
If the weight function $\omega\equiv1$, then the vertex-weighted Laplacian reduces to the combinatorial Laplacian and \[(d+1)m_k-d\sum_{v\in V}\omega(v)=(d+1)\left(\delta_k+k+1\right)-d n\]
Therefore, Theorem \ref{thm-1} recovers Theorem \ref{thm-01}.
\end{remark}

Before presenting the proof of Theorem \ref{thm-1}, we introduce some concepts and notations. Let $X$ be a simplicial complex on vertex set $V$. For any subset $U\subseteq V$, denote by
\[
X[U]\colon= \{\sigma \in X \colon \sigma \subseteq U\}
\]
the \textit{induced subcomplex} of $X$ on $U$. For any $k\ge 0$, $\sigma\in X(k)$ and $v\in V\setminus \sigma$, define
\[N_{\sigma}(v)\colon=\{\eta\in \sigma(k-1)\colon v\in \operatorname{lk}_X(\eta)\}\text{ and }M_{\sigma}(v)\colon=\{\sigma\setminus\eta\colon \eta\in N_{\sigma}(v)\}.\]
Clearly, $|N_{\sigma}(v)|=|M_{\sigma}(v)|$.

The proof of Theorem \ref{thm-1} relies on two key lemmas. The first lemma, established by Zhan, Huang and Lin \cite{ZHL2026}, is stated as follows.
\begin{lemma}[\cite{ZHL2026}]\label{lem-17}
	Let $X$ be a simplicial complex on vertex set $V$, with $h(X)=d$. Let $\sigma \in X(k)$ and $v \in V \backslash \sigma$. If $v \notin \operatorname{lk}_X(\sigma)$, then
	\[	\left|N_\sigma(v)\right|=\left|M_\sigma(v)\right| \leq d\]
	Furthermore, if the equality holds, then
	\[	X\left[\{v\} \cup M_\sigma(v)\right] \cong \Delta_d^{(d-1)} \text { and } X[\{v\} \cup \sigma] \cong \Delta_d^{(d-1)} * \Delta_{k-d}.\]
\end{lemma}

The second lemma can be viewed as a parallel result to \cite[Lemma 1.4]{L2020} in the context of weighted simplicial complexes.
\begin{lemma}\label{lem-1}
	Let $X$ be a simplicial complex on vertex set $V$ of size $n$, with $h(X)=d$,  and let $\omega\colon X \rightarrow \mathbb{R}_{>0}$ be a vertex weight function of $X$. For all $k\ge 0$ and $\sigma\in X(k)$, we have
	\[\sum_{\eta\in \sigma(k-1)}\sum_{v\in \operatorname{lk}_X(\eta)}\omega(v)\le d\sum_{v\in V}\omega(v)-(d-1) \sum_{v\in \sigma}\omega(v)+(k+1-d)\sum_{v\in \operatorname{lk}_X(\sigma)}\omega(v).\]
\end{lemma}
\begin{proof}By directed calculations, we have
	\[\begin{array}{lll}
		&&\sum_{\eta\in \sigma(k-1)}\sum_{v\in \operatorname{lk}_X(\eta)}\omega(v)=\sum_{v\in V}\left(\omega(v)\sum_{\eta\in \sigma(k-1)\colon \{v\}\cup \eta\in X}1\right)
		\\[2mm]
&=&\sum_{v\in \sigma}\left(\omega(v)\sum_{\eta\in \sigma(k-1)\colon \{v\}\cup \eta\in X}1\right)+\sum_{v\in \operatorname{lk}_X(\sigma)}\left(\omega(v)\sum_{\eta\in \sigma(k-1)\colon \{v\}\cup \eta\in X}1\right)\\[2mm]
&&+\sum_{v\in V\setminus(\sigma\cup \operatorname{lk}_X(\sigma))}\left(\omega(v)\sum_{\eta\in \sigma(k-1)\colon \{v\}\cup \eta\in X}1\right)\\[2mm]
&=&\sum_{v\in \sigma}\omega(v)+(k+1)\sum_{v\in \operatorname{lk}_X(\sigma)}\omega(v)+\sum_{v\in V\setminus(\sigma\cup \operatorname{lk}_X(\sigma))}\left(\omega(v)\sum_{\eta\in \sigma(k-1)\colon \{v\}\cup \eta\in X}1\right)\\[2mm]
&=&\sum_{v\in \sigma}\omega(v)+(k+1)\sum_{v\in \operatorname{lk}_X(\sigma)}\omega(v)+\sum_{v\in V\setminus(\sigma\cup \operatorname{lk}_X(\sigma))}\left(\omega(v)\left|N_{\sigma}(v)\right|\right).
	\end{array}\]
By Lemma \ref{lem-17}, $\left|N_{\sigma}(v)\right|\le d$. Therefore, we have
		\[\begin{array}{lll}
		&&\sum_{\eta\in \sigma(k-1)}\sum_{v\in \operatorname{lk}_X(\eta)}\omega(v)\\[2mm]
		&\le& \sum_{v\in \sigma}\omega(v)+(k+1)\sum_{v\in \operatorname{lk}_X(\sigma)}\omega(v)+d\sum_{v\in V\setminus(\sigma\cup \operatorname{lk}_X(\sigma))}\omega(v)\\[2mm]&=& \sum_{v\in \sigma}\omega(v)+(k+1)\sum_{v\in \operatorname{lk}_X(\sigma)}\omega(v)+d\left(\sum_{v\in V}\omega(v)-\sum_{v\in\sigma}\omega(v)-\sum_{v\in \operatorname{lk}_X(\sigma)}\omega(v)\right)\\[2mm]&=&d\sum_{v\in V}\omega(v)-(d-1) \sum_{v\in \sigma}\omega(v)+(k+1-d)\sum_{v\in \operatorname{lk}_X(\sigma)}\omega(v).
	\end{array}\]
    This completes the proof.
\end{proof}
Next, we introduce the necessary definitions and notations that will be used in the proof of the Theorem \ref{thm-1}.
Let \[E_k\left(X\right)\colon=\left\{{\left\{\sigma, \tau\right\}}\colon \sigma,\tau \in X(k),\sigma\sim \tau\right\},\] and let $C(E_k\left(X\right), \mathbb{R})$ be the real vector space formally generated by all elements of $E_k\left(X\right)$ modulo the relation $\left\{\sigma, \tau\right\}+\left(-\left\{\sigma, \tau\right\}\right)=0$. The vector space $C^*(E_k\left(X\right), \mathbb{R})$ is defined as dual space of $C(E_k\left(X\right), \mathbb{R})$, with basis $\left\{e_{\left\{\sigma,\tau\right\}}\colon \left\{\sigma,\tau\right\}\in E_k\left(X\right) \right\}$ such that
\[e_{\left\{\sigma,\tau\right\}}\left(\left\{\eta,\gamma\right\}\right)= \begin{cases}1 & \text { if }\left\{\eta,\gamma\right\}=\left\{\sigma,\tau\right\}, \\ 0 & \text { otherwise },\end{cases}\]
for any $\left\{\sigma,\tau\right\}\in E_k\left(X\right)$. Let $\omega\colon X\to \mathbb{R}$ be a positive vertex weight function on $X$. Define an inner product $\langle \cdot, \cdot \rangle$ on $C^*(E_k(X),\mathbb{R})$ induced by $\omega$ by declaring the basis $\{ e_{\{\sigma,\tau\}} : \{\sigma,\tau\} \in E_k(X)\}$ orthogonal and setting
\[	\left\langle e_{\left\{\sigma,\tau\right\}}, e_{\left\{\eta,\gamma\right\}}\right\rangle= \begin{cases}\prod_{v\in \sigma\cup \tau}\omega(v) & \text { if } \left\{\sigma,\tau\right\}=\left\{\eta,\gamma\right\}, \\ 0 & \text { otherwise },\end{cases}\]
and then extending bilinearly to all of $C^*(E_k(X),\mathbb{R})$.

\begin{proof}[\rm\textbf{Proof of Theorem \ref{thm-1}}]
	For $k=-1$, it is easy to verify that 
	\[m_{-1}(X)=\lambda_1^{\uparrow}(L_{-1}^{\omega}(X))=\sum _{v\in V}\omega(v),\]
	and the result holds.  
For $k\ge 0$, let $\rho_k\colon C^k(X, \mathbb{R})\rightarrow C^*(E_k(X), \mathbb{R})$ be the $\mathbb{R}$-linear operator determined on basis elements by
\[
\rho_k\left(e_{\sigma}\right)=\sum_{\tau\in X(k)\colon \tau\sim \sigma}\operatorname{sgn}\left(\sigma\cap \tau,\sigma\right)e_{\left\{\sigma, \tau\right\}}
\]
for any $\sigma\in X(k)$. 
With respect to the above inner product, its adjoint $\rho^{\omega *}_k\colon C^*(E_k(X), \mathbb{R})\rightarrow C^k(X, \mathbb{R})$ is given by
\[
\rho^{\omega *}_k\left(e_{\left\{\sigma, \tau\right\}}\right)=\sum_{\eta\in \left\{\sigma, \tau\right\}}\frac{\prod_{v\in\sigma\cup\tau}\omega(v)}{\omega(\eta)}\operatorname{sgn}\left(\sigma\cap \tau,\eta\right)e_{\eta}.
\]
Since $\rho^{\omega *}_k\rho_k$ is self-adjoint and positive semidefinite, all of its eigenvalues are real and nonnegative. In what follows, we also use $\rho^{\omega *}_k\rho_k$ to denote its matrix representation with respect to the basis $\{e_{\sigma}\colon \sigma\in X(k)\}$. By noticing that
\[\begin{array}{lll}
&&\rho_k^{\omega *}\rho_k(e_{\tau})=\rho_k^{\omega *}\left(\sum_{\sigma\in X(k)\colon \sigma\sim\tau}\operatorname{sgn}(\sigma\cap\tau,\tau)e_{\sigma,\tau}\right)\\[2mm]
&=&\sum_{\sigma\in X(k)\colon \sigma\sim\tau}\operatorname{sgn}(\sigma\cap\tau,\tau)\rho_k^{\omega *}(e_{\sigma,\tau})\\[2mm]
&=&\sum_{\sigma\in X(k)\colon \sigma\sim\tau}\operatorname{sgn}(\sigma\cap\tau,\tau)\left(\frac{\prod_{u\in \sigma\cup\tau}\omega(u)}{\omega(\tau)}\operatorname{sgn}(\sigma\cap\tau,\tau)e_{\tau}+\frac{\prod_{u\in \sigma\cup\tau}\omega(u)}{\omega(\sigma)}\operatorname{sgn}(\sigma\cap\tau,\sigma)e_{\sigma}\right)\\[2mm]
&=&\left(\sum_{\sigma\in X(k)\colon \sigma\sim\tau}\frac{\prod_{u\in\sigma\cup \tau}\omega(u)}{\omega(\tau)}\right)e_{\tau}+\sum_{\sigma\in X(k)\colon \sigma\sim\tau}(-1)^{\varepsilon(\sigma,\tau)}\frac{\prod_{u\in \sigma\cup\tau}\omega(u)}{\omega(\sigma)}e_{\sigma}\\[2mm]
&=&\left(\sum_{\sigma\in X(k)\colon \sigma\sim\tau}\omega(\sigma\setminus\tau)\right)e_{\tau}+\sum_{\sigma\in X(k)\colon \sigma\sim\tau}(-1)^{\varepsilon(\sigma,\tau)}\omega(\tau\setminus\sigma)e_{\sigma},
\end{array}\]
and that
\[\begin{array}{lll}
&&\sum_{\sigma\in X(k)\colon \sigma\sim \tau}\omega(\sigma\setminus \tau)\\[2mm]
&=&\sum_{\eta\in\tau(k-1)}\sum_{v\in \operatorname{lk}_X(\eta)\setminus\{\tau\cup \operatorname{lk}_{X}(\tau)\}}\omega(v)\\[2mm]
&=&\sum_{\eta\in\tau(k-1)}\sum_{v\in \operatorname{lk}_X(\eta)}\omega(v)-\sum_{v\in\tau}\omega(v)-(k+1)\sum_{v\in \operatorname{lk}_X(\tau)}\omega(v),
\end{array}\]
 we have
\[\left(\rho^{\omega *}_k\rho_k\right)_{\sigma, \tau}= \begin{cases}\sum_{\eta\in\tau(k-1)}\sum_{v\in \operatorname{lk}_X(\eta)}\omega(v)-\sum_{v\in\tau}\omega(v)-(k+1)\sum_{v\in \operatorname{lk}_X(\tau)}\omega(v)& \text { if } \sigma=\tau,\\[4mm] (-1)^{\varepsilon(\sigma, \tau)} \omega(\tau\setminus\sigma) & \text { if } \sigma\sim \tau,\\[4mm] 0 & \text { otherwise, }\end{cases}\]
for every $\sigma, \tau \in X(k)$.
By Lemma \ref{lem-3}, we have
\[L_k^{\omega}(X)=D^{\omega}_k+\rho^{\omega *}_k\rho_k,\]
where $D^{\omega}_k$ is a diagonal matrix with diagonal entries
\begin{equation}\label{eq-1}
	\left(D^{\omega}_k\right)_{\tau,\tau}=2\sum _{v \in \tau} \omega(v)+(k+2)\sum _{v \in \operatorname{lk}_X(\tau)} \omega(v)-\sum _{\eta \in \tau(k-1)}\sum _{v \in \operatorname{lk}_X(\eta)} \omega(v).
\end{equation}
Let $0 \neq \operatorname{x} \in \mathbb{R}^{f_k(X)}$ be a unit eigenvector of $L_k^\omega(X)$ with eigenvalue $\lambda_1^{\uparrow}(L_k^{\omega}(X))$.
Then
\[\begin{array}{lll}
	&&\lambda_1^{\uparrow}(L_k^{\omega}(X))=\operatorname{x}^{\top} L_k^{\omega}(X)\operatorname{x}=\operatorname{x}^{\top} D^{\omega}_k\operatorname{x}+\operatorname{x}^{\top} \rho^{\omega *}_k\rho^{\omega}_k\operatorname{x}\ge \min_{\tau\in X(k)}\left(D^{\omega}_k\right)_{\tau,\tau}\\[2mm]&\ge& \min_{\tau\in X(k)}\left\{2\sum _{v \in \tau} \omega(v)+(k+2)\sum _{v \in \operatorname{lk}_X(\tau)} \omega(v)-\sum _{\eta \in \tau(k-1)}\sum _{v \in \operatorname{lk}_X(\eta)} \omega(v)\right\}.
\end{array}\]
From Lemma \ref{lem-1}, 
we have
\[\begin{array}{lll}
	&&\lambda_1^{\uparrow}(L_k^{\omega}(X))\\[2mm]&\ge& \min_{\tau\in X(k)}\left\{2\sum _{v \in \tau} \omega(v)+(k+2)\sum _{v \in \operatorname{lk}_X(\tau)} \omega(v)-\sum _{\eta \in \tau(k-1)}\sum _{v \in \operatorname{lk}_X(\eta)} \omega(v)\right\}\\[2mm]&\ge& \min_{\tau\in X(k)}\left\{(d+1)\sum_{v \in \operatorname{lk}_X(\tau)} \omega(v)+(d+1) \sum_{v\in \tau}\omega(v)-d\sum_{v\in V}\omega(v)\right\}
	\\[2mm]&=&\min_{\tau\in X(k)}\left\{(d+1)\sum_{v \in \tau\cup \operatorname{lk}_X(\tau)} \omega(v)-d\sum_{v\in V}\omega(v)\right\}
	\\&=&(d+1)m_k-d\sum_{v\in V}\omega(v).
\end{array}\]	
This completes the proof.
\end{proof}

As a consequence of Theorem \ref{thm-1}, we obtain a new sufficient condition for the vanishing of the cohomology groups.
\begin{corollary}
	Let $X$ be a simplicial complex on vertex set $V$ of size $n$, with $h(X)=d$, and let $\omega\colon X \rightarrow \mathbb{R}_{>0}$ be a vertex weight function of $X$. Then $\tilde{H}^k(X ; \mathbb{R})=0$ for all $0\le k\le \operatorname{dim}(X)$ that satisfies
	\[\min_{\sigma\in X(k)}\sum_{v \in \sigma\cup\operatorname{lk}(\sigma)} \omega(v)>\frac{d}{d+1}\sum_{v\in V}\omega(v).\]
\end{corollary}
\begin{proof}
If $k$ satisfies $\min_{\sigma\in X(k)}\sum_{v \in \sigma\cup\operatorname{lk}(\sigma)} \omega(v)>\frac{d}{d+1}\sum_{v\in V}\omega(v)$,
then by Theorem  \ref{thm-1}, we have
\[\lambda_1^{\uparrow}(L_k^{\omega}(X))\ge (d+1)m_k-d\sum_{v\in V}\omega(v)>0.\] It follows that $\tilde{H}^k(X ; \mathbb{R})=0$.
\end{proof}

From the inequality $\sum_{v \in \sigma\cup \operatorname{lk}X(\sigma)} \omega(v) \ge \sum_{v \in \sigma} \omega(v)$, a weaker bound can be derived immediately from Theorem \ref{thm-1}.
\begin{corollary}\label{cor-l-1}
Let $X$ be a simplicial complex on vertex set $V$ of size $n$, with $h(X)=d$, and let $\omega\colon X \rightarrow \mathbb{R}_{>0}$ be a vertex weight function of $X$. Then for all $-1\le k\le \operatorname{dim}(X)$, we have	
\[	\lambda_1^{\uparrow}(L_k^{\omega}(X))\ge (d+1)\min_{\sigma\in X(k)}\sum_{v \in \sigma} \omega(v) -d\sum_{v\in V}\omega(v).\]
\end{corollary}

\subsection{Example attaining the lower bound in Theorem \ref{thm-1}}

Now we construct a simplicial complex whose spectral gap attains the lower bound in Theorem \ref{thm-1}. Let $d\ge 1$, $r\ge 1$  and $t\ge 1$ be integers. For any $k\ge -1$, define \[\vartheta_k\colon=\left\{\begin{array}{cl}
	\left\lfloor\frac{k+1}{d}\right\rfloor & \text { if } \quad-1 \le k \le d t-2,\\
	t & \text { if } \quad d t-1 \le k \le d t+r-1.
\end{array}\right.\]
\begin{prop}\label{prop-1}
	Let $X_i=\Delta_d^{(d-1)}$ for $i\in [t]$ and $X_{t+1}=\Delta_{r-1}$ be simplicial complexes on pairwise disjoint vertex sets $V_i$, each equipped with a vertex weight function $\omega_i \colon X_i \to \mathbb{R}_{>0}$.
	Let $X=X_1 * \cdots * X_{t+1}$ be the join complex, endowed with the vertex weight function $\omega$ given by $\omega(v)=\omega_i(v)$ for $v\in V_i$. Then
	\[\lambda_1^{\uparrow}(L_k^{\omega}(X))=\min\left\{
\sum_{i \in I}\sum_{v\in V_{{i}}}\omega(v)
\colon I\in\binom{[t]}{t-\vartheta_k} 
\right\}+\sum\limits_{v\in V_{t+1}}\omega(v).\] 
\end{prop}
\begin{proof}
	Note that $\operatorname{dim}(X)=d t+r-1$. For $-1 \le k \le d t+r-1$, by Theorem \ref{thm-j}, we have
\begin{equation}\label{eq-03}
\footnotesize\lambda_1^{\uparrow}(L_k^{\omega}(X))=\min \left\{\lambda_1^{\uparrow}\left(L_{i_1}^{\omega_1}(X_1)\right)+\cdots+\lambda_1^{\uparrow}\left(L_{i_t}^{\omega_t}(X_t)\right)+\lambda_1^{\uparrow}\left(L_{i_{t+1}}^{\omega_{t+1}}(X_{t+1})\right)\colon \begin{array}{l}
		-1 \le i_1, \dots, i_t \le d-1, \\
		-1 \le i_{t+1} \le r-1, \\
		i_1+\cdots+i_t+i_{t+1}=k-t
	\end{array}\right\}.
\end{equation}
	By Corollary \ref{cor-2}, for $j\in [t]$, we have
	\[\lambda_1^{\uparrow}\left(L_{i_j}^{\omega_j}(X_j)\right)=\begin{cases}
    \sum_{v\in V_j}\omega_j(v)  & \text { if } -1 \le i_j \le d-2, \\
		0  & \text { if } i_j=d-1,
    \end{cases}\]
	and \[\lambda_1^{\uparrow}\left(L_{i_{t+1}}^{\omega_{t+1}}(X_{t+1})\right)=\lambda_1^{\uparrow}\left(L_{i_{t+1}}^{\omega_{t+1}}(\Delta_{r-1})\right) =\sum_{v\in V_{t+1}}\omega_{t+1}(v) \text{ for all }-1 \le i_{t+1} \le r-1.\] Let $m$ be the maximal integer between $0$ and $t$ such that there exist $-1 \le i_{s_1}, \dots, i_{s_{t-m}} \le d-2$ and $-1 \le i_{t+1} \le r-1$ satisfying
	\begin{equation}\label{eq-02}
		m(d-1)+i_{s_1}+\cdots+i_{s_{t-m}}+i_{t+1}=k-t,
	\end{equation}
	where $s_{\ell}\in [t]$ for each $\ell\in [t-m]$. This implies that
\[\begin{array}{lll}
	m&=&\max\left\{m\in [t]\colon m(d-1)-(t-m+1) \le k-t\right\}\\[2mm]
	&=&\max\left\{m\in [t]\colon m\le \frac{k+1}{d}\right\}\\[2mm]
	&=&\min\left\{t,\left\lfloor\frac{k+1}{d}\right\rfloor\right\}\\[2mm]
	&=& \left\{\begin{array}{cl}
		\left\lfloor\frac{k+1}{d}\right\rfloor & \text { if } \quad-1 \le k \le d t-2\\
		t & \text { if } \quad d t-1 \le k \le d t+r-1
	\end{array}\right.\\[6mm]
	&=&\vartheta_k.
\end{array}\]
	Take ${s_1},\ldots,{s_{t-\vartheta_k}}$ to be the indices satisfying \eqref{eq-02} that minimizing $\sum_{\ell=1}^{t-\vartheta_k}\sum_{v\in V_{{s_{\ell}}}}\omega(v)$. It follows from \eqref{eq-03} that \[ \lambda_1^{\uparrow}(L_k^{\omega}(X))=\sum_{\ell=1}^{t-\vartheta_k}\sum\limits_{v\in V_{s_{\ell}}}\omega(v)+\sum_{v\in V_{t+1}}\omega_{t+1}(v).\]
Therefore,
	\[\lambda_1^{\uparrow}(L_k^{\omega}(X))=\min\left\{
	\sum_{i \in I}\sum_{v\in V_{{i}}}\omega(v)
	\colon I\in\binom{[t]}{t-\vartheta_k} 
	\right\}+\sum\limits_{v\in V_{t+1}}\omega(v).\] 
  This completes the proof.
\end{proof} 
\begin{prop}\label{prop-2}
	Let $X_i=\Delta_d^{(d-1)}$ for $i\in [t]$ and $X_{t+1}=\Delta_{r-1}$ be simplicial complexes on pairwise disjoint vertex sets $V_i$, and let $X=X_1 * \cdots * X_{t+1}$ be the join complex, endowed with the vertex weight function $\omega\colon  \bigcup_{i=1}^{t+1}V_i \to \mathbb{R}_{>0}$
    such that $\omega $ is constant on each $V_i$ for each $i\in [t]$.  Then, for $-1\le k \le dt+r-1$, we have
	\[\lambda_1^{\uparrow}(L_k^{\omega}(X))=(d+1)m_k(X)-d\sum_{v\in \bigcup_{i=1}^{t+1}V_i}\omega(v),\]
		where $m_k(X)=\min_{\sigma\in X(k)}\sum_{v \in \sigma\cup \operatorname{lk}_X(\sigma)} \omega(v)$.
\end{prop}
\begin{proof}
	 Let $V\colon= \bigcup_{i=1}^{t+1}V_i$ denote the vertex set of $X$, and $\omega|_{V_i}\equiv c_i$, for each $i\in  [t]$. 
By Proposition \ref{prop-1}, we have 
	\begin{equation}\label{eq-l}
		\lambda_1^{\uparrow}(L_k^{\omega}(X))=(d+1)\min\left\{
		\sum_{i\in I}c_{i}\colon
		I\in \binom{[t]}{t-\vartheta_k} 
		\right\}+\sum_{v\in V_{t+1}}\omega(v). 
	\end{equation}
For any $\sigma \in X(k)$, notice that a vertex $v \in V \setminus \sigma$ belongs to $\operatorname{lk}_X(\sigma)$ unless there exists $i \in[t]$ such that $v\in V_i$ and $\left|\sigma \cap V_i\right|=d$, in which case $v$ is precisely the unique vertex in $V_i\setminus \sigma$.
This implies that 
\[\operatorname{lk}_X(\sigma)\colon=\left(V\setminus\sigma\right)\setminus\left\{V_i\setminus \sigma\colon i\in [t], \left|\sigma \cap V_i\right|=d\right\}.\]
Therefore,
\[\sum_{v \in \sigma\cup \operatorname{lk}_X(\sigma)} \omega(v)=\sum_{v \in V} \omega(v)-\sum_{v \in \left\{V_i\setminus \sigma\colon i\in [t], \left|\sigma \cap V_i\right|=d\right\}} \omega(v)=\sum_{v \in V} \omega(v)-\sum_{i\in I_{\sigma}}c_i,\] where
	$I_{\sigma}=\left\{i\in [t]\colon \left|\sigma \cap V_i\right|=d\right\}$. Observe that 
$
\max_{\sigma \in X(k)} |I_{\sigma}| = m=\vartheta_k$, where $m$ is the parameter defined in the proof of Proposition \ref{prop-1}.	
	Therefore, we have 
		\[\begin{array}{lll}
			m_k(X)&=&\min_{\sigma\in X(k)}\sum_{v \in \sigma\cup \operatorname{lk}_X(\sigma)} \omega(v)\\[2mm]&=&\sum_{v \in V} \omega(v)-\max_{\sigma \in X(k)}\sum_{i\in I_{\sigma}}c_i \\[2mm]
		&=&\sum_{v \in V} \omega(v)-\max \left\{
		\sum_{i\in I}c_{i}
		\colon 	I\in \binom{[t]}{\vartheta_k} 
		\right\}\\[2mm]
		&=&\sum_{v \in V} \omega(v)-\left(\sum_{i=1}^{t}c_i-\min\left\{
		\sum_{i\in I}c_{i}
		\colon  I\in \binom{[t]}{t-\vartheta_k} 
		\right\}\right)\\[2mm]
		&=& \sum_{v\in V}\omega(v)-\sum_{i=1}^{t}c_i+\min\left\{
		\sum_{i\in I}c_{i}
		\colon 	I\in \binom{[t]}{t-\vartheta_k} 
		\right\}.
	\end{array}\]
	It follows from \eqref{eq-l} that
	\[\begin{array}{lll}
	&&(d+1)m_k(X)-d\sum_{v\in V}\omega(v)\\[2mm]&=&\sum_{v \in V} \omega(v)-(d+1)\sum_{i=1}^{t}c_i+(d+1)\min\left\{
	\sum_{i\in I}c_{i}
	\colon 	I\in \binom{[t]}{t-\vartheta_k} 
	\right\}\\[2mm]
		&=&\sum_{v\in V_{t+1}}\omega(v) +(d+1)\min\left\{
		\sum_{i\in I}c_{i}
		\colon 	I\in \binom{[t]}{t-\vartheta_k} 
		\right\}\\[2mm]
		&=&\lambda_1^{\uparrow}(L_k^{\omega}(X)).
	\end{array}\]

We complete this proof.
\end{proof}
By taking $c_i=1$ for each $i$, from Propositions \ref{prop-1} and \ref{prop-2}, we immediately recover a result of Lew (see \cite[Proposition 1.5]{L2020}).
\begin{corollary}[\cite{L2020}]
Let $X=\left(\Delta_d^{(d-1)}\right)^{* t} * \Delta_{r-1}$, and let $n=(d+1) t+r$ be the number of vertices of $X$. Then
\[\lambda_1^{\uparrow}(L_k(X))= \begin{cases}(d+1)\left(t-\left\lfloor\frac{k+1}{d}\right\rfloor\right)+r & \text { if }-1 \leq k \leq d t-1, \\ r & \text { if } \quad d t \leq k \leq d t+r-1,\end{cases}\]
and
\[\delta_k(X)=\left\{\begin{array}{cl}
	n-(k+1)-\left\lfloor\frac{k+1}{d}\right\rfloor & \text { if }-1 \leq k \leq d t-1, \\
	n-(k+1)-t & \text { if } d t \leq k \leq d t+r-1 .
\end{array}\right.\]
\end{corollary}

\subsection{The complexes attaining the lower bound in Corollary \ref{cor-l-1}}

In \cite{ZHL2026}, Zhan, Huang and Lin characterize the simplicial complexes $X$ for which the spectral gap of the combinatorial $k$-dimensional Laplacian attains the lower bound $(h(X)+1)(k+1)-h(X)|V(X)|$, thereby proving the conjecture posed by Lew (see [\citealp{L2020}, Conjecture 5.1]).
\begin{theorem}[\cite{ZHL2026}]\label{thm-21}
Let $X$ be a simplicial complex on vertex set $V$ of size $n$, with $h(X)=d$, such that $\lambda^{\uparrow}(L_k(X))=(d+1)(k+1)-d n$ for some $k$. Then
	\[X \cong\left(\Delta_d^{(d-1)}\right)^{*(n-k-1)} * \Delta_{(d+1)(k+1)-d n-1}\]
(and in particular, $\operatorname{dim}(X)=k$ ).
\end{theorem}

Proposition \ref{prop-2} provides an example of a weighted simplicial complex for which the spectral gap of the  vertex-weighted $k$-dimensional Laplacian attains the lower bound given in Theorem \ref{thm-1}. It is interesting that, for each $k$, this example also attains the lower bound given in Corollary \ref{cor-l-1} by taking $t=n-k-1$ and $r=(d+1)(k+1)-dn-1$. In this section, we give a complete characterization of the complexes attaining the lower bound given in Corollary \ref{cor-l-1}. Our main result is stated as follows.
\begin{theorem}\label{thm-22}
	Let $X$ be a simplicial complex on vertex set $V$ of size $n$ with $h(X)=d$, and endowed with a vertex weight function $\omega\colon  X \rightarrow \mathbb{R}_{>0}$. If
	\[\lambda_1^{\uparrow}(L_k^{\omega}(X))=(d+1)\min _{\sigma\in X(k)}\sum_{v\in \sigma}\omega(v)-d\sum_{v\in V}\omega(v)\]
	 for some $k$,  then we have
	\begin{enumerate}[\rm (i)]
		\item 	$\operatorname{dim}(X)=k$ and $X \cong\left(\Delta_d^{(d-1)}\right)^{*(n-k-1)} * \Delta_{(d+1)(k+1)-dn-1}$;
		\item for each $\Delta_d^{(d-1)}$, $\omega|_{V\left(\Delta_d^{(d-1)} \right)}$ is constant.
	\end{enumerate}
\end{theorem}
\begin{remark}
	By taking $\omega \equiv 1$, Theorem \ref{thm-22} immediately specializes to Theorem \ref{thm-21}.
\end{remark}

Next we prove Theorem \ref{thm-22}. Recall that  \[E_k\left(X\right)\colon=\left\{{\left\{\sigma, \tau\right\}}\colon \sigma,\tau \in X(k),\sigma\sim \tau\right\}\] and 
\begin{equation}\label{eq-Ls}
    L_k^{\omega}(X)=D^{\omega}_k+\rho^{\omega *}_k\rho_k,
\end{equation}
where  $D^{\omega}_k$, $\rho_k$ and $\rho^{\omega *}_k$ are as previously defined in the proof of Theorem \ref{thm-1}.
\begin{proof}
Let $\operatorname{x}=\left(x_\sigma\right)_{\sigma \in X(k)} \in \mathbb{R}^{f_k(X)}$ be an eigenvector of $L_k^{\omega}(X)$ associated with the $k$-th spectral gap $\lambda_1^{\uparrow}(L_k^{\omega}(X))$. Without loss of generality, we can assume that $\max _{\sigma \in X(k)}\left|x_\sigma\right|=1$. Let $\sigma_0 \in X(k)$ be such that $\left|x_{\sigma_0}\right|=1$. Then by Theorem \ref{thm-4},
\begin{equation}\label{eq-7}
	   \lambda_1^{\uparrow}(L_k^{\omega}(X))\ge L_k^{\omega}(X)_{\sigma_0, \sigma_0}-\sum_{\sigma_0\ne \sigma \in X(k)}\left|L_k^{\omega}(X)_{\sigma_0, \sigma}\right|= L_k^{\omega}(X)_{\sigma_0, \sigma_0}-\sum_{\substack{\sigma\in X(k)\colon\\\left\{\sigma,\sigma_0\right\}\in E_k(X)}}\left|L_k^\omega(X)_{\sigma_0, \sigma}\right|.
\end{equation}
By Lemma \ref{lem-3}, we have
\begin{equation}\label{eq-d}
    L_k^\omega(X)_{\sigma_0, \sigma_0}=\sum_{u \in \operatorname{lk}_X(\sigma_0)} \omega(u)+\sum_{v \in \sigma_0} \omega(v)
\end{equation}
and
\begin{equation}\label{eq-nds}
    \begin{aligned}
	\sum_{\substack{\sigma\in X(k)\colon\\\left\{\sigma,\sigma_0\right\}\in E_k(X)}}\left|L_k^\omega(X)_{\sigma_0, \sigma}\right|&=\sum_{\substack{\sigma\in X(k)\colon\\\left\{\sigma,\sigma_0\right\}\in E_k(X)}}\omega(\sigma\setminus\sigma_0)=\sum_{\substack{v\in V\setminus \sigma_0\colon\\v\notin \operatorname{lk}_X(\sigma_0)}}\omega(v)\left|N_{\sigma_0}(v)\right|\\[2mm]&=\sum_{\eta\in \sigma_0(k-1)}\sum_{\substack{v\in V\setminus\sigma_0\colon\\v\in \operatorname{lk}_X(\eta)\setminus \operatorname{lk}_X(\sigma_0)}}\omega(v)\\&=\sum_{\eta\in \sigma_0(k-1)}\left(\sum_{v\in \operatorname{lk}_X(\eta)}\omega(v)-\sum_{v\in \operatorname{lk}_X(\sigma_0)}\omega(v)-\omega(\sigma_0\setminus\eta) \right)\\[2mm]&=\sum_{\eta\in \sigma_0(k-1)}\sum_{v\in \operatorname{lk}_X(\eta)}\omega(v)-(k+1)\sum_{v\in \operatorname{lk}_X(\sigma_0)}\omega(v)-\sum_{v\in\sigma_0}\omega(v).
\end{aligned}
\end{equation}
Combining \eqref{eq-d} with \eqref{eq-nds}, it follows from \eqref{eq-7} that
\begin{align*}
    \lambda_1^{\uparrow}(L_k^{\omega}(X)) &\ge \sum_{v \in \operatorname{lk}_X(\sigma_0)} \omega(v)+\sum_{v \in \sigma_0} \omega(v)-\sum_{\eta\in \sigma_0(k-1)}\sum_{v\in \operatorname{lk}_X(\eta)}\omega(v)+(k+1)\sum_{v\in \operatorname{lk}_X(\sigma_0)}\omega(v)+\sum_{v\in\sigma_0}\omega(v)\\[2mm]&=(k+2)\sum_{v \in \operatorname{lk}_X(\sigma_0)} \omega(v)-\sum_{\eta\in \sigma_0(k-1)}\sum_{v\in \operatorname{lk}_X(\eta)}\omega(v)+2\sum_{v\in\sigma_0}\omega(v)\\[2mm]&\ge(k+2)\sum_{u \in \operatorname{lk}_X(\sigma_0)} \omega(u)-d\sum_{v\in V}\omega(v)+(d-1) \sum_{v\in \sigma_0}\omega(v)-(k+1-d)\sum_{v\in \operatorname{lk}_X(\sigma_0)}\omega(v)
	\\&\ \ \ \ +2\sum_{v\in\sigma_0}\omega(v)\\[2mm]&=(d+1)\sum_{v \in \operatorname{lk}_X(\sigma_0)} \omega(v)+(d+1) \sum_{v\in \sigma_0}\omega(v)-d\sum_{v\in V}\omega(v)
	\\&\ge(d+1)\min _{\sigma\in X(k)}\sum _{v\in \sigma}\omega(v)-d\sum_{v\in V}\omega(v)\\&=\lambda_1^{\uparrow}(L_k^{\omega}(X)),
\end{align*}
where the second inequality is obtained from Lemma \ref{lem-1}. Hence, all of the above inequalities are equalities, and the following facts hold:
\begin{fact}\label{f-1}
\rm	 $\lambda_1^{\uparrow}(L_k^{\omega}(X))=L_k^{\omega}(X)_{\sigma_0, \sigma_0}-\sum_{\sigma \in X(k), \sigma \neq \sigma_0}\left|L_k^{\omega}(X)_{\sigma_0, \sigma}\right|$.
\end{fact}
\begin{fact}\label{f-2}
\rm	$\operatorname{lk}_X(\sigma_0)=\emptyset$. 
\end{fact}
\begin{fact}\label{f-2-1}
\rm	$\sum_{v \in \sigma_0} \omega(v)=\min_{\sigma\in X(k)}\sum_{v\in \sigma}\omega(v)$.
\end{fact}
\begin{fact}\label{f-3}
\rm		For every $v\in V\setminus \sigma_0$, $\left|N_{\sigma_0}(v)\right|=\left|M_{\sigma_0}(v)\right|=d$.
\end{fact}
From the above facts, we derive the following claims:
\begin{claim}\label{c-1}
	For every $v \in V \setminus \sigma_0$, we have
	\[X\left[\{v\} \cup M_{\sigma_0}(v)\right] \cong \Delta_d^{(d-1)} \text { and } X\left[\{v\} \cup V\left(\sigma_0\right)\right] \cong \Delta_d^{(d-1)} * \Delta_{k-d}.\]
\end{claim}
\begin{proof}[\rm\textbf{Proof of Claim \ref{c-1}}]
	By Fact \ref{f-2}, we have $v\notin \operatorname{lk}_X(\sigma_0)$ for any $v \in V \setminus \sigma_0$. Combining Lemma \ref{lem-17} with Fact \ref{f-3}, the claim then follows immediately.
\end{proof}

By Theorem \ref{thm-4} and Fact \ref{f-1}, we have that $\left|x_\sigma\right|=\left|x_{\sigma_0}\right|=1$ and the quantities \[L_k^{\omega}(X)_{\sigma_0, \sigma}x_\sigma=(-1)^{\varepsilon(\sigma, \sigma_0)} \omega(\sigma\setminus\sigma_0)x_\sigma\] are sign–constant for all $\sigma \in X(k) \setminus\left\{\sigma_0\right\}$ with $\left\{\sigma, \sigma_0\right\}\in E_k$.
 Hence, exactly one of the following two cases occurs:
\begin{enumerate}[\rm  \textbf{Case} 1.]
	\item $x_{\sigma}=(-1)^{\varepsilon(\sigma, \sigma_0)}$ for all $\sigma \in X(k) \setminus\left\{\sigma_0\right\}$ with $\left\{\sigma, \sigma_0\right\}\in E_k$;
	\item $x_{\sigma}=-(-1)^{\varepsilon(\sigma, \sigma_0)}$ for all $\sigma \in X(k) \setminus\left\{\sigma_0\right\}$ with $\left\{\sigma, \sigma_0\right\}\in E_k$.
\end{enumerate}
Without loss of generality, assume Case 1 holds. Let \[\boldsymbol{x}\colon=\sum _{\sigma\in X(k)}x_{\sigma}e_{\sigma}\] be the cochain in $C^k(X, \mathbb{R})$ associated with the eigenvector
$\operatorname{x}$ chosen above.  Let $\mathcal{D}^{\omega}_k\colon C^k(X, \mathbb{R})\rightarrow C^k(X, \mathbb{R})$ denote the linear operator whose matrix representation with respect to the basis $\left\{e_{\sigma}\colon \sigma\in X(k)\right\}$ is $D^{\omega}_k$. 
\begin{claim}\label{c-1-1}
The following statements hold:
\begin{enumerate}[\rm (i)]
	\item $\rho_k\boldsymbol{x}=\boldsymbol{0}$;
	\item For every $\sigma^{\prime} \in X(k) \setminus \left\{\sigma_0\right\}$ with $\left\{\sigma^{\prime}, \sigma_0\right\}\in E_k(X)$, we have $\operatorname{lk}_X(\sigma^{\prime})=\emptyset$ and 
\[\sum _{v\in\sigma^{\prime}}\omega(v)=\sum _{v\in\sigma_0}\omega(v).\]
\end{enumerate}
\end{claim}
\begin{proof}[\rm\textbf{Proof of Claim \ref{c-1-1}}]
By \eqref{eq-1} and Lemma \ref{lem-1}, for any $\sigma \in X(k)$,
\begin{equation}\label{eq-3}
	\begin{aligned}
		\left(D^{\omega}_k\right)_{\sigma,\sigma}&=2\sum _{v \in \sigma} \omega(v)+(k+2)\sum _{v \in \operatorname{lk}_X(\sigma)} \omega(v)-\sum _{\eta \in \sigma(k-1)}\sum _{v \in \operatorname{lk}_X(\eta)} \omega(v)\\&\ge(d+1)\sum_{v \in \sigma\cup \operatorname{lk}_X(\sigma)} \omega(v)-d\sum_{v\in V}\omega(v)\\&\ge (d+1)\min_{\sigma\in X(k)}\sum_{v\in \sigma}\omega(v)-d\sum_{v\in V}\omega(v)\\&=\lambda_1^{\uparrow}(L_k^{\omega}(X)),
	\end{aligned}
\end{equation}
and thus \begin{equation}\label{eq-md}
    \min_{\sigma\in X(k)}	\left(D^{\omega}_k\right)_{\sigma,\sigma}\ge \lambda_1^{\uparrow}(L_k^{\omega}(X)).
\end{equation}
Moreover, by \eqref{eq-Ls}, we obtain the operator identity
\[L_k^{\omega}(X)=\mathcal{D}^{\omega}_k+\rho^{\omega *}_k\rho_k.\]
This implies that
\[\begin{aligned}
	\lambda_1^{\uparrow}(L_k^{\omega}(X))\left\langle\boldsymbol{x},\boldsymbol{x}\right\rangle&=	\left\langle\boldsymbol{x},L_k^{\omega}(X)\boldsymbol{x}\right\rangle\\&=	\left\langle\boldsymbol{x},\mathcal{D}^{\omega}_k\boldsymbol{x}\right\rangle+\left\langle\boldsymbol{x},\rho^{\omega *}_k\rho_k\boldsymbol{x}\right\rangle\\&\ge \min_{\sigma\in X(k)}	\left(D^{\omega}_k\right)_{\sigma,\sigma}\left\langle\boldsymbol{x},\boldsymbol{x}\right\rangle+\left\langle\rho_k\boldsymbol{x},\rho_k\boldsymbol{x}\right\rangle\\&\ge 	\lambda_1^{\uparrow}(L_k^{\omega}(X))\left\langle\boldsymbol{x},\boldsymbol{x}\right\rangle,
\end{aligned}\]
where the last inequality follows from \eqref{eq-md}.
Therefore, all of the above inequalities are equalities, i.e., $\left\langle\boldsymbol{x},\mathcal{D}^{\omega}_k\boldsymbol{x}\right\rangle=\lambda_1^{\uparrow}(L_k^{\omega}(X))\left\langle\boldsymbol{x},\boldsymbol{x}\right\rangle$ and $\rho_k\boldsymbol{x}=\boldsymbol{0}$. 
It follows that \[\left(D^{\omega}_k\right)_{\sigma^{\prime},\sigma^{\prime}}= \min_{\sigma\in X(k)}	\left(D^{\omega}_k\right)_{\sigma,\sigma}=\lambda_1^{\uparrow}(L_k^{\omega}(X))\] 
for each $\sigma^{\prime}\in X(k)$ with $x_{\sigma^{\prime}}\ne 0$. 
Since $\left|x_{\sigma^{\prime}}\right|=1\ne 0$
for all $\sigma^{\prime} \in X(k) \setminus \left\{\sigma_0\right\}$ with $\left\{\sigma^{\prime}, \sigma_0\right\}\in E_k(X)$, we have $\left(D^{\omega}_k\right)_{\sigma^{\prime},\sigma^{\prime}}=\lambda_1^{\uparrow}(L_k^{\omega}(X))$ for such $\sigma^{\prime}$, which implies that all of the inequalities in \eqref{eq-3} are equalities. Consequently, for all $\sigma^{\prime} \in X(k) \setminus \left\{\sigma_0\right\}$ with $\left\{\sigma^{\prime}, \sigma_0\right\}\in E_k(X)$, we have $\operatorname{lk}_X(\sigma^{\prime})=\emptyset$ and 
\[\sum _{v\in\sigma^{\prime}}\omega(v)=\min _{\sigma\in X(k)}\sum_{v\in \sigma}\omega(v)=\sum _{v\in\sigma_0}\omega(v).\]
\end{proof}

\begin{claim}\label{c-1-2}
	For any $v\in V\setminus\sigma_0$ and $u\in M_{\sigma_0}(v)$, we have 
\[\omega(u)=\omega(v).\]
\end{claim}
\begin{proof}[\rm\textbf{Proof of Claim \ref{c-1-2}}]
Let $v\in V\setminus\sigma_0$, $u\in M_{\sigma_0}(v)$, and define \[\tau\colon=\left(\sigma_0\setminus\{u\}\right)\cup \{v\}.\]
By the definition of $M_{\sigma_0}(v)$, we have $\tau\in X(k)$. Noticing that $|\sigma_0\cap\tau|=|\sigma_0\setminus\{u\}|=k$, it follows from Fact \ref{f-2} that $\{\tau,\sigma_0\}\in E_k(X)$. By Claim \ref{c-1-1}, we have \[\sum_{v^{\prime}\in\tau}\omega(v^{\prime})=\sum _{v^{\prime}\in\sigma_0}\omega(v^{\prime}).\] Therefore, \[\omega(u)=\sum _{v^{\prime}\in\sigma_0}\omega(v^{\prime})-\sum _{v^{\prime}\in\sigma_0\cap\tau}\omega(v^{\prime})=\sum _{v^{\prime}\in\tau}\omega(v^{\prime})-\sum _{v^{\prime}\in\sigma_0\cap\tau}\omega(v^{\prime})=\omega(v).\]
\end{proof}

\begin{claim}\label{c-2}
	Let $\sigma, \eta \in X(k)$ be such that $\left|\sigma \cap \sigma_0\right|=\left|\eta \cap \sigma_0\right|=|\sigma \cap \eta|=k$. Then
	\[\operatorname{sgn}\left( \sigma_0 \cap \sigma,\sigma_0\right) \operatorname{sgn}\left(\sigma_0 \cap \sigma,\sigma\right) \operatorname{sgn}\left( \sigma_0 \cap \eta,\sigma_0\right) \operatorname{sgn}\left(\sigma_0 \cap \eta,\eta\right) \operatorname{sgn}(\sigma \cap \eta,\sigma) \operatorname{sgn}(\sigma \cap \eta,\eta)=-1 .\]
\end{claim}
\begin{proof}[\rm\textbf{Proof of Claim \ref{c-2}}]
By Fact \ref{f-2}, $\operatorname{lk}_X(\sigma_0)=\emptyset$. This implies $\sigma \cup \sigma_0 \notin X(k+1)$ and $\eta \cup \sigma_0 \notin X(k+1)$. Thus, $\left\{\sigma,\sigma_0\right\},\left\{\eta,\sigma_0\right\}\in E_k(X)$. 
Then from the assumtion of Case 1, we have
\begin{equation}\label{eq-4}
	\left\{\begin{array}{l}
		x_\sigma=(-1)^{\varepsilon(\sigma, \sigma_0)}=\operatorname{sgn}\left( \sigma_0 \cap \sigma,\sigma_0\right) \operatorname{sgn}\left(\sigma_0 \cap \sigma,\sigma\right), \\
		x_\eta=(-1)^{\varepsilon(\eta, \sigma_0)}=\operatorname{sgn}\left( \sigma_0 \cap \eta,\sigma_0\right) \operatorname{sgn}\left(\sigma_0 \cap \eta,\eta\right).
	\end{array}\right.
\end{equation}
From Claim \ref{c-1-1} (ii), we have
$\operatorname{lk}_X(\sigma)=\operatorname{lk}_X(\eta)=\emptyset$, which implies 
$\sigma \cup \eta \notin X(k+1)$, and thereby $\left\{\sigma ,\eta\right\}\in E_k(X)$. 
It follows from Claim \ref{c-1-1} (i) that $\rho_k\boldsymbol{x}=\boldsymbol{0}$. Therefore, the component of $\rho_k\boldsymbol{x}$ indexed by $\{\sigma,\eta\}$ satisfies \[0=\left(\rho_k\boldsymbol{x}\right)_{\{\sigma, \eta\}}=\operatorname{sgn}(\sigma \cap \eta,\sigma)x_\sigma +\operatorname{sgn}(\sigma \cap \eta,\eta)x_\eta,\] i.e.,
\begin{equation}\label{eq-8}
	\operatorname{sgn}(\sigma \cap \eta,\sigma)\operatorname{sgn}(\sigma \cap \eta,\eta) x_\sigma x_\eta=-1.
\end{equation}
Combining the identities in \eqref{eq-4} and \eqref{eq-8}, we immediately deduce the claim.
\end{proof}

\begin{claim}\label{c-3}
	$M_{\sigma_0}(u) \cap M_{\sigma_0}(v)=\emptyset$ for any two distinct $u, v \in V \setminus \sigma_0$.
\end{claim}
\begin{proof}[\rm\textbf{Proof of Claim \ref{c-3}}]
Suppose to the contrary that $M_{\sigma_0}(u) \cap M_{\sigma_0}(v)\ne \emptyset$ for some two distinct  $u, v \in V \setminus \sigma_0$.
Let $w \in M_{\sigma_0}(u) \cap M_{\sigma_0}(v)$, and set $\sigma=\{u\} \cup\left(\sigma_0 \setminus\{w\}\right)$, $\eta=\{v\} \cup\left(\sigma_0 \setminus\{w\}\right)$. By definition, $\sigma, \eta \in X(k)$ and $\left|\sigma \cap \sigma_0\right|=\left|\eta \cap \sigma_0\right|=|\sigma \cap \eta|=k$. Fix a total order $\prec$ on $V$ such that $w \prec u \prec v \prec z$ for all $z \in \sigma_0 \setminus\{w\}$. It is easy to see that
\[\operatorname{sgn}\left( \sigma_0 \cap \sigma,\sigma_0\right) \operatorname{sgn}\left(\sigma_0 \cap \sigma,\sigma\right) \operatorname{sgn}\left( \sigma_0 \cap \eta,\sigma_0\right) \operatorname{sgn}\left(\sigma_0 \cap \eta,\eta\right) \operatorname{sgn}(\sigma \cap \eta,\sigma) \operatorname{sgn}(\sigma \cap \eta,\eta)=1^6=1,\]
which contradicts Claim \ref{c-2}. Hence, $M_{\sigma_0}(u) \cap M_{\sigma_0}(v)=\emptyset$.
\end{proof}

Define \[B\colon=\cup_{v \in V \setminus \sigma_0} M_{\sigma_0}(v)\text{  and  } A\colon =\sigma_0\setminus B.\]
Let $r=|A|$. Clearly, $X[A] \cong \Delta_{r-1}$. By Fact \ref{f-3}, we have \[r=|A|=k+1-(n-k-1) d=(d+1)(k+1)-d n.\]
Let $V\setminus \sigma_0\colon=\left\{v_1,v_2,\dots, v_{n-k-1}\right\}$ and \[Y\colon=X\left[\{v_1\} \cup M_{\sigma_0}(v_1)\right]*X\left[\{v_2\} \cup M_{\sigma_0}(v_2)\right]*\dots*X\left[\{v_{n-k-1}\} \cup M_{\sigma_0}(v_{n-k-1})\right].\]
By Claim \ref{c-1}, for every $v\in V\setminus\sigma_0$, we have $X\left[\{v\} \cup M_{\sigma_0}(v)\right] \cong \Delta_d^{(d-1)}$, and thus
\[Y \cong\left(\Delta_d^{(d-1)}\right)^{*(n-k-1)} \text{ and }\operatorname{dim}(Y)=d(n-k-1)-1=k-r.\]
Set $X^{\prime}=Y * X[A]$. Then 
\[X^{\prime} \cong\left(\Delta_d^{(d-1)}\right)^{*(n-k-1)} * \Delta_{r-1} \text{ and }\operatorname{dim}\left(X^{\prime}\right)=d(n-k-1)+ r-1=k.\] Moreover, by Proposition \ref{prop-1}, Claim \ref{c-1-2}, Claim \ref{c-3} and Fact \ref{f-3}, we have 
\begin{align*}
   \lambda_1^{\uparrow}\left(L_{k}^{\omega}\left(X^{\prime}\right)\right)&=\sum_{v\in A}\omega(v)\\&=\sum_{v\in \sigma_0}\omega(v)-\sum_{v\in V\setminus\sigma_0}\sum_{u\in M_{\sigma_0}(v)}\omega(u)\\&=\sum_{v\in \sigma_0}\omega(v)-\sum_{v\in V\setminus\sigma_0}\sum_{u\in M_{\sigma_0}(v)}\omega(v)\\&=\sum_{v\in \sigma_0}\omega(v)-d\sum_{v\in V\setminus\sigma_0}\omega(v)\\&=(d+1)\sum_{v\in \sigma_0}\omega(v)-d\sum_{v\in V}\omega(v)\\&=(d+1)\min _{\sigma\in X(k)}\sum_{v\in \sigma}\omega(v)-d\sum_{v\in V}\omega(v)\\&=\lambda_1^{\uparrow}(L_k^{\omega}(X)) .
\end{align*}
By Claim \ref{c-1} and Claim \ref{c-3}, we have $X \subseteq X^{\prime}$. Therefore, $k \le \operatorname{dim}(X) \le \operatorname{dim}\left(X^{\prime}\right)=k$, which implies $\operatorname{dim}(X)=\operatorname{dim}\left(X^{\prime}\right)=k$.

Next we show that $X=X^{\prime}$. Suppose to the contrary that $X\subsetneq X^{\prime}$. Noticing that the fact $\operatorname{dim}(X)=k$, it follows that $L_k^{\omega,\operatorname{up}}(X)=0$, and thus $L^{\omega}_k(X)=L_k^{\omega,\operatorname{down}}(X)$.
 By Lemma \ref{lem-4}, its matrix representation $L^{\omega}_k(X)$ satisfies
\[L^{\omega}_k(X)_{\sigma, \tau}= \begin{cases}\sum _{v \in \sigma} \omega(v), & \text { if } \sigma=\tau, \\[3mm] (-1)^{\varepsilon(\sigma, \tau)} \omega(\tau \setminus \sigma), & \text { if }|\sigma \cap \tau|=k,\\[3mm] 0, & \text { otherwise, }\end{cases}\]
which implies that $L^{\omega}_k(X)$ is a principal submatrix of $L^{\omega}_k(X^{\prime})$. Note that \[\lambda_1^{\uparrow}(L_k^{\omega}(X))=\sum _{v\in A}\omega(v)=\lambda_1^{\uparrow}\left(L_{k}^{\omega}\left(X^{\prime}\right)\right).\] By Corollary \ref{cor-06}, we obtain a contradiction provided that the following claim holds:
\begin{claim}\label{c-4}
	Every eigenvector of $L^{\omega}_k\left(X^{\prime}\right)$ corresponding to $\lambda_1^{\uparrow}\left(L_{k}^{\omega}\left(X^{\prime}\right)\right)=\sum _{v\in V\left(\Delta_{r-1}\right)}\omega(v)$ has no zero entries.
\end{claim}
\begin{proof}[\rm\textbf{Proof of Claim \ref{c-4}}]
	Recall that $X^{\prime}=Y * X[A]$ and $\operatorname{dim}\left(X^{\prime}\right)=k$, where \[Y \cong\left(\Delta_d^{(d-1)}\right)^{*(n-k-1)}\] with $\operatorname{dim}(Y)=k-r$ and $X[A] \cong \Delta_{r-1}$ with $\operatorname{dim}\left(X[A]\right)=r-1$.
This implies that $A \subseteq \sigma$ for every $\sigma \in X^{\prime}(k)$, and thus \[Y(k-r)=\left\{\sigma \setminus A\colon  \sigma \in X^{\prime}(k)\right\}.\] Therefore, for any $\sigma, \tau \in X^{\prime}(k)$, \[\left\{\sigma \setminus A, \tau \setminus A\right\}\in E_{k-r}(Y) \text{ if and only if }\left\{\sigma, \tau\right\}\in E_k\left(X^{\prime}\right).\]
Moreover, since $\operatorname{dim}(Y)=k-r$ and $\operatorname{dim}\left(X^{\prime}\right)=k$, we have \[\operatorname{lk}_Y(\sigma \setminus A)=\operatorname{lk}_{X^{\prime}}(\sigma)=\emptyset\] for all $\sigma \in X^{\prime}(k)$. 
It follows that
\[E_k\left(X^{\prime}\right)\colon=\left\{{\left\{\sigma, \tau\right\}}\colon \sigma,\tau \in X(k),|\sigma\cap\tau|=k\right\}\]
and \[E_k\left(Y\right)\colon=\left\{{\left\{\sigma \setminus A, \tau \setminus A\right\}}\colon \left\{\sigma, \tau\right\}\in E_k\left(X^{\prime}\right)\right\}. \] 
Fix an ordering $\prec$ on $V\left(X^{\prime}\right)$ such that $u \prec v$ whenever $v \in A$ and $u \in V\left(X^{\prime}\right) \setminus A$.  
For any $\left\{\sigma, \tau\right\} \in E_k(X^{\prime})$, we have $\tau\setminus\sigma=\left(\tau\setminus A\right)\setminus\left(\sigma\setminus A\right)$ and $(-1)^{\varepsilon(\sigma, \tau)} =(-1)^{\varepsilon(\sigma\setminus A, \tau\setminus A)}$ under the ordering, which implies that 
\[L^{\omega}_k\left(X^{\prime}\right)_{\sigma, \tau}=(-1)^{\varepsilon(\sigma, \tau)}\omega\left( \tau\setminus\sigma\right)=(-1)^{\varepsilon(\sigma\setminus A, \tau\setminus A)}\omega\left( \left(\tau\setminus A\right)\setminus\left(\sigma\setminus A\right)\right)=L^{\omega}_{k-r}(Y)_{\sigma\setminus A, \tau\setminus A}.\] Consequently, the corresponding Laplacian matrices satisfy
	\[L^{\omega}_k\left(X^{\prime}\right)=L^{\omega}_{k-r}(Y)+\sum _{v\in A}\omega(v) I_{f_k(X^{\prime})},\]
    where $I_{f_k(X^{\prime})}$ is the $f_k(X^{\prime}) \times f_k(X^{\prime})$ identity matrix.
Combining this relation with $\lambda_1^{\uparrow}\left(L_{k}^{\omega}\left(X^{\prime}\right)\right)=\sum _{v\in A}\omega(v)$, we deduce that $\lambda_1^{\uparrow}\left(L_{k-r}^{\omega}\left(Y\right)\right)=0$, and that the matrices $L^{\omega}_k\left(X^{\prime}\right)$ and $L^{\omega}_{k-r}(Y)$ share exactly the same eigenvectors with respect to $\lambda_1^{\uparrow}\left(L_{k}^{\omega}\left(X^{\prime}\right)\right)$ and $\lambda_1^{\uparrow}\left(L_{k-r}^{\omega}\left(Y\right)\right)$, respectively. Therefore,
it suffices to show that every eigenvector of $L^{\omega}_{k-r}(Y)$ with respect to $\lambda_1^{\uparrow}\left(L_{k-r}^{\omega}\left(Y\right)\right)$ has no zero entries.

For the sake of contradiction, assume that $\operatorname{y}=\left(y_\tau\right)_{\tau \in Y(k-r)} \in \mathbb{R}^{f_{k-r}(Y)}$ is an eigenvector of $L^{\omega}_{k-r}(Y)$ with respect to $\lambda_1^{\uparrow}\left(L_{k-r}^{\omega}\left(Y\right)\right)$ such that $y_{\tau_0}=0$ for some $\tau_0 \in Y(k-r)$. By  Lemma \ref{lem-8}, Lemma \ref{lem-9} and Lemma \ref{lem-10}, we have that the polyhedron of $Y$ satisfies
 \[\|K^{Y}\| \cong \left\|K^{\left(\Delta_d^{(d-1)}\right)^{*(n-k-1)}}\right\| \cong\left\|K^{\Delta_d^{(d-1)}}\right\|^{*(n-k-1)} \cong\left(S^{d-1}\right)^{*(n-k-1)} \cong S^{d(n-k-1)-1}=S^{k-r}.\]
Hence, $Y$ is a triangulation of the $(k-r)$-dimensional sphere $S^{k-r}$. 
Let $Z=Y \setminus \tau_0$, which is a $(k-r)$-dimensional proper subcomplex of $Y$. Since any proper subcomplex of a triangulation of an $(k-r)$-dimensional sphere has trivial $(k-r)$-dimensional cohomology, it follows that $\widetilde{H}^{k-r}(Z ; \mathbb{R})=0$. By Theorem \ref{thm-6}, we have
	\begin{equation}\label{eq-6}
		\lambda_1^{\uparrow}\left(L_{k-r}^{\omega}\left(Z\right)\right)>0.
	\end{equation}
	Note that $\operatorname{dim}(Z)=\operatorname{dim}(Y)=k-r$ since $d=h(X)\ge 1$. It is easy to see that $L^{\omega}_{k-r}(Z)$ is a proper principal submatrix of $L^{\omega}_{k-r}(Y)$ indexed by $Y(k-r) \setminus\tau_0$. As $\operatorname{y}$ is an eigenvector of $L^{\omega}_{k-r}(Y)$ with respect to $\lambda_1^{\uparrow}\left(L_{k-r}^{\omega}\left(Y\right)\right)$ and satisfies $y_{\tau_0}=0$, it follows from Corollary \ref{cor-06} that 
	\[\lambda_1^{\uparrow}\left(L_{k-r}^{\omega}\left(Z\right)\right)=\lambda_1^{\uparrow}\left(L_{k-r}^{\omega}\left(Y\right)\right)=0,\]
which contradicts to \eqref{eq-6}. This completes the proof of Claim \ref{c-4}.
\end{proof}
Returning to the main argument, recall that $L^{\omega}_k(X)$ is a principal submatrix of $L^{\omega}_k(X^{\prime})$ and \[\lambda_1^{\uparrow}(L_k^{\omega}(X))=\lambda_1^{\uparrow}\left(L_k^{\omega}\left(X^{\prime}\right)\right).\]
 By Corollary \ref{cor-06}, $L^{\omega}_k(X^{\prime})$ has an eigenvector $\operatorname{x}$ with respect to $\lambda_1^{\uparrow}\left(L_k^{\omega}\left(X^{\prime}\right)\right)$ such that $x_{\tau}=0$ for all $\tau \notin X(k)$, which contradicts Claim \ref{c-4}.
Therefore, \[X=X^{\prime} \cong\left(\Delta_d^{(d-1)}\right)^{*(n-k-1)} * \Delta_{r-1}=\left(\Delta_d^{(d-1)}\right)^{*(n-k-1)} * \Delta_{(d+1)(k+1)-d n-1}.\] In particular, $\operatorname{dim}(X)=\operatorname{dim}\left(X^{\prime}\right)=k$.
 Hence, (i) holds.
  Combining (i) with Claim \ref{c-1-2}, we have that $\omega|_{V\left(\Delta_d^{(d-1)} \right)}$ is constant for each $\Delta_d^{(d-1)}$. Hence, (ii) also holds.
  
   We complete this proof.
\end{proof}
\section{ The $i$-smallest eigenvalue of vertex-weighted Laplacian operators of simplicial complexes}

Recall that the clique complex $X$ is the simplicial complex whose simplices correspond to all cliques of its underlying graph $G_X$. The \textit{independence complex} of a graph $G=(V, E)$ is the simplicial complex $I(G)$ on vertex set $V$ whose simplices are the independent sets in $G$. Clearly, it is precisely the clique complex of the complement graph $\overline{G}$.

In 2005, Aharoni, Berger, and Meshulam \cite{ABM2005} established a recursive inequality for the smallest eigenvalues of the combinatorial Laplacians of a clique complex on $n$ vertices.
\begin{theorem}[\cite{ABM2005}]
Let $X$ be a clique complex on $n$ vertices. Then  \[k\lambda_1^{\uparrow}\left(L_k(X)\right) \ge (k+1)\lambda_1^{\uparrow}\left(L_{k-1}(X)\right)-n.\]
\end{theorem}
As a corollary of the above inequality, they also obtained a lower bound for $\lambda_1^{\uparrow}\left(L_k(X)\right)$.
\begin{corollary}[\cite{ABM2005}]\label{cor-3}
	Let $X$ be a clique complex on $n$ vertices. Then
	\[
	\lambda_1^{\uparrow}\bigl(L_k(X)\bigr) \ge (k+1)\,\lambda_1^{\uparrow}\bigl(L(G_X) + J\bigr) - k n,
	\]
	where $L(G_X)$ denotes the Laplacian of the underlying graph $G_X$, and $J$ denotes the $n \times n$ matrix with all entries equal to $1$.
\end{corollary}

In 2024, Lew \cite{L2024} derived the following lower bounds for the $i$-th smallest eigenvalue of the vertex-weighted Laplacians of a clique complex $X$ using additive compound matrices. This result not only generalizes Corollary \ref{cor-3} to vertex-weighted Laplacians, but also provides bounds for all eigenvalues. Recall that the underlying graph $G_X$ has vertex set $V$ and edge set $X(1)$. For any vertex weight function $\omega$, the vertex-weighted Laplacian of $G_X$ is the matrix $L^\omega(G_X) \in \mathbb{R}^{\left|V\right|\times \left|V\right|}$ defined by
\[L^w(G_X)_{u, v}= \begin{cases}\sum_{u^{\prime} \in N_{G_X}(u)} \omega\left(u^{\prime}\right), & \text { if } u=v, \\ -\omega(v), & \text { if }v\in \operatorname{lk}_X(u), \\ 0, & \text { otherwise. }\end{cases}\]
Let $J^{\omega}$ be the $\lvert V \rvert \times \lvert V \rvert$ matrix with entries $J^{\omega}_{u,v} = \omega(v)$ for all $u,v \in V$.
\begin{theorem}[\cite{L2024}]\label{thm-12}
	Let $X$ be a clique complex on vertex set $V(X)$, and let $\omega\colon X \rightarrow \mathbb{R}_{\geq 0}$ be a vertex weight function of $X$. For $k \geq 0$ and $1 \leq i \leq f_k(X)$, 
	\[\lambda_i^{\uparrow}\left(L_k^{\omega}(X)\right) \geq S_{k+1, i}^{\uparrow}\left(L^{\omega}(G_X)+J^{\omega}\right)-k \sum_{v \in V} \omega(v).\]	
\end{theorem}
Recently, Zhan, Huang, and Zhou \cite{ZHZ2025} employed elementary matrix techniques to extend the results of Theorem \ref{thm-12} to the combinatorial Laplacians of general simplicial complexes. They introduced a new notation $\sigma[j]$ to distinguish a simplicial complex $X$ from the clique complex of its underlying graph $G_X$.
For any $0\le k \le \operatorname{dim}(X)$, $0 \leq j \leq k+1$ and $\sigma \in X(k)$,
\[\sigma[j]\colon=\left\{u \in \bigcap_{v\in \sigma}\operatorname{lk}_X(v)\setminus \operatorname{lk}_X(\sigma)\colon \left|N_{\sigma}(u)\right|=j
\right\}.\]
One readily checks that $\{\sigma[0], \ldots, \sigma[k+1]\}$ constitutes a partition of the set
\[\bigcap_{v\in \sigma}\operatorname{lk}_X(v)\setminus \operatorname{lk}_X(\sigma).\]
Note that $\sigma[j]=\emptyset$ whenever $X$ is a clique complex, since in this case the condition $u \in \operatorname{lk}_X(v)$ for all $v \in \sigma$ already implies $u \in \operatorname{lk}_X(\sigma)$. 
Moreover, for clique complex, if the condition is removed, then the quantity $\max_{\sigma\in X(k)}|\sigma[j]|$ coincides with the parameter $D_k(X,j)$ originally introduced by Shukla and Yogeshwaran \cite{SY2020}.
\begin{theorem}[\cite{ZHZ2025}]\label{thm-12-1}
Let $X$ be a simplicial complex on $n$ vertices. Then, for $k \geq 0$ and $1 \leq i \leq f_k(X)$,
\[\lambda_i^{\uparrow}\left(L_k(X)\right) \geq S_{k+1, i}^{\uparrow}\left(L\left(G_X\right)+J\right)-k n-\max _{\sigma \in X(k)} \sum_{j=0}^{k+1}(j+1) \cdot|\sigma[j]|.\]   
\end{theorem}

Building directly upon the method established in \cite{ZHZ2025}, this section extends Theorem \ref{thm-12-1} from the combinatorial Laplacians to the vertex-weighted Laplacians of general simplicial complexes, yielding an upper bound on the dimension of the $k$-th cohomology group $\widetilde{H}^k(X;\mathbb{R})$ of the complex $X$.

\subsection{Lower bound of the $i$-smallest eigenvalue}
We now state our main result as follows.
\begin{theorem}\label{thm-11}
	Let $X$ be a simplicial complex on vertex set $V$ with a vertex weight function $\omega\colon X\to \mathbb{R}_{>0}$. Then, for $k \geq 0$ and $1 \leq i \leq f_k(X)$,
	\[\lambda^{\uparrow}_i\left(L^{\omega}_k(X)\right) \ge S^{\uparrow}_{k+1, i}\left(L^{\omega}\left(G_X\right)+J^{\omega}\right)-k\sum _{u\in V}\omega(u)-\max _{\sigma \in X(k)}\left\{\sum_{j=0}^{k+1}(j+1)\sum _{u\in \sigma[j]}\omega(u)\right\}.\]
\end{theorem}
\begin{remark}
On the one hand, when $\omega \equiv 1$, an application of Theorem \ref{thm-11} directly yields lower bounds for the $i$-th smallest eigenvalue of the combinatorial Laplacians of general simplicial complexes given by \cite{ZHZ2025}; on the other hand, if $X$ be a clique complex, then it follows from the above arguments that \[\max _{\sigma \in X(k)}\left\{\sum_{j=0}^{k+1}(j+1)\sum _{u\in \sigma[j]}\omega(u)\right\}=0.\] Therefore, Theorem \ref{thm-11} immediately reduces to Theorem \ref{thm-12}.
\end{remark}

Let $X$ be a simplicial complex with a vertex weight function $\omega\colon X\to \mathbb{R}_{>0}$. To prove Theorem \ref{thm-11}, we define two $f_k(X) \times f_k(X)$ matrices $P^{\omega}$ and $Q^{\omega}$ by
\[P^{\omega}_{\sigma, \tau}= \begin{cases}\sum _{v \in\sigma}\sum _{\left\{u,v\right\}\in X(1)}\omega(u)-\sum _{v \in \operatorname{lk}_X(\sigma)} \omega(v),& \text { if } \sigma=\tau, \\[4mm] -(-1)^{\varepsilon(\sigma, \tau)} \omega(\tau \setminus \sigma), & \text { if }\sigma \sim \tau \text { and } \sigma\triangle\tau\in X(1),\\[4mm] 0, & \text { otherwise, }\end{cases}\]
and
\[Q^{\omega}_{\sigma, \tau}= \begin{cases}\sum _{v \in\sigma}\sum _{\left\{u,v\right\}\in X(1)}\omega(u)+\sum _{v \in \sigma} \omega(v), & \text { if } \sigma=\tau, \\[4mm] (-1)^{\varepsilon(\sigma, \tau)} \omega(\tau\setminus\sigma), & \text { if } \sigma\sim\tau \text { and } \sigma\triangle\tau\notin X(1),\\[4mm] 0, & \text { otherwise, }\end{cases}\]	for any $\sigma, \tau \in X(k)$, where $\varepsilon(\sigma, \tau)$ denotes the number of elements in $\sigma \cap \tau$ between $i$ and $j$ for $|\sigma \cap \tau|=k, \sigma \setminus \tau=\{i\}, \tau \setminus \sigma=\{j\}$. By Lemma \ref{lem-3}, \[L_k^{\omega}(X)=Q^{\omega}-P^{\omega}.\]

The proof of Theorem \ref{thm-11} mainly relies on the following lemma.
\begin{lemma}\label{lem-13}
For $k \geq 0$,
	\[\lambda_{1}^{\downarrow}(P^{\omega}) \le k\sum _{u\in V}\omega(u)+\max _{\sigma \in X(k)}\left\{\sum_{j=0}^{k+1}(j+1)\sum _{u\in \sigma[j]}\omega(u)\right\}.\]
\end{lemma}
\begin{proof}
By Theorem \ref{thm-4}, we have
\[\lambda_{1}^{\downarrow}(P)\le \max _{\sigma \in X(k)}\left\{\sum _{v \in\sigma}\sum _{u\in \operatorname{lk}_X(v)}\omega(u)-\sum_{v \in \operatorname{lk}_X(\sigma)} \omega(v)+\sum _{\substack{\tau\in X(k)\colon \\ \sigma\sim\tau, \sigma\triangle\tau\in X(1)}}\omega\left(\tau\setminus\sigma\right)\right\}.\]
Let 
\[T_1\colon=\sum _{v \in\sigma}\sum _{u\in \operatorname{lk}_X(v)}\omega(u)\]
and \[T_2\colon=\sum _{\substack{\tau\in X(k)\colon \\ \sigma\sim \tau, \sigma\triangle\tau\in X(1)}}\omega\left(\tau\setminus\sigma\right).\]
Therefore, it suffices to estimate the upper bounds of $T_1$ and $T_2$.

Firstly, we discuss the upper bounds of $T_1$.
For any $u\in V$, it suffices to compute the number of occurrences of $\omega(u)$ in $T_1$, i.e., the number of $v\in \sigma$ satisfying $u\in \operatorname{lk}_X(v)$. If $u \in \operatorname{lk}(\sigma)$, then $u \in \operatorname{lk}(v)$ for every $v \in \sigma$, implying that $\omega(u)$ appears exactly $k+1$ times in $T_1$. For $u \notin \operatorname{lk}(\sigma)$, two cases may arise: either $u \in \bigcup_{j=0}^{k+1} \sigma[j]$ or $u \notin \bigcup_{j=0}^{k+1} \sigma[j]$. In the former case, by the definition of $\sigma[j]$, $\omega(u)$ appears exactly $k+1$ times; in the latter case, since $u \notin \bigcup_{j=0}^{k+1} \sigma[j]$, $\omega(u)$ appears at most $k$ times. Therefore,
\begin{equation}\label{eq-9}
\begin{aligned}
T_1&\le (k+1)\left(
      \sum_{u\in\operatorname{lk}_X(\sigma)}\omega(u)
      + \sum_{u\in\bigcup_{j=0}^{k+1}\sigma[j]}\omega(u)
     \right)+ k\left(
      \sum_{u\in V}\omega(u)
      - \sum_{u\in\operatorname{lk}_X(\sigma)}\omega(u)\right.
\\[-1mm]
&\qquad
\left.
      - \sum_{u\in\bigcup_{j=0}^{k+1}\sigma[j]}\omega(u)
    \right)
\\[3mm]
&= k\sum_{u\in V}\omega(u)
 + \sum_{u\in\operatorname{lk}_X(\sigma)}\omega(u)
 + \sum_{u\in\bigcup_{j=0}^{k+1}\sigma[j]}\omega(u).
\end{aligned}
\end{equation}

We next consider $T_2$. Fix $\sigma \in X(k)$. On the one hand, if $\tau \in X(k)$ is such that $\sigma \sim\tau$ and $\sigma\triangle\tau\in X(1)$, and let $\left\{u\right\}=\tau\setminus\sigma$, $\{w\}=\sigma\setminus\tau$, it follows that 
\[u \notin \operatorname{lk}(\sigma), u\in\operatorname{lk}(\sigma\setminus\{w\}) \text{ and } u \in \operatorname{lk}(v)\]
for all $v \in \sigma$. Thus, we have $u \in \bigcup_{j=1}^{k+1} \sigma[j]$. On the other hand, for $1 \leq j \leq k+1$ and $u \in \sigma[j]$, there are exactly $j$ simplices in $X(k)$, say $\tau_1^u, \dots, \tau_j^u$, such that \[\sigma \sim\tau_i^u, \sigma\triangle\tau_i^u\in X(1)\text{ and } \tau_i^u \backslash \sigma=\{u\}\]
for $1 \leq i \leq j$. Note that for any two distinct $u \in \sigma[j]$ and $u^{\prime} \in \sigma[j^{\prime}]$, the two sets $\left\{\tau_1^u, \dots, \tau_j^u\right\}$ and $\left\{\tau_1^{u^{\prime}}, \dots, \tau_{j^{\prime}}^{u^{\prime}}\right\}$ are disjoint. Therefore, 
\begin{equation}\label{eq-10}
	T_2=\sum _{\substack{\tau\in X(k)\colon \\ \sigma\sim\tau, \sigma\triangle\tau\in X(1)}}\omega\left(\tau\setminus\sigma\right)=\sum_{j=1}^{k+1}\sum _{u\in \sigma[j]}j\omega(u).
\end{equation}
Combinning \eqref{eq-9} with \eqref{eq-10}, we have
\begin{align*}
	\lambda_{1}^{\downarrow}(P) 
	&\le \max _{\sigma \in X(k)} \left\{
	\sum _{v \in\sigma}\sum _{u\in \operatorname{lk}_X(v)}\omega(u)
	-\sum _{v \in \operatorname{lk}_X(\sigma)} \omega(v)
	+\sum _{\tau\in X(k)\colon \sigma\sim\tau, \sigma\triangle\tau\in X(1)}\omega(\tau\setminus\sigma)
	\right\} \notag \\[2mm]
	&\le \max_{\sigma\in X(k)}
\Biggl\{
k\sum_{u\in V}\omega(u)
+ \sum_{u\in\operatorname{lk}_X(\sigma)}\omega(u)
+ \sum_{u\in \bigcup_{j=0}^{k+1}\sigma[j]}\omega(u)- \sum_{v\in\operatorname{lk}_X(\sigma)}\omega(v)
\\[-1mm]
&\hphantom{\le\max_{\sigma\in X(k)}\Biggl\{}
\qquad
+ \sum_{j=1}^{k+1} j\sum_{u\in\sigma[j]}\omega(u)
\Biggr\}
\notag\\[2mm]
	&= \max _{\sigma \in X(k)}\left\{
	k\sum _{u\in V} \omega(u)
	+\sum_{j=0}^{k+1}\sum _{u\in \sigma[j]} (j+1)\omega(u)
	\right\} \notag\\[2mm]
	&= k\sum _{u\in V} \omega(u)
	+\max _{\sigma \in X(k)}\left\{
	\sum_{j=0}^{k+1}(j+1)\sum _{u\in \sigma[j]}\omega(u)
	\right\}.
\end{align*}

This completes the proof.
\end{proof}

Now we prove Theorem \ref{thm-11}.
\begin{proof}[\rm\textbf{Proof of Theorem \ref{thm-11}}]
Recall that the vertex-weighted Laplacian $L^{\omega}(G_X)$ of $G_X$ and $J^{\omega}$. We have
\begin{equation}\label{eq-10-1}
	\left(L^{\omega}(G_X)+J^{\omega}\right)_{u, v}= \begin{cases}\sum_{u^{\prime} \in N_{G_X}(u)} \omega\left(u^{\prime}\right)+\omega(u),& \text { if } u=v, \\ \omega(v),& \text { if }v\notin \operatorname{lk}_X(u),\\ 0, & \text { otherwise,}\end{cases}
\end{equation}
for all $u, v \in V$. 

For $\sigma, \tau \in X(k)$ with $|\sigma \cap \tau|=k$, let $\sigma\setminus\tau=\{u\}$ and $\tau\setminus\sigma=\{v\}$. If $v \notin \operatorname{lk}_X(u)$, then $\sigma \sim \tau$ and $\sigma \triangle \tau\notin X(1)$.
Conversely, if $\sigma \sim \tau$ and $\sigma \triangle \tau\notin X(1)$, then $v \notin \operatorname{lk}_X(u)$. Therefore,
 \[v \notin \operatorname{lk}_X(u) \text{ if and only if }\sigma \sim \tau \text{ and } \sigma \triangle \tau\notin X(1).\]
  It follows from \eqref{eq-10-1} and Theorem \ref{thm-9} that the matrix $Q^{\omega}$ is a principal submatrix of $\left(L^{\omega}\left(G_X\right)+J^{\omega}\right)^{[k+1]}$, with rows and columns indexed by $X(k)$. 
It is easy to see that $\left(L^{\omega}\left(G_X\right)+J^{\omega}\right)^{[k+1]}$ is a matrix representation of some self-adjoint operator. Therefore, by Corollary \ref{cor-1} and Theorem \ref{thm-10}, we have
	\[\lambda_{i}^{\uparrow}(Q^{\omega})\ge\lambda_i^{\uparrow}\left(\left(L^{\omega}\left(G_X\right)+J^{\omega}\right)^{[k+1]}\right)=S_{k+1, i}^{\uparrow}\left(L^{\omega}\left(G_X\right)+J^{\omega}\right)\]
	Then it follows from Lemma \ref{lem-12} and Lemma \ref{lem-13} that
\[\begin{aligned}
		\lambda^{\uparrow}_i\left(L^{\omega}_k(X)\right) & \geq \lambda^{\uparrow}_i(Q^{\omega})+\lambda^{\uparrow}_{1}(-P^{\omega})\\&\ge S^{\uparrow}_{k+1, i}\left(L^{\omega}\left(G_X\right)+J^{\omega}\right)-\lambda^{\uparrow}_{f_k(X)}(P^{\omega})\\&=S^{\uparrow}_{k+1, i}\left(L^{\omega}\left(G_X\right)+J^{\omega}\right)-\lambda^{\downarrow}_{1}(P^{\omega}) \\
		& \ge S^{\uparrow}_{k+1, i}\left(L^{\omega}\left(G_X\right)+J^{\omega}\right)-k\sum _{u\in V}\omega(u)-\max _{\sigma \in X(k)}\left\{\sum_{j=0}^{k+1}(j+1)\sum _{u\in \sigma[j]}\omega(u)\right\}.
	\end{aligned}\]
	This completes the proof.
\end{proof}
\subsection{Upper bound on the dimension of the cohomology group}
As a consequence of Theorem \ref{thm-11}, we obtain an upper bound on the dimension of the $k$-th cohomology group $\widetilde{H}^k(X;\mathbb{R})$ of the simplicial complex $X$, which covers the result of Zhan et al. (see \cite[Corollary 1.5]{ZHZ2025}), improves the result of Aharoni et al. (see \cite[Corollary 1.7]{ABM2005}), and extends Lew’s result (see \cite[Theorem 1.4]{L2024}) to general simplicial complexes.
\begin{theorem}\label{thm-23}
	Let $X$ be a simplicial complex on vertex set $V$ of size $n$, and let $\omega\colon X \rightarrow \mathbb{R}_{>0}$ be a vertex weight function of $X$. For any $k \geq 0$,
	\begin{align*}
		\operatorname{dim}\left(\widetilde{H}^k(X;\mathbb{R})\right)
		&\le
		\Big|\Big\{I\in\binom{[n]}{k+1}\colon 
		\sum_{i\in I} \lambda_i^{\uparrow}\!\left(L^{\omega}\!\left(G_X\right)+J^{\omega}\right)\\
		&\qquad\le 
		k\!\sum_{u\in V}\!\omega(u)
		+\max_{\sigma \in X(k)}\!
		\Big\{\sum_{j=0}^{k+1}(j+1)\!\sum_{u\in \sigma[j]}\!\omega(u)\Big\}
		\Big\}\Big|.
	\end{align*}
\end{theorem}

\begin{corollary}[{\cite[Corollary 1.5]{ZHZ2025}}]\label{cor-1-5}
Let $X$ be a simplicial complex on $n$ vertices. For any $k \geq 0$,
\[\operatorname{dim}\left(\widetilde{H}^k(X ; \mathbb{R})\right) \leq\left\{A \in\binom{[n]}{k+1}\colon\sum_{i \in A} \lambda_i\left(L\left(G_X\right)+J\right) \leq k n+\max _{\sigma \in X(k)} \sum_{j=0}^{k+1}(j+1)|\sigma[j] |\right\}.\]
\end{corollary}
\begin{corollary}[{\cite[Corollary 1.7]{ABM2005}}]
Let $k$-skeleton of $X$ be same as that of the $k$-skeleton of the clique complex of $G_X$. If \[\lambda_2^{\uparrow}\left(L\left(G_X\right)\right)>\frac{k n}{k+1}+(k+1) D_k(X, k+1),\] then $\widetilde{H}^k(X; \mathbb{R})=0$.
\end{corollary}
\begin{corollary}[{\cite[Theorem 1.4]{L2024}}]
Let $G=(V, E)$ be a graph on $n$ vertices, and let $\omega\colon  V \rightarrow \mathbb{R}_{\geq 0}$. Then, for all $k \geq 0$,
\[\operatorname{dim}\left(\widetilde{H}_k(I(G) ; \mathbb{R})\right) \leq\left|\left\{I \in\binom{[n]}{k+1}\colon  \sum_{i \in I} \lambda_i^{\downarrow}\left(L^w(G)\right) \geq \sum_{v \in V} \omega(v)\right\}\right|.\]
\end{corollary}
We discuss the improvements and extensions in detail in the following remark.
\begin{remark}
On the one hand, if $\omega\equiv 1$, then Theorem \ref{thm-23} immediately reduces to Corollary \ref{cor-1-5}. If we further assume that the $k$-skeleton of $X$ coincides with that of the clique complex of $G_X$, then $\sigma[j] = \emptyset$ for all $\sigma \in X(k)$ and all $0 \le j \le k$.
 In this case, \[\max_{\sigma \in X(k)}\Big\{\sum_{j=0}^{k+1}(j+1)\sum_{u\in \sigma[j]}\omega(u)\Big\}=(k+2)\max_{\sigma \in X(k)}\left| \sigma[k+1]\right|=(k+2)D_k(X,k+1).
  \]
It follows from $\lambda_1^{\uparrow}\left(L\left(G_X\right)+J\right)=\lambda_2^{\uparrow}\left(L\left(G_X\right)\right)$ that
\[\sum_{i\in I} \lambda_i^{\uparrow}\left(L^{\omega}\left(G_X\right)+J^{\omega}\right)\ge (k+1)\lambda_2^{\uparrow}\left(L\left(G_X\right)\right).\]
Therefore, if \[ \lambda_2^{\uparrow}\left(L\left(G_X\right)\right)>\frac{kn}{k+1}+\frac{k+2}{k+1}D_k(X,k+1),
\]
 then by Theorem \ref{thm-23}, we have $\operatorname{dim}\left(\widetilde{H}^k(X;\mathbb{R})\right)=0$, i.e., $\widetilde{H}^k(X ; \mathbb{R})=0$. This improves upon the result of \cite[Corollary 1.7]{ABM2005}, since \[\frac{kn}{k+1}+(k+1)D_k(X,k+1)>\frac{kn}{k+1}+\frac{k+2}{k+1}D_k(X,k+1).\]

\vspace{8pt}
On the other hand, when $X$ is an independence complex of graph $G$, we have \[\max_{\sigma \in X(k)}
\Big\{\sum_{j=0}^{k+1}(j+1)\sum_{u\in \sigma[j]}\!\omega(u)\Big\}=0.\]
It follows from Theorem \ref{thm-23} that
\[	\operatorname{dim}\left(\widetilde{H}^k(X;\mathbb{R})\right)\le 	\Big|\Big\{I\in\binom{[n]}{k+1}\colon 
\sum_{i\in I} \lambda_i^{\uparrow} \left(L^{\omega} \left(\overline{G} \right)+J^{\omega}\right)\le 
k\sum_{u\in V} \omega(u)\Big\}\Big|.\]
Note that \[L^{\omega} \left(\overline{G} \right)+J^{\omega}=\sum_{v\in V}\omega(v)I_n-L^{\omega} \left(G\right),\]
which implies that \[\sum_{i\in I} \lambda_i^{\uparrow} \left(L^{\omega} \left(\overline{G} \right)+J^{\omega}\right)=(k+1)\sum_{v\in V}\omega(v)-\sum_{i\in I} \lambda_i^{\downarrow} \left(L^{\omega} \left(G\right)\right),\]
and thus
\[	\operatorname{dim}\left(\tilde{H}^k(I(G) ; \mathbb{R})\right) \leq\left|\left\{I \in\binom{[n]}{k+1}\colon  \sum_{i \in I} \lambda_i^{\downarrow}\left(L^\omega(G)\right) \geq \sum_{v \in V} \omega(v)\right\}\right|.\] This coincides with the result established by Lew \cite[Theorem 1.4]{L2024}.
\end{remark}

Next we prove Theorem \ref{thm-23}.

\begin{proof}[\rm\textbf{Proof of Theorem \ref{thm-23}}]
Let 
\begin{align*}
	j
	&=	\Big|\Big\{I\in\binom{V}{k+1}\colon 
	\sum_{i\in I} \lambda_i^{\uparrow}\!\left(L^{\omega}\!\left(G_X\right)+J^{\omega}\right)\le 
	k\!\sum_{u\in V}\!\omega(u)\\
	&\qquad
	+\max_{\sigma \in X(k)}\!
	\Big\{\sum_{j=0}^{k+1}(j+1)\!\sum_{u\in \sigma[j]}\!\omega(u)\Big\}
	\Big\}\Big|\\[2mm]
	&=\max\Bigg\{
	\Bigg\{1\le i\le \binom{n}{k+1}\colon \;
S^{\uparrow}_{k+1,i}\!\left(L^{\omega}\!\left(G_X\right)+J^{\omega}\right)\le
	k\!\sum_{u\in V}\!\omega(u)\\
	&\qquad
	+\max_{\sigma \in X(k)}\!\left\{\sum_{j=0}^{k+1}(j+1)\!\sum_{u\in \sigma[j]}\!\omega(u)\right\}
	\Bigg\}\cup \{0\}
	\Bigg\}.
\end{align*}
If $j=\binom{n}{k+1}$, then we have 
$\operatorname{dim}\left(\widetilde{H}^k(X;\mathbb{R})\right)\le f_k(X)\le j$ as desired. If $j<\binom{n}{k+1}$, then by the maximality of $j$, we have \[S^{\uparrow}_{k+1,j+1}\left(L^{\omega}\left(G_X\right)+J^{\omega}\right)>k\sum _{u\in V}\omega(u)+\max _{\sigma \in X(k)}\left\{\sum_{j=0}^{k+1}(j+1)\sum _{u\in \sigma[j]}\omega(u)\right\}.\]
Therefore, by Theorem \ref{thm-11}, it follows that $\lambda^{\uparrow}_{j+1}\left(L^{\omega}_k(X)\right) >0$. Combining this with Theorem \ref{thm-6}, we have \[	\operatorname{dim}\left(\widetilde{H}^k(X;\mathbb{R})\right)
\le j.\]	
This completes the proof.
\end{proof}

\section{The vertex-weighted Laplacian spectra of a simplicial complex and its subcomplexes}

For a simplicial complex and its subcomplex, Shukla and Yogeshwaran \cite{SY2020} established a relation between the smallest eigenvalues of their respective combinatorial Laplacians in each dimension. 
\begin{theorem}[{\cite[Theorem 1.4]{SY2020}}]\label{thm-02}
	For every simplicial complex $X$ and every subcomplex $X^{\prime}$ of $X$, and for $k \geq 1$,
	\[\lambda^{\uparrow}_1\left(L_k\left(X^{\prime}\right)\right) \geq \lambda^{\uparrow}_1\left(L_k\left(X\right)\right)-(k+2) \max _{\sigma \in X^{\prime}(k)}\left(\deg^{+}_X(\sigma)-\deg^{+}_{X^{\prime}}(\sigma)\right).\]
\end{theorem}

Recently, Zhan, Huang, and Zhou \cite{ZHZ2025} extended Theorem \ref{thm-02} to each eigenvalue of the combinatorial Laplacians.
\begin{theorem}[{\cite[Theorem 1.7]{ZHZ2025}}]\label{thm-1-7}
Let $X$ be a simplicial complex, and let $X^{\prime}$ be a subcomplex of $X$. Then, for $k \geq 0$ and $1 \leq i \leq f_k\left(X^{\prime}\right)$,
\[\lambda^{\uparrow}_i\left(L_k\left(X^{\prime}\right)\right) \geq \lambda^{\uparrow}_i\left(L_k(X)\right)-(k+2) \max _{\sigma \in X^{\prime}(k)}\left(\deg^{+}_X(\sigma)-\deg^{+}_{X^{\prime}}(\sigma)\right).\]
\end{theorem}

In this section, applying the method developed in \cite{ZHZ2025}, we derive a relationship between the $i$-th smallest eigenvalues of the vertex-weighted Laplacians of a simplicial complex and those of its subcomplexes, which generalizes Theorem \ref{thm-1-7} to all eigenvalues of the vertex-weighted Laplacians.

We shall now on use $X$ to denote a complex and $X^{\prime}$ to denote a subcomplex of $X$.  In order to measure the difference between the two complexes, we introduce a new notation. Let $\omega\colon X\to \mathbb{R}$ be a vertex weight function of simplicial complex $X$. For $1\le k \le \operatorname{dim}(X)$, define
\[S^{\omega}_k\left(X, X^{\prime}\right)\colon =\max_{\sigma \in X^{\prime}(k)}\sum_{v\in \operatorname{lk}_X(\sigma)\setminus \operatorname{lk}_{X^{\prime}}(\sigma)}\omega(v).\] 
In particular, when the weight function $\omega \equiv 1$, then the quantity $S^{\omega}_k\left(X, X^{\prime}\right)$ coincides with the parameter $S_k\left(X, X^{\prime}\right)$ originally introduced by Shukla and Yogeshwaran \cite{SY2020}.

\begin{theorem}\label{thm-2}
Let $X$ be a simplicial complex with vertex weight function $\omega\colon X \rightarrow \mathbb{R}_{>0}$. For every subcomplex $X^{\prime}$ of $X$, and for $k \geq 1$,
	\[\lambda^{\uparrow}_i\left(L^{\omega}_k\left(X^{\prime}\right)\right) \ge  \lambda^{\uparrow}_i\left(L^{\omega}_k(X)\right)-(k+2)S^{\omega}_k\left(X, X^{\prime}\right).\]
\end{theorem}
\begin{proof}
	Let $L^{\prime}$ be the principal submatrix of $L_k^{\omega}(X)$ with rows and colums indexed by $X^{\prime}(k)$. By Lemma \ref{lem-3}, we have \[\left(L^{\prime}-L_k^{\omega}(X^{\prime})\right)_{\sigma, \tau}= \begin{cases}\sum _{v \in \operatorname{lk}_X(\sigma)\setminus \operatorname{lk}_{X^{\prime}}(\sigma)} \omega(v),& \text { if } \sigma=\tau, \\[4mm] -(-1)^{\varepsilon(\sigma, \tau)} \omega(\tau \setminus \sigma), & \text { if }|\sigma \cap \tau|=k, \sigma \cup \tau \in X(k+1) \setminus X^{\prime}(k+1),\\[4mm] 0, & \text { otherwise},\end{cases}\]	for every $\sigma, \tau \in X(k)$. 
	By the Geršgorin circle theorem, we obtain
	\[\begin{aligned}
		\lambda^{\downarrow}_{1}\left(L^{\prime}-L^{\omega}_k\left(X^{\prime}\right)\right)& \le \max _{\sigma\in X^{\prime}(k)}\left\{\sum _{v \in \operatorname{lk}_X(\sigma)\setminus \operatorname{lk}_{X^{\prime}}(\sigma)} \omega(v)+\sum_{ \substack{\tau\in X(k)\colon |\sigma \cap \tau|=k,\\ \sigma \cup \tau \in X(k+1) \setminus X^{\prime}(k+1)}}\omega(\tau \setminus \sigma)\right\}\\&=\max _{\sigma\in X^{\prime}(k)}\left\{\sum _{v \in \operatorname{lk}_X(\sigma)\setminus \operatorname{lk}_{X^{\prime}}(\sigma)} \omega(v)+(k+1)\sum_{ v\in \operatorname{lk}_X(\sigma)\setminus \operatorname{lk}_{X^{\prime}}(\sigma)}\omega(v)\right\}\\&=(k+2)S^{\omega}_k\left(X, X^{\prime}\right),
	\end{aligned}\]
	and thus \[\lambda^{\uparrow}_{1}\left(L^{\omega}_k\left(X^{\prime}\right)-L^{\prime}\right)=-\lambda^{\downarrow}_{1}\left(L^{\prime}-L^{\omega}_k\left(X^{\prime}\right)\right)\ge -(k+2)S^{\omega}_k\left(X, X^{\prime}\right).\]
	By Lemma \ref{lem-12} and Corollary \ref{cor-1},
	\[\begin{aligned}
		\lambda_i^{\uparrow}\left(L^{\omega}_k\left(X^{\prime}\right)\right) & \geq \lambda_i^{\uparrow}\left(L^{\prime}\right)+\lambda^{\uparrow}_{1}\left(L^{\omega}_k\left(X^{\prime}\right)-L^{\prime}\right) \\
		& \geq \lambda^{\uparrow}_i\left(L^{\omega}_k(X)\right)-(k+2)S^{\omega}_k\left(X, X^{\prime}\right).
	\end{aligned}\]
	This completes the proof.
\end{proof}

\begin{remark}
When $\omega\equiv 1$, it is easy to see that
\[S^{\omega}_k\left(X, X^{\prime}\right)=S_k\left(X, X^{\prime}\right)=\max _{\sigma \in X^{\prime}(k)}\left(\deg^{+}_X(\sigma)-\deg^{+}_{X^{\prime}}(\sigma)\right).\] Therefore, Theorem \ref{thm-2} immediately reduces to Theorem \ref{thm-1-7}.
\end{remark}

As a consequence of Theorem \ref{thm-2}, we obtain a sufficient condition for the vanishing of the reduced cohomology groups of a subcomplex $X^{\prime}$.

\[\lambda^{\uparrow}_i\left(L^{\omega}_k(X)\right) \ge S^{\uparrow}_{k+1, i}\left(L^{\omega}\left(G_X\right)+J^{\omega}\right)-k\sum _{u\in V}\omega(u)-\max _{\sigma \in X(k)}\left\{\sum_{j=0}^{k+1}(j+1)\sum _{u\in \sigma[j]}\omega(u)\right\}.\]
\begin{corollary}\label{cor-4}
	Let $X$ be a simplicial complex with vertex weight function $\omega\colon X \rightarrow \mathbb{R}_{>0}$, and let $X^{\prime}$ be a subcomplex of $X$. If \[\begin{aligned}
	    \lambda_{1}^{\uparrow}\left(L^{\omega}(G_X)+J^{\omega}\right)&>\frac{k}{k+1} \sum_{v \in V} \omega(v)+\frac{1}{k+1}\max _{\sigma \in X(k)}\left\{\sum_{j=0}^{k+1}(j+1)\sum _{u\in \sigma[j]}\omega(u)\right\}\\&\quad+\frac{k+2}{k+1} S^{\omega}_k\left(X, X^{\prime}\right),
	\end{aligned}\] then $\widetilde{H}^k\left(X^{\prime}\right)=0$.
\end{corollary}
\begin{proof}
	By Theorem \ref{thm-2} and Theorem \ref{thm-11}, we have
	\[\begin{aligned}
		&\lambda_1^{\uparrow}\left(L_k^{\omega}\left(X^{\prime}\right)\right)\\&\geq \lambda_1^{\uparrow}(L_k^{\omega}(X))-(k+2) S^{\omega}_k\left(X, X^{\prime}\right)
		\\& \ge S_{k+1, 1}^{\uparrow}\left(L^{\omega}(G_X)+J^{\omega}\right)-k \sum_{v \in V} \omega(v)-\max _{\sigma \in X(k)}\left\{\sum_{j=0}^{k+1}(j+1)\sum _{u\in \sigma[j]}\omega(u)\right\}\\&\quad -(k+2) S^{\omega}_k\left(X, X^{\prime}\right)
		\\&\ge (k+1)\lambda_{1}^{\uparrow}\left(L^{\omega}(G_X)+J^{\omega}\right)
	-k \sum_{v \in V} \omega(v)-\max _{\sigma \in X(k)}\left\{\sum_{j=0}^{k+1}(j+1)\sum _{u\in \sigma[j]}\omega(u)\right\}\\&\quad-(k+2) S^{\omega}_k\left(X, X^{\prime}\right).
	\end{aligned}\]
	Therefore, if \[\begin{aligned}
	    \lambda_{1}^{\uparrow}\left(L^{\omega}(G_X)+J^{\omega}\right)&>\frac{k}{k+1} \sum_{v \in V} \omega(v)+\frac{1}{k+1}\max _{\sigma \in X(k)}\left\{\sum_{j=0}^{k+1}(j+1)\sum _{u\in \sigma[j]}\omega(u)\right\}\\&\quad+\frac{k+2}{k+1} S^{\omega}_k\left(X, X^{\prime}\right),
	\end{aligned}\] then $\widetilde{H}^k\left(X^{\prime};\mathbb{R}\right)=0$.
\end{proof}

\begin{remark}
If we restrict to $\omega \equiv 1$, then
\[S^{\omega}_k\left(X, X^{\prime}\right)=S_k\left(X, X^{\prime}\right)\]
and
\[\lambda_{1}^{\uparrow}\left(L^{\omega}(G_X)+J^{\omega}\right)=\lambda_{1}^{\uparrow}\left(L(G_X)+J\right)=\lambda_{2}^{\uparrow}\left(L(G_X)\right).\]
In this case, Corollary \ref{cor-4} generalizes the result established by Shukla and Yogeshwaran (see \cite[Corollary 1.5]{SY2020}) from clique complexes to general simplicial complexes, since
\[
\frac{1}{k+1}\max _{\sigma \in X(k)}\left\{\sum_{j=0}^{k+1}(j+1)\sum _{u\in \sigma[j]}\omega(u)\right\}= 0
\]
when $X$ is a clique complex.
\end{remark}

\section{Concluding remarks}

For any simplicial complex $X$ on the vertex set $V$ with a positive weight function $\omega$ and any $0 \le k \le \operatorname{dim}(X)$, Theorem \ref{thm-18} implies that the upper bound $\sum_{v \in V} \omega(v)$ is an eigenvalue of $L_k^{\omega}(X)$ whenever either $f_k(X) > \binom{n-1}{k+1}$ or $f_{k+1}(X) > \binom{n-1}{k+2}$ holds. However, this condition is not optimal. Below we present an example of a simplicial complex that does not satisfy it but nonetheless has $\sum_{v \in V} \omega(v)$ as its largest eigenvalue.
\begin{example}
	\begin{figure}[htbp]
		\centering
		\includegraphics[width=5cm]{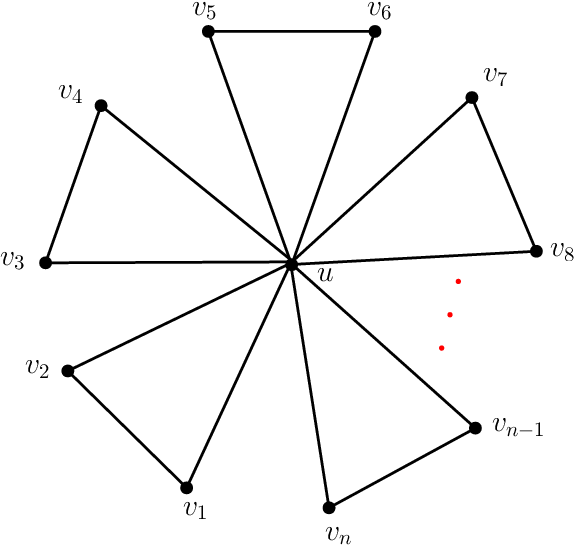}
		\caption{The friendship graph $F_n$}
		\label{fig-4}
	\end{figure}
	The friendship graph $F_n$ is obtained by gluing together $n$ triangles along a single common vertex (see Fig. \ref{fig-4}). Let $X$ denote the clique complex of $F_n$ ($n>2$). 
Note that $\operatorname{dim}(X)=2$. Moreover, 
\[
f_0(X)=2n+1>\binom{2n}{1}, \ \ 
f_1(X)=3n<\binom{2n}{2}, \ \ 
f_2(X)<\binom{2n}{3}.
\]
Therefore, $f_1(X)$ and $f_2(X)$ do not satisfy the sufficient condition in Theorem \ref{thm-18}. However,  a straightforward computation shows that $\sum_{v\in V(X)}\omega(v)$ is an eigenvalue 
of $L_{1}^{\omega}(X)$ with multiplicity $1$.
\end{example}
This naturally raises the following problem.
\begin{problem}
	Determine the necessary and sufficient conditions under which the upper bound in Theorem \ref{thm-18} is attained.
\end{problem}

Theorem \ref{thm-22} shows that the simplicial complex
\[
\left(\Delta_d^{(d-1)}\right)^{*(n-k-1)} * \Delta_{(d+1)(k+1)-dn-1},
\]
on vertex set $V$, equipped with a positive vertex weight function $\omega$ that is constant on the vertex set of each copy of $\Delta_d^{(d-1)}$, is the unique simplicial complex whose $k$-th vertex-weighted spectral gap attains the lower bound\begin{equation}\label{eq-nlk}
	(d+1)\min_{\sigma\in X(k)}\sum_{v\in \sigma} \omega(v)-d\sum_{v\in V}\omega(v),
\end{equation}
provided by Corollary \ref{cor-l-1} for some $-1 \le k \le \dim(X)$. 
 Here, uniqueness refers to the combinatorial structure (i.e., up to isomorphism); the vertex weight function $\omega$ itself is not unique and may vary while still attaining the bound. It is clear that the bound in \eqref{eq-nlk} is weaker than the lower bound
\begin{equation}\label{eq-lk}
	(d+1)\min_{\sigma\in X(k)}\sum_{v\in \sigma\cup \operatorname{lk}_X(\sigma)} \omega(v)-d\sum_{v\in V}\omega(v)
\end{equation}
established in Theorem \ref{thm-1}. Motivated by Theorem \ref{thm-22}, we seek to determine whether the simplicial complexes whose $k$-th vertex-weighted  spectral gap attains the lower bound in \eqref{eq-lk} for some $-1 \le k \le \operatorname{dim}(X)$ are unique, and to provide a complete characterization of these complexes.

If $k=\operatorname{dim}(X)$, then $\operatorname{lk}_X(\sigma)=\emptyset$ for any $\sigma\in X(k)$, and thus the bound in \eqref{eq-lk} reduces to the bound $\min_{\sigma\in X(k)}\sum_{v \in \sigma} \omega(v)-d\sum_{v\in V}\omega(v)$ given in \eqref{eq-nlk}. We immediately obtain the following corollary, which indicates the uniqueness of the simplicial complex attaining the lower bound in this case.
\begin{corollary}
Let $X$ be a simplicial complex on vertex set $V$ of size $n$ with $h(X)=d$, and endowed with a vertex weight function $\omega\colon  X \rightarrow \mathbb{R}_{>0}$. If
\[\lambda_1^{\uparrow}(L_k^{\omega}(X))=(d+1)\min _{\sigma\in X(k)}\sum_{v\in \sigma\cup\operatorname{lk}_X(\sigma)}\omega(v)-d\sum_{v\in V}\omega(v)\]
for $k=\operatorname{dim}(X)$,  then $X \cong\left(\Delta_d^{(d-1)}\right)^{*(n-k-1)} * \Delta_{(d+1)(k+1)-dn-1}$ and $\omega|_{V\left(\Delta_d^{(d-1)} \right)}$ is constant for each $\Delta_d^{(d-1)}$.
\end{corollary}
If $k=-1$, then $X(k)=\left\{\emptyset\right\}$ and $\operatorname{lk}_X(\emptyset)=V$. In this case, the bound in \eqref{eq-lk} reduces to $\sum_{v\in V}\omega(v)$. It is a basic fact that, for any simplicial complex $X$ on the vertex set $V$, endowed with a positive vertex weight function $\omega$, the spectrum is $\mathbf{s}^{\omega}_{-1}\left(X\right)=\left\{\sum_{v\in V}\omega(v)\right\}$, and hence the spectral gap automatically attains this bound. Consequently, the simplicial complexes achieving this lower bound are far from unique.
For $0 \le k < \operatorname{dim}(X)$, however, a complete characterization of all simplicial complexes whose $k$-th vertex-weighted spectral gap attains the lower bound in \eqref{eq-lk} remains a challenging open problem. 
\begin{problem}
	Characterize the simplicial complexes whose $k$-th spectral gap attains the lower bound in \eqref{eq-lk} for some $0 \le k < \operatorname{dim}(X)$.
\end{problem}

Recall that, for any simplicial complex $X$ on the vertex set $V$ with a positive weight function $\omega$ and any $0 \le k \le \operatorname{dim}(X)$, Theorem \ref{thm-11} provides the following lower bound for the $i$-th smallest eigenvalue of $L^{\omega}_k(X)$:
\begin{equation}\label{eq-5l}
	\lambda^{\uparrow}_i\bigl(L^{\omega}_k(X)\bigr) \ge 
	S^{\uparrow}_{k+1, i}\Bigl(L^{\omega}(G_X)+J^{\omega}\Bigr)
	- k \sum_{u \in V} \omega(u)
	- \max_{\sigma \in X(k)} \Biggl\{ \sum_{j=0}^{k+1} (j+1) \sum_{u \in \sigma[j]} \omega(u) \Biggr\}.
\end{equation}
As illustrated by the following example, this lower bound is sharp.
\begin{example}
For the complete simplicial complex $\Delta_{n-1}$ on $n$ vertices, let $V$ and $\omega$ denote its vertex set and a positive weight function, respectively. From Corollary \ref{cor-2}, we have
\[\mathbf{s}_{k }^{\omega}\left(\Delta_{n-1}\right)=\left\{\left(\sum_{v\in V}\omega(v)\right)^{\left(\binom{n}{k+1}\right)} \right\},\] which implies that $\lambda^{\uparrow}_i\bigl(L^{\omega}_k(\Delta_{n-1})\bigr)=\sum_{v\in V}\omega(v)$ for any $1\le i\le f_k(\Delta_{n-1})$. Moreover, the underlying graph of $\Delta_{n-1}$ satisfies $G_{\Delta_{n-1}} \cong K_n$. Hence, $L^{\omega}(G_{\Delta_{n-1}}) + J^{\omega}$ is a diagonal matrix with all diagonal entries equal to $\sum_{v \in V} \omega(v)$, and thus $S^{\uparrow}_{k+1, i}\Bigl(L^{\omega}(G_X)+J^{\omega}\Bigr)=(k+1)\sum_{v \in V} \omega(v)$. Finally, since $\Delta_{n-1}$ itself is a clique complex, the last term in \eqref{eq-5l} vanishes.Therefore, the lower bound in Theorem \ref{thm-11} is attained for $\Delta_{n-1}$.
\end{example}
So, we end up this paper by the following problem
\begin{problem}
	Characterize the simplicial complexes $X$ that attain the equality of \eqref{eq-5l} for $-1\le k\le \operatorname{dim}(X)$ and $1\le i\le f_k(X)$.
\end{problem}

\section*{Acknowledgement}

This work is supported by National Natural Science Foundation of China (No. 12371362).

\clearpage
\addcontentsline{toc}{section}{References}  
\bibliographystyle{plain}
\end{document}